\documentclass[10pt,twoside]{article}

\usepackage{geometry}
\usepackage{color}

\usepackage[numbers,sort&compress]{natbib}
\usepackage{graphicx,latexsym,euscript,makeidx,color,bm}
\usepackage{amsmath,amsfonts,amssymb,amsthm,thmtools,mathtools,mathrsfs,enumerate}
\usepackage[colorlinks,linkcolor=blue,anchorcolor=green,citecolor=red]{hyperref}

\usepackage[utf8]{inputenc}
\usepackage[T1]{fontenc}

\usepackage{tocloft}  % 加载目录排版宏包

\def\5n{\negthinspace \negthinspace \negthinspace \negthinspace \negthinspace }
\def\4n{\negthinspace \negthinspace \negthinspace \negthinspace }
\def\3n{\negthinspace \negthinspace \negthinspace }
\def\2n{\negthinspace \negthinspace }
\def\1n{\negthinspace }
\def\dbE{\mathbb{E}}     
\def\dbF{\mathbb{F}} \def\sF{\mathscr{F}}    
\def\dbG{\mathbb{G}}         
\def\dbH{\mathbb{H}}

\def\dbP{\mathbb{P}}     
     
\def\dbR{\mathbb{R}}

\def\ms{\medskip}

\def\no{\noindent}        \def\q{\quad}                      
    \def\qq{\qquad}                    
    \def\hb{\hbox}                     
                   
         \def\rf{\eqref}                    
  \def\deq{\triangleq}               
           
\def\les{\leqslant}     \def\le{\leqslant}                
\def\leq{\leqslant}       \def\geq{\geqslant}
\def\ges{\geqslant}       \def\esssup{\mathop{\rm esssup}}   \def\[{\Big[}
           \def\]{\Big]}

\def\nn{\nonumber}                              
\def\a{\alpha}

\def\e{\varepsilon}             
           \def\i{\infty}   

\theoremstyle{plain}

\newtheorem{exam}{\bf Example}[section]

\makeatletter

\@addtoreset{equation}{section}
\makeatother
\newtheorem{theorem}{Theorem}[section]
\newtheorem{definition}[theorem]{Definition}
\newtheorem{proposition}[theorem]{Proposition}
\newtheorem{corollary}[theorem]{Corollary}
\newtheorem{lemma}[theorem]{Lemma}
\newtheorem{remark}[theorem]{Remark}

\makeatletter

\@addtoreset{equation}{section}
\makeatother

\allowdisplaybreaks[4]
\begin{document}
	\title{
		\Large \bf  Backward doubly stochastic differential equations with  and without reflection under weak conditions\thanks{
	%					 This paper is supported by National Key R\&D Program of China 2022YFA1006102. And 
			YH is partially supported by Lebesgue Center of Mathematics ``Investissements d'avenir'' program-ANR-11-LABX-0020-01, by CAESARS-ANR-15-CE05-0024 and by MFG-ANR16-CE40-0015-01.
			JW is supported  by National Natural Science Foundation of China (Grant No. 12571478) and Guangdong Basic and Applied Basic Research Foundation (Grant No. 2025B151502009),
			and  Shenzhen Science and Technology Program (Grant
			No. JCYJ20230807093309021).
		}
	}
	
	\author{
		Shuxian Gao\thanks{
			Department of Mathematics,
			Southern University of Science and Technology, Shenzhen, 518055, China
			(Email: {\tt 12431007@mail.sustech.edu.cn}).}~,~~~~
		Ying Hu\thanks{Univ. Rennes, CNRS, IRMAR - UMR 6625, F-35000 Rennes, France
			(Email: {\tt ying.hu@univ-rennes.fr}).}~,~~~~
		%
		%		%
		%		Shige Peng\thanks{School of Mathematics, Frontiers Science Center for Nonlinear Expectations (Ministry of Education), Shandong University, Jinan, 250100, China
			%	(Email: {\tt peng@sdu.edu.cn}).}~,~~~~
		%%
		%
		Jiaqiang Wen\thanks{Department of Mathematics and SUSTech International Center for Mathematics,
			Southern University of Science and Technology, Shenzhen, 518055, China
			(Email: {\tt wenjq@sustech.edu.cn}).}
	}
	
	%	\author{Ying Hu, Feng Li, Jiaqiang Wen}
	\date{}
	\maketitle
	
	\no\bf Abstract. \rm
	In this paper, we study the  solvability  of backward doubly stochastic differential equations (BDSDEs, for short),  both with and without reflection, under weak conditions on the generator.
First, when the generator $f$ is of general growth in $y$ and linear growth in $z$, we establish the existence, uniqueness, comparison principle, and the existence of maximal solutions.
Second,  when   \(f\) is of linear growth in  \(y\) and quadratic growth in \(z\) with bounded terminal value, we prove the existence, uniqueness, and comparison principle.
Finally, when   $f$ is of general growth in $y$ and quadratic growth in $z$ with bounded terminal value, we prove the existence of maximal solutions.

	\ms
	
	\no\bf Keywords:  
	\rm  
	Backward doubly stochastic differential equations; reflection; general growth; quadratic growth
	
	\ms
	
	\no\bf AMS subject classifications: \rm  60H10;  60H20
	
	{
		\centering
		\tableofcontents
	}

	\section{Introduction}\label{section1}
	Throughout this paper, let $(\Omega, \mathscr{F}, \mathbb{P})$ be a complete probability space on which two standard independent Brownian motions $\{W_t ; 0 \leqslant t<\i\}$ and $\{B_t ; 0 \leqslant t<\i\}$, taking values in $\mathbb{R}^{d}$ and $\mathbb{R}^{l}$, respectively, are defined. Let $T>0$ be a fixed terminal time, and denote by $\mathscr{N}$ the class of $\dbP$-null sets of $\sF$, where
	\begin{equation*}
		\sF_{t} \deq   \sF_{t}^{W} \vee \sF_{t,T}^{B},\q\
		\mathscr{G}_t \deq \mathscr{F}^{W}_{t} \vee \mathscr{F}^{B}_{T} ,\q\ 
		\forall t\in[0,T].
	\end{equation*}
	Here, for any process $\{\eta_t;0\les t\les T\},$ we set $\sF_{s,t}^{\eta}=\sigma\{\eta_r-\eta_s;s\les r\les t\}\vee \mathscr{N}$ and $\sF_{t}^{\eta}=\sF_{0,t}^{\eta}$.
	Note that  the collection $\dbG\deq\{\mathscr{G}_t; t\in[0,T]\}$ is a filtration, whereas  $\dbF\deq\{\mathscr{F}_t; t\in[0,T]\}$  is not, since it is neither increasing nor decreasing. We consider the following reflected backward doubly stochastic differential equations (RBDSDEs, for short):
	\begin{equation}\label{r1}
		\left\{
		\begin{aligned}
			&	Y_t=\xi +\int_{t}^{T}f(s,Y_s,Z_s)ds+\int_{t}^{T}g(s,Y_s,Z_s)dB_s+K_T-K_t-\int_t^TZ_sdW_s,\q\ 0\leq t\leq T,\\
			&	Y_t\geq S_t\q\ \hbox{and} \q\ \int_0^T (Y_s-S_s)dK_s=0.
		\end{aligned}\right.
	\end{equation}
	In the above equations, \(K\) is a continuous non-decreasing  process starting from zero, the \(dB\)-integral is a backward It\^o integral,  and the \(dW\)-integral is a forward It\^o integral. The terminal value \(\xi\) is \(\mathscr{F}_T\)-measurable,  and \(f : \Omega \times [0,T] \times \mathbb{R}\times \mathbb{R}^d \longmapsto \mathbb{R}\) and \(g : \Omega \times [0,T] \times \mathbb{R}\times \mathbb{R}^d\longmapsto \mathbb{R}^l \) are jointly measurable.  
	The solution process $Y$ is constrained to stay above a given obstacle process $S$, which is  referred to as the obstacle.
	The	process $K$ acts as the minimal force that keeps \(Y\)   above the obstacle \(S\),  as characterized by the Skorokhod condition.

	By setting \(K\equiv 0\) and removing the obstacle constraint and the Skorokhod condition, we obtain
 the following non-reflected backward doubly stochastic differential equations (BDSDEs, for short):
	\begin{align}\label{1}
		Y_t=\xi+\int_t^T f(s,Y_s,Z_s)ds+\int_{t}^{T}g(s,Y_s,Z_s)dB_s-\int_t^T Z_sdW_s,\q\ 0\le t\le T,
	\end{align}
	which  were first introduced by Pardoux and Peng \cite{94}. Under  uniform Lipschitz assumptions on the coefficients $f$ and $g$, they established  that,  for any square-integrable terminal value 
	$\xi$,  BDSDE \eqref{1} admits a unique solution pair $(Y,Z)$.
	Since then, BDSDEs have found significant applications in areas such as stochastic control and stochastic partial differential equations (see \cite{10m,hmz,wy}). 
	 In addition, the existence and uniqueness of BDSDEs under weaker conditions on the generator \(f\) have been studied extensively. In particular, an important contribution to the theory of BDSDEs with quadratic growth was made by Hu, Wen, and Xiong \cite{wen1}. Other recent works on BDSDEs under weak conditions can be found in \cite{11,man16,z10}, among others.
	Furthermore, when \(g=0\) in RBDSDEs \eqref{r1}, the equations reduce to reflected backward stochastic differential equations (RBSDEs, for short), which were first introduced by El Karoui et al. \cite{97r} and have important applications in finance.
	Specifically,    \(S\) represents the payoff process,  while the increasing process  \(K\) can be interpreted as the cumulative effort required to keep the solution above the barrier, thereby providing a dynamic representation of the cost of enforcing the constraint. Some recent works on RBSDEs can be found, for instance, in \cite{ho16,q18,08x}.

	On the other hand, RBDSDEs \eqref{r1} originate from the classical nonlinear backward stochastic differential equations (BSDEs, for short), introduced by Pardoux and Peng \cite{90} under Lipschitz conditions.
	Since then, BSDEs have been widely applied in stochastic control, financial mathematics, and partial differential equations (see \cite{el_karoui_backward_1997,yong_stochastic_1999,zhang_backward_2017}).
	Moreover, to broaden their applicability, many researchers have studied BSDEs under non-Lipschitz conditions, which is in fact an open problem posed by Peng \cite{peng_open_1999}.
	For instance, concerning the non-Lipschitz condition on the generator \(f\) with respect to \(y\), an early result was obtained by Mao \cite{95}. Lepeltier and San Mart\'in \cite{97} proved the existence of solutions to BSDEs when the generator \(f\) is continuous in $y$.  Pardoux \cite{99} studied the case where \(f\) is of general growth and  monotone in \(y\). 
	Furthermore,  concerning the non-Lipschitz condition on the generator $f$ with respect to  \(z\), Kobylanski \cite{00} first established the existence and uniqueness of  BSDEs whose generator \(f\) is of  quadratic growth in $z$. Lepeltier and San Mart\'in \cite{98} considered the case where $f$ is of superlinear growth in $y$ and quadratic growth in $z$. Briand, Lepeltier, and San Mart\'in \cite{07o} established a general existence result when $f$ is of general growth and monotone in \(y\), and is of quadratic growth in \(z\).
	Moreover, for quadratic BSDEs, Briand and Hu \cite{06h,08h} studied the one-dimensional case with unbounded terminal value, while Hu and Tang \cite{ht16} and Fan, Hu, and Tang \cite{f23} studied the multi-dimensional case. Other recent results on quadratic BSDEs can be found in \cite{wen1,hao_mean-field_2025,delbaen_backward_2011}, among others.

	In this paper, we  systematically study the well-posedness of RBDSDE \eqref{r1}.
	Beyond its purely mathematical interest---particularly in connection with one of the open problems proposed by Peng \cite{peng_open_1999}---the study of RBDSDEs is also motivated by several promising potential applications. For instance,  this type of equation has wide-ranging applications in stochastic control problems, stochastic partial differential equations, and mathematical finance
	(see Hu, Liang, and Tang \cite{HuLiangTang2020,HuLiangTang2024}). 
	The first work in this direction was carried out by Bahlali et al. \cite{09}, followed by further studies in \cite{liluo13, ren10,be20}, among others. These works studied RBDSDEs under uniform Lipschitz continuity conditions.
	To the best of our knowledge, 
the combination of the one-barrier reflected setting and the general growth assumptions imposed in this paper has not been addressed in the existing literature cited above. 
	This gap restricts the further development and application of RBDSDEs. Therefore, the aim of this paper is to fill this gap by studying the well-posedness of RBDSDE \eqref{r1} under weaker conditions.	
	It should be pointed out that there are two main difficulties in this direction:
	\begin{itemize}
		\item [(i)]   Standard techniques for addressing the general growth conditions of \(f\) in \(y\) rely on conditional expectations (see Pardoux \cite{99}). However, these methods cannot be applied to  RBDSDE \eqref{r1}, because the presence of the backward It\^o's  integral implies that   \(\mathbb{F} = \{\mathscr{F}_t; t \in [0,T]\}\) is not a filtration.
		\item [(ii)]  The BMO method, a standard tool for dealing with BSDEs with quadratic growth, is not applicable to  RBDSDE \rf{r1}. This limitation arises from two factors:  first,  the solution process does not admit the usual martingale structure required by the classical BMO framework; 
		 second, a counterexample constructed by Hu et al. \cite{wen1} shows that a bounded terminal value does not guarantee a bounded solution when \(f\) is of quadratic growth in \(z\).
	\end{itemize}
	To establish the well-posedness of RBDSDE \eqref{r1} under weak conditions, we introduce several new techniques to overcome the above difficulties. The main contributions and innovations of this paper are summarized as follows: 
	\begin{itemize}
		\item [(i)]  
		When the generator $f$ is of general growth and  monotone in \(y\)  (see condition {\bf (F1)} below) and Lipschitz in \(z\),  we establish the existence, uniqueness, and comparison principle for solutions of RBDSDEs \rf{r1} (see \autoref{t3.8}).
		To the best of our knowledge, the corresponding result  under
		the present assumptions
		appears to be new even in the non-reflected case (see \autoref{t3.2}).
		The main idea of the proof is as follows. To derive an estimate for the solution \(Y\), we note that the conditional expectation cannot be applied here. Therefore, we first consider  the case where the generator $f$ does not depend on $z$ and the coefficient $g$ does not depend on $(y,z)$.
		In this simple case, we start with the auxiliary zero drift linear BDSDE \eqref{Y0}, apply
		  a truncation argument to its solution, and then construct a sequence of solutions to classical BSDEs.
		   Next, for the general case, we use a fixed point argument on \(L^2_{\mathbb{F}}([0,T];\mathbb{R})\times
		   L^2_{\mathbb{F}}([0,T];\mathbb{R}^d)\) equipped with the equivalent weighted norm introduced in \autoref{t3.2}.
		The corresponding mapping is a contraction with contraction factor \(\frac{1+\alpha}{2}<1\).
		\item [(ii)]   	When the generator $f$ is of general growth and monotone in \(y\)  (see condition {\bf (F1)} below) but is not Lipschitz in \(z\), we obtain the existence of maximal solutions for RBDSDEs \rf{r1}, both with and without reflection (see \autoref{t3.3} and \autoref{t3.9}, respectively). 
		We note that the classical approach for proving the existence of maximal solutions, as in  Lepeltier and San Mart\'in \cite{97},  relies on constructing a ``nice'' approximation of the generator $f$.  
		  This approach works well when $f$ is of linear growth in \(y\), but becomes difficult to apply when $f$ is of general growth in \(y\). This is because the standard inf-convolution approximation does not preserve the required growth structure. To overcome this difficulty, we use a truncation technique and derive a priori estimates for the solution $Y$. 
		\item [(iii)]  
		When the generator $f$ is of linear growth in $y$ and quadratic growth in $z$ (see condition \(\textbf{(F2)}\) below), we establish the existence, uniqueness, and comparison principle for solutions of RBDSDEs \rf{r1}, both with and without reflection (see \autoref{tb4.2} and \autoref{t5.4}, respectively). To prove existence and uniqueness, we impose condition \textbf{(G2)} (see below) on the coefficient \(g\). This condition is weaker than the corresponding condition introduced by Hu et al. \cite{wen1}, in the sense that it allows \(g\) to be non-homogeneous with respect to \(z\). 
		 Moreover, with the help of bounded a priori estimates for the solution $Y$ (see \autoref{l4.1} and \autoref{l4.9}), we establish existence by applying the monotone stability theorem together with a double approximation technique. The uniqueness result is then proved via a comparison principle. More precisely, we first prove uniqueness under a structural condition \(\textbf{(STR)}\) (see \autoref{t5.2}), and then complete the proof of uniqueness by a change-of-variables argument. 
		\item [(iv)]  
		When the generator $f$ is of general growth in $y$ and quadratic growth in $z$ (see conditions \(\textbf{(F3)}\) and \(\textbf{(F3')}\) below), we obtain the existence of maximal solutions to RBDSDEs \rf{r1}, both with and without reflection (see \autoref{t4.3} and \autoref{t4.12}, respectively). %
		For clarity of presentation, we first consider the case where $f$ is of linear growth in $y$. By applying a generalized comparison theorem, we estimate the bounds of the solution $Y$ through a pair of ordinary differential equations. We then use a truncation technique to establish the existence of maximal solutions to RBDSDEs \rf{r1}, both with and without reflection. %
		 Finally, for the general case where $f$ is of general growth in $y$, we apply a similar method to that used in the linear growth case to prove the desired results. 
	\end{itemize}

	The rest of this paper is organized as follows. In \autoref{section2}, we introduce some preliminary notions and assumptions.
	\autoref{section3} is devoted to the study of RBDSDEs when the generator \(f\) is of general growth and monotone in \(y\), and is of linear growth in $z$.
	In \autoref{section4} and \autoref{sectionnew5}, we consider two more general situations: first, when \(f\) is of linear growth in \(y\) and quadratic growth in \(z\); and second, when $f$ is of general growth in $y$ and quadratic growth in $z$.
	Finally, \autoref{conclusion} concludes the results.

	\section{Preliminaries}\label{section2}
	For any real $p\ges1$,  \(t\in[0,T]\), and Euclidean space $\dbH$,  we introduce the following spaces:
	\begin{align*}
		L_{\mathscr{F}_T}^{p}(\Omega;\mathbb{H}) &= \Big\{ \xi : \Omega\mapsto \mathbb{H}\ \Big| \ \xi \text{ is $\mathscr{F}_T$-measurable  and }  \|\xi\|_{L^p} \triangleq \left(\mathbb{E}\left[|\xi|^p\right]\right)^{\frac{1}{p}} < \infty \Big\}, \\
		L_{\mathscr{F}_T}^{\infty}(\Omega;\mathbb{H}) &= \Big\{ \xi : \Omega \mapsto \mathbb{H}\ \Big| \ \xi \text{ is $\mathscr{F}_T$-measurable   and }  \|\xi\|_{\infty} \triangleq  \esssup_{\omega \in \Omega} |\xi(\omega)| < \infty \Big\},  \\
		L_{\mathbb{F}}^{p}([t,T];\mathbb{H}) &= \Big\{ \varphi : [t,T] \times \Omega \mapsto \mathbb{H} \ \Big| \  \varphi  \text{ is a jointly measurable process, and}\text{ for a.e. }
		s\in [t,T], \\
		&\qq	\varphi_s \text{ is } \mathscr{F}_s \text{-measurable, and }  
		\|\varphi\|_{L_{\mathbb{F}}^{p}([t,T];\mathbb{H})} \triangleq \left[\mathbb{E} \left(\int_{t}^{T} |\varphi_{s}|^{2} ds\right)^{\frac{p}{2}}\right]^{\frac{1}{p}} < \infty \Big\},  \\
		S_{\mathbb{F}}^{p}([t,T];\mathbb{H}) &= \Big\{ \varphi : [t,T] \times \Omega \mapsto \mathbb{H} \ \Big| \ \varphi_s \text{ is } \mathscr{F}_s \text{-measurable, continuous, and } \\
		&\qquad\qquad\qquad\qquad\quad  \|\varphi\|_{S_{\mathbb{F}}^{p}([t,T];\mathbb{H})} \triangleq \left( \mathbb{E} \Big[ \sup_{s \in [t,T]} |\varphi_{s}|^{p} \Big] \right)^{\frac{1}{p}} < \infty \Big\}, \\
		S_{\mathbb{F}}^{\infty}([t,T];\mathbb{H}) &= \Big\{ \varphi : [t,T] \times \Omega \mapsto \mathbb{H} \ \Big| \ \varphi_s \text{ is } \mathscr{F}_s \text{-measurable, continuous, and }  \\
		&\qquad\qquad\qquad\qquad\quad	 \|\varphi\|_{S_{\mathbb{F}}^{\infty}([t,T];\mathbb{H})} \triangleq \esssup_{(s,\omega) \in [t,T] \times \Omega} |\varphi_{s}(\omega)| < \infty \Big\}, \\
		A_{\mathbb{F}}^{p}([t,T];\mathbb{R}_+) &= \Big\{ K : [t,T] \times \Omega \mapsto \mathbb{R}_+\  \Big| \ K\   \text{is continuous and non-decreasing on}\ [t,T],\ 
		\text{starts from} \\
	& \q \quad\text{zero at the left endpoint, i.e.}\ K_t=0,	\ \text{and}\ K_s  \text{ is } \mathscr{F}_s \text{-measurable,}
  \  \mathbb{E}\left[(K_T)^p\right] < \infty \Big\}.
	\end{align*}

	\begin{definition}\sl
		A triple of    processes $(Y,Z,K)\in S^2_{\mathbb{F}}([0,T];\mathbb{R})\times L^2_{\mathbb{F}}([0,T];\mathbb{R}^d)\times A^2_{\mathbb{F}}([0,T];\mathbb{R}_+)$  is called a solution of   RBDSDE  \rf{r1}  with obstacle $S$  if it satisfies \rf{r1},
		 $\dbP$-a.s.
	\end{definition}

Next, we introduce several assumptions that will be used in the following sections. These include the general growth of the generator \(f\) with respect to \(y\), the quadratic growth of \(f\) with respect to \(z\), and the boundedness of the terminal value \(\xi\), among others. We collect these conditions here in order to present their similarities and differences as clearly and comprehensively as possible.
	In the following, we always  let   \(C>0\), % \(\l>0\),
	\(\a\in(0,1)\),  and $\mu\in\dbR$ be some constants,  let $\lambda: [0,T]\longmapsto \mathbb{R}_+$ satisfy    \(\int_0^T |\lambda(t)|^2dt <\infty\),  and  let $\varphi:\mathbb{R}_+ \longmapsto \mathbb{R}_+$ be a non-decreasing continuous function.
	$$
	\textbf{(A1)}\q \xi\in L_{\mathscr{F}_T}^{2}(\Omega;\mathbb{R}); 
	\qq \textbf{(A2)}\q \xi\in  L_{\mathscr{F}_T}^{\infty}(\Omega;\mathbb{R}). \qq\qq\qq\qq \qq\qq \q \qq\	$$ 
	Assume that the coefficient $f(t,y,z):[0,T]\times\Omega\times\mathbb{R}\times\mathbb{R}^d\longmapsto \mathbb{R}$  is jointly measurable and continuous with respect to $(y,z)$,  and that,    for any
	$(t,\omega,y,y',z)\in [0,T]\times\Omega\times\mathbb{R}\times\mathbb{R}\times\mathbb{R}^d$, $f(\cdot,y,z)\in  L^2_{\mathbb{F}}([0,T];\mathbb{R})$,
	\begin{align*}
		\textbf{(F1)}\q&  |f(t,y,z)|\leq |f(t,0,0)|+\varphi(|y|)+ C |z| \hb{ with }  \varphi(0)=0, \hbox{ and } \label{F111}\\ 
		&\qq\qq\q	\left(y-y'\right) \big(f(t,y,z)-f(t,y',z)\big)\leq \mu \big|y-y'\big|^2. \nonumber \\
		\textbf{(F2)}\q&  |f(t,y,z)|\leq C(1+|y|+|z|^2).  \\ 
		\textbf{(F3)}\q& |f(t,y,z)|\leq \varphi(|y|)+ C |z|^2\q\hbox{and}\q
		\left(y-y'\right) \big(f(t,y,z)-f(t,y',z)\big)\leq \mu \big|y-y'\big|^2.
	\end{align*}
	Assume that the coefficient $g:[0,T]\times\Omega\times\mathbb{R}\times\mathbb{R}^d\longmapsto \mathbb{R}^l$  is jointly measurable, and   that,  for any
	$(t,\omega,y,y',z,z')\in [0,T]\times\Omega\times\mathbb{R}\times\mathbb{R}\times\mathbb{R}^d\times\mathbb{R}^d$, $g(\cdot,y,z)\in  L^2_{\mathbb{F}}([0,T];\mathbb{R}^l)$,
	\begin{align*}
		\textbf{(G1)}\q& |g(t,y,z)-g(t,y',z')|^2\leq C|y-y'|^2+\alpha |z-z'|^2.\qq\qq\qq\qq\qq\qq\qq
		\\
		\textbf{(G2)}\q&  		|g(t,y,z)|^2\leq |\lambda (t)|^2+ \alpha|z|^2.
		\\
		\textbf{(G3)}\q&   \Bigg\{ 	\varphi\Bigg(\Bigg|\mathbb{E}\left[\int_t^T g(s,0,0)dB_s ~\Big | ~ \mathscr{G}_t \right]\Bigg| \Bigg)\Bigg\}_{t\in [0,T]}
		\in L^2_{\mathbb{F}}([0,T];\mathbb{R}).
	\end{align*}
We point out that the	condition \(\textbf{(G3)}\) is imposed to handle the nonlinear growth in \(y\).  This condition is automatically satisfied under several commonly used exponential or polynomial integrability assumptions, as illustrated below.
	We  provide several examples in which \(g\) satisfies  condition \(\textbf{(G3)}\):
	\begin{exam}\label{eg}  \sl 
		For any \( p \geq 1 \) and \( k \in \mathbb{R}_+ \),  the following cases satisfy condition \hb{\textbf{(G3)}}:
		\begin{align*}
			(1)& \q\ 	\varphi(x)=|x|^p \q\  \hbox{and} \q\ g(\cdot,0,0)\in L^{2p}_{\mathbb{F}}([0,T];\mathbb{R}^l);\\
			(2)& \q\  \varphi(x)=e^{kx} \q\  \hbox{and} \q\ 
			\mathbb{E}\left[ 	e^{2k^2\int_0^T|g(t,0,0)|^2dt}\right]<\infty;\\
			(3)& \q\ \varphi(x)=e^{kx^2}\q\  \hbox{and} \q\ 
			\mathbb{E}\left[\frac{1}{\sqrt{1-4k\int_0^T |g(t,0,0)|^2dt}}  \right]<\infty,
		\end{align*}
	where (2) and (3) hold provided that \(g(\cdot,0,0)\) is deterministic, or more generally independent of the Brownian motion \(B\).
	To  derive (2), one uses estimates for exponential martingales together with Jensen's inequality.
	For (3),  one needs \(4k\int_0^T |g(t,0,0)|^2dt <1,\) a.s..  Then the conclusion follows from
	 the fact that the stochastic integral is conditionally Gaussian with respect to \(\hat{g}\triangleq \sigma\{g(t,0,0); 0\leq  t\leq T\}\).
	Indeed,
	let 
	\begin{align*}
	Y_t\triangleq \mathbb{E}\left[\int_t^T g(s,0,0)dB_s ~\Big | ~ \mathscr{G}_t \right],\q\ 0\leq t\leq T.
	\end{align*}
 Suppose that \(g(\cdot,0,0)\) is independent of \(B\). Conditionally on \(\hat{g}\), the random variable \(X_t\triangleq \int^T_t g(s,0,0)dB_s\) is centered Gaussian with conditional variance
 \begin{align*}
 	G_t=\int_t^T |g(s,0,0)|^2ds.
 \end{align*}
 Since \(\varphi\) is non-decreasing, conditional Jensen's inequality gives
 \begin{align*}
 	\mathbb{E}\left[\varphi(|Y_t|)^2\right]
 	\leq \mathbb{E}\left[\varphi\left( \Bigg|\int_t^T g(s,0,0)dB_s \Bigg|  \right)^2\right].
 \end{align*}
 Moreover, the following
conditional moment identity holds:
\begin{align*}
&\mathbb{E}\left[e^{2k|X_t|^2}\ |\  \hat{g} \right]=
\frac{1}{\sqrt{1-4kG_t}}, \q\ 4kG_t<1.
\end{align*}
This proves assertion (3).
	\end{exam}
	For the obstacle $S$,  we always assume that $S_T\leq \xi$, $ \dbP$-a.s.. In addition, we give the following assumptions:
	\begin{align*}
		\textbf{(S1)}\q& S^+\in S^2_{\dbF}([0,T];\dbR) \q \hbox{and} \q
		\mathbb{E}\left[\varphi^2(\sup\limits_{0\leq t\leq T}e^{\mu t}S_t^+)\right]<\infty. \qq\qq\qq\qq\qq\qq
		\\
		\textbf{(S2)}\q& 	S\in S_{\mathbb{F}}^{\infty}([0,T];\mathbb{R}).
	\end{align*}
	
	\section{\texorpdfstring{ 	General growth in  $y$ and linear growth in  $z$}{	General growth in  y and linear growth in  z}}\label{section3}

	In this section, we study the existence, uniqueness, and maximal solutions for BDSDEs and reflected BDSDEs when the generator \(f\) is of general growth in \(y\) and linear growth in \(z\). To better highlight the differences among these results, we divide this section into four subsections.
	Moreover, since a change of variables will be used frequently in the proofs, we first present the following remark for simplicity of exposition.

	\begin{remark}\label{r3.1}\sl 
		The triple $(Y,Z,K)$ is a solution of  reflected BDSDE \eqref{r1} if and only if  the following triple
		\begin{equation}\label{new solution}
			(\bar{Y},\bar{Z},\bar{K})
			\triangleq \left(e^{\mu t}Y_t,e^{\mu t}Z_t,\int_0^te^{\mu s}dK_s\right)
		\end{equation}
		is a solution of reflected BDSDE \eqref{r1}  with $({\xi},{f},{g},S)$  replaced by $(\bar{\xi} , \bar{f} , \bar{g}, \bar{S} )$, where
		for any \((t,y,z)\in [0,T]\times \mathbb{R}\times \mathbb{R}^d\),
		\begin{align*}
			\big(\bar{\xi},\bar{f}(t,y,z),\bar{g}(t,y,z),\bar{S}_t\big)
			=\big(e^{\mu T}\xi,\ e^{\mu t}f(t,e^{-\mu t}y,e^{-\mu t}z)-\mu y, \ e^{\mu t}g(t,e^{-\mu t}y,e^{-\mu t}z),\ e^{\mu t}S_t\big).
		\end{align*}
		Moreover,  the transformed coefficient 
		$\bar{f}$ satisfies the same monotonicity condition as in \(\textbf{(F1)}\) 
		with $\mu =0$, i.e., 
		\begin{align*}
			\left(y-y'\right)	\big(\bar{f}(t,y,z)-\bar{f}(t,y',z)\big)\leq 0, \q\
			\forall (t,\omega,y,y',z)\in [0,T]\times\Omega\times\mathbb{R}\times\mathbb{R}\times\mathbb{R}^d, 
		\end{align*}
		which implies that $\bar{f}$ is monotone non-increasing in $y$.
		Meanwhile, the transformed obstacle  $\bar{S}$ satisfies
		\begin{align*}
			\mathbb{E}\left[\sup\limits_{0\leq t\leq T}(\bar{S}_t^+)^2\right]<\infty\q \hbox{and}\q 
			\mathbb{E}\left[\varphi^2\big(\sup\limits_{0\leq t\leq T}(\bar{S}_t^+)\big)\right]<\infty.
		\end{align*}
	\end{remark}

	In view of \autoref{r3.1}, after applying the above change of variables, we shall work under assumptions \(\textbf{(F1)}\) and \(\textbf{(S1)}\) throughout the proofs in \autoref{section3}.

	\subsection{BDSDE: existence and uniqueness}\label{sec3.1}

	In this part, we study the existence and uniqueness of BDSDEs when the generator \(f\) is of general growth with respect to \(y\). This condition  is weaker than those imposed in Pardoux and Peng \cite{94} and 
	Wu and Zhang \cite{11}. 
	However, when the generator \(f\) is of general growth in \(y\), the classical methods for BSDEs are difficult to apply to BDSDEs. The reason is that the classical approach relies on conditional expectations, whereas such conditional expectations cannot be used for BDSDEs because \(\mathbb{F} = \{\mathscr{F}_t; t \in [0,T]\}\) is not a filtration. 
	Therefore, to overcome these difficulties, we use a truncation argument together with suitable construction techniques.

	\begin{theorem}\label{t3.2}\sl
	Assume that conditions \textbf{(A1)}, \textbf{(F1)}, \textbf{(G1)},  and \textbf{(G3)}  are satisfied.  Moreover, assume that $f$ is  Lipschitz continuous  in $z$, i.e.,
		\begin{align*}
			\big| f(t,y,z)-f(t,y,z')\big|
			\leq C |z-z'|.
		\end{align*}
		Then  BDSDE \eqref{1} admits a unique solution
		$(Y,Z)\in S^2_{\mathbb{F}}([0,T];\mathbb{R})\times L^2_{\mathbb{F}}([0,T];\mathbb{R}^d)$. 
	\end{theorem}
	
	Before proving the existence and uniqueness result, we recall some facts concerning the Yosida approximation of monotone and continuous functions (see Da Prato and Zabczyk \cite{92d} and Hu \cite{00h}).

	\begin{lemma}\label{l3.5} \sl
		Let $F:\mathbb{R}^n\longmapsto \mathbb{R}^n$ be a continuous function satisfying
		\begin{align*}
			\left(x^1-x^2,F(x^1)-F(x^2)\right)\leq 0,\qq 	 x^1,x^2 \in \mathbb{R}^n.
		\end{align*}
		Then,  for any $\varepsilon>0$ and  $y\in\mathbb{R}^n$, there exists a unique $x=J^{\varepsilon}(y)$ such that
		\begin{align*}
			x-\varepsilon F(x)=y.
		\end{align*}
		Moreover,  for any $\e,\varepsilon^1,\varepsilon^2 >0$, and 	$x,x^1,x^2\in\dbR^n$,  we have
		\begin{align*}
			&	\big|J^{\varepsilon}(x^1)-J^{\varepsilon}(x^2)\big|\leq |x^1-x^2| 
			\q\hbox{and}\q 	\lim\limits_{\varepsilon\rightarrow 0}J^{\varepsilon}(x)=x,\\
			&	 \left(x^1-x^2,F^{\varepsilon}(x^1)-F^{\varepsilon}(x^2)\right)\leq 0\q\hbox{and}\q
			\big	|F^{\varepsilon}(x^1)-F^{\varepsilon}(x^2)\big|\leq \frac{2}{\varepsilon}|x^1-x^2|,  \\
			&	 |F^{\varepsilon}(x)|\leq |F(x)| \q\hbox{and}\q 
			\big(x^1-x^2,F^{\varepsilon^1}(x^1)-F^{\varepsilon^2}(x^2)\big)
			\leq (\varepsilon^1+\varepsilon^2)\big(|F(x^1)|+|F(x^2)|\big)^2,
		\end{align*}
		where $F^{\varepsilon}$  is the Yosida approximation of the function $F$ defined by
		\begin{align*}
			F^{\varepsilon}(x)=F(J^{\varepsilon}(x))=\frac{1}{\varepsilon} [J^{\varepsilon}(x)-x].
		\end{align*}
		Additionally, 
		for any set $\{x^{\varepsilon}\}_{{\varepsilon>0}}\subset\mathbb{R}^n$ and $x\in\mathbb{R}^n$, if
		$\lim\limits_{\varepsilon\rightarrow 0}x^{\varepsilon}=x,$ then
		\begin{align*}
			\lim\limits_{\varepsilon\rightarrow 0}F^{\varepsilon}(x^{\varepsilon})=F(x).
		\end{align*}
		
	\end{lemma}

Before proving Theorem \ref{t3.2}, we  present a result, which establishes the existence and uniqueness of solutions to BDSDE \eqref{1} in the case where the generator \(f\) is independent of \(z\) and the coefficient \(g\) is independent of \((y,z)\).
By \autoref{r3.1}, we may assume without loss of generality that \(\mu =0\).

	\begin{proposition}\label{p3.6} \sl
		Assume that the assumptions of \autoref{t3.2} hold,   the following BDSDE
		\begin{align}\label{3.2}
			Y_t=\xi+\int_t^Tf(s,Y_s)ds+\int_t^Tg(s)dB_s-\int_t^TZ_sdW_s, \q\ 0\leq t\leq T,
		\end{align}
		admits  a unique  solution  $(Y,Z)\in S^2_{\mathbb{F}}([0,T];\mathbb{R})\times L^2_{\mathbb{F}}([0,T];\mathbb{R}^d)$.
	\end{proposition}

	\begin{proof} 
		{\bf Existence.}
		For each $n\in\mathbb{N}$, define
		\begin{align}\label{3.4}
			\rho_n(x)=\frac{\inf\{n,|x|\}}{|x|}x\q\
			\hbox{and} \q\ \rho_n(0)=0.
		\end{align}
			Since the Brownian motions \(W\) and \(B\) are independent, it follows that the increment \(W_t-W_s\) is independent of \(\mathscr{G}_s\), and \(W\) remains
		a Brownian motion with respect to the enlarged filtration
		\(\mathbb{G}=(\mathscr{G}_t)_{0\leq t\leq T}\). 
		Apply the  result of  Pardoux and Peng (see  \cite[Proposition 1.2]{94}) under the filtration \(\mathbb{G}\), 
			\begin{align}\label{Y0}
			Y_t^0=\int_t^Tg(s)dB_s-\int_t^TZ_s^0dW_s,\q\ 0\leq t\leq T.
		\end{align}
			Moreover,
		\begin{align*}
			Y_t^0=\mathbb{E}\left[\int_t^T g(s)dB_s ~\Big | ~ \mathscr{G}_t \right], \q\ 0\leq t\leq T. 
		\end{align*}
		  In particular, the construction in that proposition provides a jointly measurable version of 
		the integrand \(Z^0\) such that \(Z_t^0\) is \(\mathscr{F}_t\)-measurable for \(dt\otimes d\mathbb{P}\)-almost every \((t,\omega)\).
Thus the solution pair   of the auxiliary BDSDE \eqref{Y0} can be chosen so that \((Y^0,Z^0)\in  S^2_{\mathbb{F}}([0,T];\mathbb{R})\times L^2_{\mathbb{F}}([0,T];\mathbb{R}^d)\).

		\ms
		
		{\bf\it  Step 1}: Assume that there is a positive  constant $c$ such that
		\begin{align*}
		\|\xi\|_{\infty}+\esssup\limits_{(t,\omega)}|f(t,0)|
		\leq c.
		\end{align*}
		Denote 
		\begin{align*}
			Y_t^{0,n}= \rho_n(Y_t^0) \q\ \hbox{and} \q\
			\bar{f}_n(t,y)=f(t,y+Y_t^{0,n}).
		\end{align*}
		Let $\bar{f}_n^{\varepsilon}$ with $\varepsilon>0$ be the Yosida approximations of the generator $\bar{f}_n$.
		By the classical result of Pardoux and Peng \cite{94}, 	the following BSDE
		\begin{equation}\label{3.3}
			\bar{Y}_t^{n,\varepsilon}=\xi+\int_t^T\bar{f}_n^{\varepsilon}(s,\bar{Y}_s^{n,\varepsilon})ds  -\int_t^T\bar{Z}_s^{n,\varepsilon}dW_s,\q\ 0\leq t\leq T,
		\end{equation}
		admits  a unique 	solution
		$(\bar{Y}^{n,\varepsilon},\bar{Z}^{n,\varepsilon})\in S^2_{\mathbb{F}}([0,T];\mathbb{R})\times L^2_{\mathbb{F}}([0,T];\mathbb{R}^d)$. 
		For the approximating equation, we perform the Yosida
	construction in the closed subspace of \(\mathbb{F}\)-measurable processes. Each approximation is \(\mathscr{F}_t\)-measurable by the 
	preceding representation result, and the convergence in 
	\( S^2_{\mathbb{F}}([0,T];\mathbb{R})\times L^2_{\mathbb{F}}([0,T];\mathbb{R}^d)\) preserves this measurability. Consequently the limiting pair admits an
	\(\mathbb{F}\)-measurable version.
		Applying It\^o's formula to $e^{  t}|\bar{Y}_t^{n,\varepsilon}|^2$, we have
		\begin{align*}
			e^{  t}|\bar{Y}_t^{n,\varepsilon}|^2 +\int_t^T e^s\big(|\bar{Y}_s^{n,\varepsilon}|^2+|\bar{Z}_s^{n,\varepsilon}|^2\big)ds =&\
			e^{  T}|\xi|^2+2\int_t^T    e^{  s}\bar{Y}_s^{n,\varepsilon}\bar{f}_n^{\varepsilon}(s,\bar{Y}_s^{n,\varepsilon}) ds -2\int_t^T e^{  s} \bar{Y}_s^{n,\varepsilon}\bar{Z}_s^{n,\varepsilon}dW_s.
		\end{align*}
		Note that 
		\begin{align*}
			2y\bar{f}_n^{\varepsilon}(t,y)
			\leq |y|^2+\big(| f(t,0)|+\varphi(n)\big)^2.
		\end{align*}
		Taking the conditional expectation on both sides, since the measurability holds under the enlarged filtration \(\mathbb{G}\), 
		we  obtain 
		\begin{align*}
			e^{  t}|\bar{Y}_t^{n,\varepsilon}|^2
			\leq \mathbb{E}\left[e^{  T} |\xi|^2 +\int_0^T e^s \big(| f(s,0)|+\varphi(n)\big)^2 ds ~\Big | ~ \mathscr{G}_t \right].
		\end{align*}
		Thus
		for any \(t\in[0,T]\),  
		\begin{align*}
			|\bar{Y}_t^{n,\varepsilon}| \leq \Bigg\{  \mathbb{E}\left[e^{  T}\left( |\xi|^2 +\int_0^T   \big(| f(s,0)|+\varphi(n)\big)^2 ds \right)~\Big | ~ \mathscr{G}_t \right]  \Bigg\} ^{\frac{1}{2}}
			\triangleq 
			X_t^n.
		\end{align*}
		Next, we prove that  the sequence $ (\bar{Y}^{n,\varepsilon},\bar{Z}^{n,\varepsilon})_{\varepsilon >0}$   converges  in  $S^2_{\mathbb{F}}([0,T];\mathbb{R})\times L^2_{\mathbb{F}}([0,T];\mathbb{R}^d)$.
		For this, applying
		It\^o's  formula to $\big|\bar{Y}_t^{n,\varepsilon^1}-\bar{Y}_t^{n,\varepsilon^2}\big|^2$, we have
		\begin{align*}
			\big|\bar{Y}_t^{n,\varepsilon^1}-\bar{Y}_t^{n,\varepsilon^2}\big|^2+\int_t^T\big|\bar{Z}_s^{n,\varepsilon^1}-\bar{Z}_s^{n,\varepsilon^2}\big|^2ds
			=&\ \int_t^T2\left(\bar{Y}_s^{n,\varepsilon^1}-\bar{Y}_s^{n,\varepsilon^2}\right)\left(\bar{f}_n^{\varepsilon^1}(s,\bar{Y}_s^{n,\varepsilon^1})-\bar{f}_n^{\varepsilon^2}(s,\bar{Y}_s^{n,\varepsilon^2})\right)ds\\
			& 
			-\int_t^T2\left(\bar{Y}_s^{n,\varepsilon^1}-\bar{Y}_s^{n,\varepsilon^2}\right)\left(\bar{Z}_s^{n,\varepsilon^1}-\bar{Z}_s^{n,\varepsilon^2}\right)dW_s.
		\end{align*}
		From \autoref{l3.5}, we have that
		\begin{align*}
			\left(\bar{Y}_s^{n,\varepsilon^1}-\bar{Y}_s^{n,\varepsilon^2}\right)\left(\bar{f}_n^{\varepsilon^1}(s,\bar{Y}_s^{n,\varepsilon^1})-\bar{f}_n^{\varepsilon^2}(s,\bar{Y}_s^{n,\varepsilon^2})\right)
			\leq (\varepsilon^1 +\varepsilon^2)\left(|\bar{f}_n^{\varepsilon^1}(s,\bar{Y}_s^{n,\varepsilon^1})|+|\bar{f}_n^{\varepsilon^2}(s,\bar{Y}_s^{n,\varepsilon^2})|\right)^2.
		\end{align*}
		Therefore, we  get 
		\begin{align}\label{e1e2}
			\mathbb{E}\left[\big|\bar{Y}_t^{n,\varepsilon^1}-\bar{Y}_t^{n,\varepsilon^2}\big|^2 \right]+ \mathbb{E} \int_t^T\big|\bar{Z}_s^{n,\varepsilon^1}-\bar{Z}_s^{n,\varepsilon^2}\big|^2ds  
			\leq  
			2(\varepsilon^1 +\varepsilon^2)\mathbb{E} \int_t^T\left(|\bar{f}_n^{\varepsilon^1}(s,\bar{Y}_s^{n,\varepsilon^1})|+|\bar{f}_n^{\varepsilon^2}(s,\bar{Y}_s^{n,\varepsilon^2})|\right)^2ds.
		\end{align}
		For each fixed \(n\in \mathbb{N}\) and any $\varepsilon >0$, 
		\begin{align*}
			|\bar{f}_n^{\varepsilon}(t,\bar{Y}^{n,\varepsilon})|\leq |f(t,0)| +\varphi(X_t^n+n)  \  \in L^2_{\mathbb{F}}([0,T];\mathbb{R}),
		\end{align*}
		which can be verified by Doob's inequality. Therefore, 
		let \(\varepsilon^1,\varepsilon^2\longrightarrow 0 \) in \eqref{e1e2}, we get
		\begin{align*}
			\sup\limits_{0\leq t\leq T}	\mathbb{E}\left[\big|\bar{Y}_t^{n,\varepsilon^1}-\bar{Y}_t^{n,\varepsilon^2}\big|^2 \right]+ \mathbb{E} \int_t^T\big|\bar{Z}_s^{n,\varepsilon^1}-\bar{Z}_s^{n,\varepsilon^2}\big|^2ds\longrightarrow 0.
		\end{align*}
		Moreover, using the Burkholder--Davis--Gundy inequality, we  derive that the sequence \((\bar{Y}^{n,\varepsilon})_{\varepsilon >0}\) is also a Cauchy sequence  in \(S^2_{\mathbb{F}}([0,T];\mathbb{R})\).
		This implies that there exists a limiting process $(\bar{Y}^n,\bar{Z}^n)\in S^2_{\mathbb{F}}([0,T];\mathbb{R})\times L^2_{\mathbb{F}}([0,T];\mathbb{R}^d)$ such that  
		\begin{align*}
			\mathbb{E}\left[\sup\limits_{0\leq t\leq T}\big|\bar{Y}_t^{n,\varepsilon}-\bar{Y}^n_t
			\big|^2
			+\int_{0}^{T}\big|\bar{Z}_s^{n,\varepsilon}-\bar{Z}^n_s\big|^2ds\right]\longrightarrow 0, \q\ \hbox{as}~ \varepsilon \longrightarrow 0.
		\end{align*}
		By letting $\varepsilon \longrightarrow 0$ in \eqref{3.3} and applying the  dominated convergence theorem, we can conclude that the limiting process $(\bar{Y}^n,\bar{Z}^n)\in S^2_{\mathbb{F}}([0,T];\mathbb{R})\times L^2_{\mathbb{F}}([0,T];\mathbb{R}^d)$  satisfies
		\begin{equation}\label{3.3 new}
			\bar{Y}_t^{n}=\xi+\int_t^T\bar{f}_n (s,\bar{Y}_s^{n} )ds  -\int_t^T\bar{Z}_s^{n }dW_s,\q\ 0\leq t\leq T,
		\end{equation}
		where
		\begin{align*}
			\bar{f}_n(t,y)\longrightarrow \bar{f}(t,y) \q\ \hbox{with} \q\ 	\bar{f}(t,y)=f(t,y+Y_t^0).
		\end{align*}
		Applying the same technique from \eqref{3.3} to \eqref{3.3 new}, 
		we deduce that, for any \(t\in[0,T]\),
		\begin{align*}
			|\bar{Y}_t^{n }| \leq \Bigg\{  \mathbb{E}\left[e^{  T}\left( |\xi|^2 +\int_0^T   \big(| f(s,0)|+\varphi(|Y_s^0|)\big)^2 ds \right)~\Big | ~ \mathscr{G}_t \right]  \Bigg\} ^{\frac{1}{2}}
			\triangleq 
			X_t^0.
		\end{align*}
		Note that under the assumptions of \autoref{t3.2},  	$ \bar{f}(t,y)$  is  continuous in \(y\).  For any $(t,\omega,y)\in [0,T]\times\Omega\times\mathbb{R}$,
		\begin{align*}
			|\bar{f}(t,y)|\leq |\bar{f}(t,0)|+\bar{\varphi}_t(|y|) \q\ \hbox{and}\q\
			(y-y')\big(	\bar{f}(t,y)-\bar{f}(t,y')\big) \leq 0,
		\end{align*}
		where \( \bar{\varphi}_t(|y|) =\varphi(|y|+|Y_t^0|)\). 
		From It\^o's formula,  we   have
		\begin{align*}
	 \mathbb{E}\left[\big|\bar{Y}_t^n-\bar{Y}_t^m\big|^2+\int_t^T\big|\bar{Z}_s^n-\bar{Z}_s^m\big|^2ds\right]
			=2  \mathbb{E} \int_t^T \left(\bar{Y}_s^n-\bar{Y}_s^m \right)\left(\bar{f}_n(s,\bar{Y}_s^n)
			-\bar{f}_m(s,\bar{Y}_s^m) \right) ds,
		\end{align*}
		where
		\begin{align*}
			\bar{f}_n(s,\bar{Y}_s^n)
			-\bar{f}_m(s,\bar{Y}_s^m) 
			=\big(\bar{f}(s,\bar{Y}_s^n)
			-\bar{f}(s,\bar{Y}_s^m) \big)+
			\big(\bar{f}_n(s,\bar{Y}_s^n)
			-\bar{f}(s,\bar{Y}_s^n) \big) 
			+\big(\bar{f}(s,\bar{Y}_s^m)
			-\bar{f}_m(s,\bar{Y}_s^m) \big).
		\end{align*}
		The first term can be controlled by monotonicity. For the last two terms,   we find that \(|\bar{Y}_t^n |\leq X_t^0\), \(\forall t\in [0,T]\), and 
		   the sequence \((\bar{Y}^n)_{n\geq 0}\) is bounded  in \(S^2_{\mathbb{F}}([0,T];\mathbb{R})\),  then, by the growth condition of \(\bar{f}_n\) and a standard localization argument, we can deduce the uniform integrability of the following two products:
		\begin{align*}
			\left(\bar{Y}_s^n-\bar{Y}_s^m \right)\left(\bar{f}_n(s,\bar{Y}_s^n)
			-\bar{f}(s,\bar{Y}_s^n) \right)
			\q\
			\hbox{and}\q\
			\left(\bar{Y}_s^n-\bar{Y}_s^m \right)\left(\bar{f}(s,\bar{Y}_s^m)
			-\bar{f}_m(s,\bar{Y}_s^m) \right).
		\end{align*}
		Moreover, 
		since \( \bar{f}_n(\omega,t,y)\longrightarrow \bar{f}(\omega,t,y)\)   \(dt\otimes d\mathbb{P}\)-a.e.  and \(\varphi(|Y^0|)\in
		L^2_{\mathbb{F}}([0,T];\mathbb{R})\), which follows from condition  \(\textbf{(G3)}\).
		Applying  Vitali's convergence theorem,	we can take the limit as $n \longrightarrow \infty$ in \eqref{3.3 new}. This yields a limiting process  \( (\bar{Y},\bar{Z})\in S^2_{\mathbb{F}}([0,T];\mathbb{R})\times L^2_{\mathbb{F}}([0,T];\mathbb{R}^d) \)  satisfying the following BSDE:
		\begin{align}\label{Ynn}
			\bar{Y}_t=\xi+\int_t^T\bar{f}(s,\bar{Y}_s)ds -\int_t^T\bar{Z}_sdW_s,\q\ 0\leq t\leq T.
		\end{align}
		Define \(   Y_t =	\bar{Y}_t  + Y_t^{0}, Z_t=\bar{Z}_t+Z_t^0  \), then the process $(Y,Z)\in S^2_{\mathbb{F}}([0,T];\mathbb{R})\times L^2_{\mathbb{F}}([0,T];\mathbb{R}^d)$ solves BDSDE \eqref{3.2}.
		
		\ms
		
		{\it Step 2}: The general case where  $\xi\in L^2_{\mathscr{F}_T}(\Omega;\mathbb{R})$ and $f(\cdot,0)\in L^2_{\mathbb{F}}([0,T];\mathbb{R})$.
		Denote
		\begin{align*}
			\xi_n=\rho_n(\xi),\q\ \hbox{and} \q\   f_n(t,y)=f(t,y)-f(t,0)+\rho_n(f(t,0)).
		\end{align*}
	It follows from Step 1 that, for each \(n\in\mathbb{N}\), there exists a  solution $(Y^n,Z^n)\in S^2_{\mathbb{F}}([0,T];\mathbb{R})\times L^2_{\mathbb{F}}([0,T];\mathbb{R}^d)$ to the following  BDSDE: 
		\begin{align}\label{3.5}
			Y_t^n=\xi_n+\int_t^Tf_n(s,Y_s^n)ds+\int_t^Tg(s)dB_s-\int_t^TZ_s^ndW_s,\q\ 0\leq t\leq T.
		\end{align}
		From It\^o's formula, one has
		\begin{align*}
			\mathbb{E}\left[\big|Y_t^n-Y_t^m\big|^2+\int_t^T\big|Z_s^n-Z_s^m\big|^2ds\right]  
			=\mathbb{E}\left[|\xi_n-\xi_m|^2\right]
			+2\mathbb{E}\int_t^T\left(Y_s^n-Y_s^m\right)\big(f_n(s,Y_s^n)-f_m(s,Y_s^m)\big)ds,
		\end{align*}
		where
		\begin{align*}
			2 \left(Y_s^n-Y_s^m\right)\big(f_n(s,Y_s^n)-f_m(s,Y_s^m)\big)
			\leq \big|Y_s^n-Y_s^m\big|^2+\big|f_n(s,0)-f_m(s,0)\big|^2.
		\end{align*}
		Note that
		\begin{align*}
			\mathbb{E}\left[|\xi_n-\xi|^2\right]\longrightarrow 0\q~\hbox{and}\q~
			\mathbb{E}\int_{0}^{T} \big|f_n(t,0)-f(t,0)\big|^2 dt \longrightarrow 0
			\quad \hbox{as}\q n\longrightarrow \infty.
		\end{align*}
		Following Gronwall's inequality and the Burkholder--Davis--Gundy  inequality, we deduce that
		\begin{align*}
			\mathbb{E}\left[\sup\limits_{0\leq t\leq T}\big|Y_t^n-Y_t^m\big|^2+\int_0^T\big|Z_s^n-Z_s^m\big|^2ds\right]\longrightarrow 0
			\quad \hbox{as}\q  n,m\longrightarrow \infty.
		\end{align*}
		Passing to a subsequence if necessary, we may assume
that \((Y^n,Z^n)\longrightarrow (Y,Z)\), \(dt\otimes d\mathbb{P}\)-a.e. For almost every \((t,\omega)\), the convergent sequence \((Y^n,Z^n)_{n\geq 1}\) is contained in a compact subset of \(\mathbb{R}\times \mathbb{R}^d\). 
Since \(f_n(t,\cdot)\longrightarrow f(t,\cdot)\) locally uniformly in \(y\), we obtain \(f_n(t,Y^n_t)\longrightarrow f(t,Y_t)\), \(dt\otimes d\mathbb{P}\)-a.e.
		Finally, by letting $n\longrightarrow \infty$ in \eqref{3.5} and
		by localization together with Vitali's convergence theorem,
		we conclude that there exists a pair of processes $(Y, Z)\in S^2_{\mathbb{F}}([0,T];\mathbb{R})\times L^2_{\mathbb{F}}([0,T];\mathbb{R}^d)$ which is a solution to BDSDE \eqref{3.2}.
		
		\ms

		{\bf Uniqueness.} 
		If there exist two solutions $(Y^1,Z^1), (Y^2,Z^2) \in S^2_{\mathbb{F}}([0,T];\mathbb{R})\times L^2_{\mathbb{F}}([0,T];\mathbb{R}^d)$ which satisfy    BDSDE \eqref{3.2}. Applying It\^o's formula to $|Y_t^1-Y_t^2|^2$, this leads
		to the  following result:
		\begin{align*}
			\big|Y_t^1-Y_t^2\big|^2=  \int_t^T  \bigg(2 (Y_s^1-Y_s^2)\big(f(s,Y_s^1)-f(s,Y_s^2)\big)-\big|Z_s^1-Z_s^2\big|^2\bigg)ds
			-2\int_t^T(Y_s^1-Y_s^2)(Z_s^1-Z_s^2)dW_s.
		\end{align*}
		Since $f$ is monotone in $y$, we have
		\begin{align*}
			\mathbb{E}\left[\big|Y_t^1-Y_t^2\big|^2\right]+\mathbb{E}\int_t^T\big|Z_s^1-Z_s^2\big|^2ds
			\leq 0,
		\end{align*}
		which implies that \(Y^1=Y^2,\) \(\mathbb{P}\)-a.s. and \(Z^1=Z^2\), \(dt\otimes d\mathbb{P}\)-a.e.
	\end{proof}

	Now we give the proof of the main result in this subsection.

	\begin{proof}[Proof of  \autoref{t3.2}]
		The uniqueness can be established similarly to \autoref{p3.6}. Here, we will focus on proving the existence.
		By \autoref{p3.6}, for any given $(U,V)\in L^2_{\mathbb{F}}([0,T];\mathbb{R})\times L^2_{\mathbb{F}}([0,T];\mathbb{R}^d)$, there exists a unique solution  $(Y,Z)\in L^2_{\mathbb{F}}([0,T];\mathbb{R})\times L^2_{\mathbb{F}}([0,T];\mathbb{R}^d)$
		such that
		\begin{align}\label{uvyz}
			Y_t=\xi +\int_t^Tf(s,Y_s,V_s)ds +\int_t^T g(s,U_s,V_s)dB_s-\int_t^T Z_sdW_s,\q\ 0\leq t\leq T.
		\end{align}
		We can construct a mapping $\Theta$ as follows:  let $(Y,Z)=\Theta (U,V)$ be a solution of BDSDE \eqref{uvyz}.
		For given $(U,V), (U',V')\in L^2_{\mathbb{F}}([0,T];\mathbb{R})\times L^2_{\mathbb{F}}([0,T];\mathbb{R}^d)$,  denote $(Y,Z)=\Theta (U,V)$, $(Y',Z')=\Theta (U',V')$, and
		\begin{align*}
			\triangle U=U-U',\qquad \triangle V=V-V',\qq
			\triangle Y=Y-Y',\qquad \triangle Z=Z-Z'.
		\end{align*} 
		For any $\gamma\in \mathbb{R}$, applying It\^o's formula to $e^{\gamma t}|\triangle Y_t|^2$, we have
		\begin{align*}
			&\	e^{\gamma t}|\triangle Y_t|^2+\int_t^Te^{\gamma s}\big(\gamma |\triangle Y_s|^2+|\triangle Z_s|^2\big)ds\\
			& 		=2\int_t^T e^{\gamma s} \triangle Y_s\Big(f(s,Y_s,V_s)-f(s,Y'_s,V'_s)\Big)ds
			+\int_t^T e^{\gamma s} \Big(g(s,U_s,V_s)-g(s,U'_s,V'_s)\Big)^2 ds\\
			&+2\int_t^T e^{\gamma s} \triangle Y_s\Big(g(s,U_s,V_s)-g(s,U'_s,V'_s)\Big)dB_s
			-2\int_t^T  e^{\gamma s} \triangle Y_s \triangle Z_s dW_s. 
		\end{align*}
		In light of the monotonicity of  $f$  and the Lipschitz continuity of $g$, taking expectations on both sides, we deduce that 
		\begin{align*}
			\mathbb{E}\int_t^Te^{\gamma s}\big(\gamma |\triangle Y_s|^2+|\triangle Z_s|^2\big)ds 
			\leq \mathbb{E}\int_t^T \big( 2Ce^{\gamma s}\triangle Y_s\triangle V_s \big) ds
			+\dbE\int_t^Te^{\gamma s}\big(C|\triangle U_s|^2+ \alpha 
			|\triangle V_s|^2 \big)ds.
		\end{align*}
		By the elementary inequality, we have
		\begin{align*}
			2Ce^{\gamma s}\triangle Y_s \triangle V_s \leq 
			e^{\gamma s}\left(\frac{2C^2}{1-\alpha}|\triangle Y_s|^2 +\frac{1-\alpha}{2}|\triangle V_s|^2\right).
		\end{align*}
		Choosing $\gamma= \frac{2C^2}{1-\alpha}+\frac{2C}{1+\alpha}$,  we obtain
		\begin{align*}
			\mathbb{E}\int_t^Te^{\frac{2C^2}{1-\alpha}+\frac{2C}{1+\alpha} s}\left( \frac{2C}{1+\alpha}|\triangle Y_s|^2+|\triangle Z_s|^2\right)ds
			&\leq  \mathbb{E}\int_t^T e^{\frac{2C^2}{1-\alpha}+\frac{2C}{1+\alpha} s}\left(C|\triangle U_s|^2+\frac{1+\alpha}{2}|\triangle V_s|^2\right)ds\\
			&	=\frac{1+\alpha}{2}\mathbb{E} \int_t^T e^{\frac{2C^2}{1-\alpha}+\frac{2C}{1+\alpha} s}\left( \frac{2C}{1+\alpha}|\triangle U_s|^2+|\triangle V_s|^2\right)ds.
		\end{align*}
		Therefore, $\Theta$ is a strict contraction mapping
		under the equivalent weighted norm, and hence there exists a fixed point $(Y,Z)\in L^2_{\mathbb{F}}([0,T];\mathbb{R})\times L^2_{\mathbb{F}}([0,T];\mathbb{R}^d)$. By the Burkholder--Davis--Gundy inequality,
		we can verify that $Y\in S^2_{\mathbb{F}}([0,T];\mathbb{R})$ and the fixed point is a  solution of BDSDE \eqref{1}.
	\end{proof}
	
	As a  result of \autoref{t3.2}, we have the following comparison theorem.

	\begin{proposition}[Comparison]\label{c3.7} \sl
		For \(i=1,2\), assume that the parameters \((\xi^i,f^i,g^i)\) satisfy
		\textbf{(A1)}, \textbf{(F1)}, \textbf{(G1)}  and \textbf{(G3)}  with common structural constants, and that \(f^i\) is Lipschitz continuous in \(z\).  Let $(Y^i,Z^i)\in S^2_{\mathbb{F}}([0,T];\mathbb{R})\times L^2_{\mathbb{F}}([0,T];\mathbb{R}^d)$  be the unique solution of the following BDSDE:
		\begin{align*}
			Y^i_t=\xi^i+\int_{t}^{T}f^i(s,Y^i_s,Z^i_s)ds+\int_{t}^{T}g^i(s,Y^i_s,Z^i_s)dB_s-\int_t^TZ^i_sdW_s,\q\ 0\leq t\leq T.
		\end{align*}
		If,  for any \(t\in[0,T]\),
		\begin{align*}
			\xi^1\leq \xi^2,\q~ f^1(t,Y^2_t,Z_t^2)\leq f^2(t,Y^2_t,Z_t^2), \q~ g^1(t,Y^2_t,Z_t^2)= g^2(t,Y^2_t,Z_t^2), \q \hbox{a.s.}, 
		\end{align*}
		\begin{align*}
			\Big( \hbox{alternatively,}\q~ 
			f^1(t,Y^1_t,Z_t^1)\leq f^2(t,Y^1_t,Z_t^1), \q~ g^1(t,Y^1_t,Z_t^1)= g^2(t,Y^1_t,Z_t^1), \q \hbox{a.s.}, \Big)
		\end{align*}
		then we have $Y^1_t\leq Y^2_t$, \(\forall t\in [0,T]\), \(\mathbb{P}\)-a.s.
	\end{proposition}

	\begin{proof}
		Since \(y\longmapsto (y^+)^2\) is not smooth enough, we use an approximation.
		For any $\varepsilon>0$, we  consider a $C^2(\mathbb{R})$ function $\phi_{\varepsilon}(y):\mathbb{R}\longmapsto \mathbb{R}$ as follows:
		\begin{equation}\label{itoy}
			\phi_{\varepsilon}(y)=\left\{
			\begin{aligned}
				-\varepsilon y-\frac{\varepsilon^2}{2},\qquad \qquad \q\  &\q y\leq-\varepsilon, \\
				\frac{y^5}{\varepsilon^3}+\frac{5y^4}{2\varepsilon^2}+\frac{2y^3}{\varepsilon}+y^2,\qquad &-\varepsilon <y <0,\\
				y^2,\qquad \qquad \q\  &\q y\geq 0.
			\end{aligned}
			\right.
		\end{equation}
		A direct computation shows that
		\begin{equation*}
			\phi'_{\varepsilon}(y)=\left\{
			\begin{aligned}
				-\varepsilon,\qquad \qquad \q\  &\q y\leq-\varepsilon, \\
				\frac{5y^4}{\varepsilon^3}+\frac{10y^3}{\varepsilon^2}+\frac{6y^2}{\varepsilon}+2y,\qquad &-\varepsilon <y <0,\\
				2y,\qquad \qquad \q\  &\q y\geq 0,
			\end{aligned}
			\right.
		\end{equation*}
		and
		\begin{equation*}
			\phi''_{\varepsilon}(y)=\left\{
			\begin{aligned}
				0,\qquad \qquad \q\  &\q y\leq-\varepsilon, \\
				\frac{20y^3}{\varepsilon^3}+\frac{30y^2}{\varepsilon^2}+\frac{12y}{\varepsilon}+2,\qquad &-\varepsilon <y <0,\\
				2,\qquad \qquad \q\  &\q y\geq 0.
			\end{aligned}
			\right.
		\end{equation*}
		We can confirm the continuity of \(\phi'_{\varepsilon}\)
		and  \(\phi''_{\varepsilon}\) at the joints.
		Moreover,
		for any $y\in \mathbb{R}$, when $\varepsilon \longrightarrow 0$, we have
		\begin{align*}
			\phi_{\varepsilon}(y)\longrightarrow (y^+)^2,\qquad
			\phi_{\varepsilon}'(y) \longrightarrow 2y^+,\qquad
			\phi_{\varepsilon}''(y) \longrightarrow 2\mathbf{1}_{\{y\geq 0\}}.
		\end{align*}
		Applying  It\^o's formula to $\phi_{\varepsilon}\left(Y_t^1-Y_t^2\right)$, we have
		\begin{align*}
			\phi_{\varepsilon}\left(Y_t^1-Y_t^2\right)=&\ \phi_{\varepsilon}\left(\xi^1-\xi^2\right)
			+\int_t^T\phi_{\varepsilon}'(Y_s^1-Y_s^2)\Big(f^1(s,Y^1_s,Z^1_s)-f^2(s,Y_s^2,Z_s^2) \Big)ds\\
			&	+\int_t^T\phi_{\varepsilon}'(Y_s^1-Y_s^2)\Big(g^1(s,Y^1_s,Z^1_s)-g^2(s,Y_s^2,Z_s^2)\Big)dB_s
			-\int_t^T\phi_{\varepsilon}'(Y_s^1-Y_s^2)(Z_s^1-Z_s^2)dW_s\\
			&	+\frac{1}{2}\int_t^T\phi_{\varepsilon}''(Y_s^1-Y_s^2)\left(\Big|g^1(s,Y^1_s,Z^1_s)-g^2(s,Y_s^2,Z_s^2)\Big|^2-\big|Z_s^1-Z_s^2\big|^2\right)ds.
		\end{align*}
		Letting $\varepsilon\longrightarrow 0$ yields
		\begin{align*}
			\left((Y_t^1-Y_t^2)^+\right)^2
			=&\ \left((\xi^1-\xi^2)^+\right)^2
			+2\int_t^T(Y^1_s-Y^2_s)^+\Big(f^1(s,Y^1_s,Z^1_s)-f^2(s,Y_s^2,Z_s^2)\Big)ds\\
			&	+2\int_t^T(Y^1_s-Y^2_s)^+\Big(g^1(s,Y^1_s,Z^1_s)-g^2(s,Y_s^2,Z_s^2)\Big)dB_s
			-2\int_t^T(Y^1_s-Y^2_s)^+(Z_s^1-Z_s^2)dW_s\\
			&	+\int_t^T\left(\Big|g^1(s,Y^1_s,Z^1_s)-g^2(s,Y_s^2,Z_s^2)\Big|^2-\big|Z_s^1-Z_s^2\big|^2\right)\mathbf{1}_{\{Y_s^1\geq Y_s^2\}}ds.
		\end{align*}
		Since $f$ is monotone in \(y\) and Lipschitz in $z$, we have
		\begin{align*}
			\ 2(Y_t^1-Y_t^2)^+\big(f^1(t,Y^1_t,Z^1_t)-f^2(t,Y_t^2,Z_t^2)\big)
			\leq \frac{2C^2}{1-\alpha} \left((Y_t^1-Y_t^2)^+\right)^2
			+\frac{1-\alpha}{2}\big|Z_t^1-Z_t^2\big|^2.
		\end{align*}
		Taking expectations on both sides, we obtain
		\begin{align*}
			\mathbb{E}\left[ \left((Y_t^1-Y_t^2)^+\right)^2\right]+\frac{1-\alpha}{2}
			\mathbb{E} \int_t^T\big|Z_s^1-Z_s^2\big|^2ds
			\leq \mathbb{E}\left((\xi^1-\xi^2)^+\right)^2
			+\left(C +\frac{2C^2}{1-\alpha}\right)\mathbb{E}
			\int_t^T \left((Y_s^1-Y_s^2)^+\right)^2 ds.
		\end{align*}
		In view of the condition that	$\xi^1\leq \xi^2$,  Gronwall's inequality then yields that
		\begin{align*}
			\mathbb{E}\left[ \left((Y_t^1-Y_t^2)^+\right)^2\right]=0.
		\end{align*}
		Therefore,  $Y_t^1\leq Y_t^2$ for any \(t\in [0,T]\), $ \dbP $-a.s. By the continuity of  \(Y^1\) and \(Y^2\), the inequality
		holds simultaneously for all \(t\in [0,T]\), $ \dbP $-a.s.
	\end{proof}

	\subsection{BDSDE: existence of maximal solutions}\label{sec3.2}
	In this part, we study the existence of maximal solutions for BDSDEs with continuous coefficients, which is an application of the comparison theorem (\autoref{c3.7}). Compared to the work of Shi, Gu, and Liu \cite{05},  we weaken the condition that  \(f\)  is of linear growth in \(y\) to general growth in \(y\). 
	The proof requires suitable
	approximations of continuous functions and
	truncation arguments.

	\begin{theorem}\label{t3.3} \sl
		Let the conditions \(\textbf{(A1)}\), \(\textbf{(F1)}\), \(\textbf{(G1)}\), \(\textbf{(G2)}\) and \(\textbf{(G3)}\) be satisfied. Then  BDSDE \eqref{1}
		has a maximal solution $(Y,Z)\in S^2_{\mathbb{F}}([0,T];\mathbb{R})\times L^2_{\mathbb{F}}([0,T];\mathbb{R}^d)$.
	\end{theorem}

	\begin{proof}
		For each	$n\in\mathbb{N}$  with $n\geq C$, we define 
		\begin{align*}
			f_n(t,y,z)=\sup\limits_{q\in\mathbb{Q}^d}\{f(t,y,q)-n|z-q|\}.
		\end{align*}
		From the above definition,  	it is easy to check that  
		for any \( (t,y,y',z,z')\in[0,T]\times \mathbb{R}\times\mathbb{R}\times \mathbb{R}^d \times \mathbb{R}^d\), 
		$$ \left\{	\begin{aligned}
			& (y-y')\big(f_n(t,y,z)-f_n(t,y',z)\big)\leq 0,\\
			&|f_n(t,y,z)-f_n(t,y,z')|\leq n|z-z'|,\\
			& |f_n(t,y,z)|\leq |f(t,0,0)|+\varphi(|y|)+C|z|.
		\end{aligned} \right.$$
		According to \autoref{t3.2}, 
		for each $n\geq C$,	  there exists a unique solution $(Y^n,Z^n)\in S^2_{\mathbb{F}}([0,T];\mathbb{R})\times L^2_{\mathbb{F}}([0,T];\mathbb{R}^d)$ satisfying
		\begin{align}\label{3.7}
			Y_t^n=\xi+\int_t^Tf_n(s,Y_s^n,Z_s^n)ds+\int_t^Tg(s,Y_s^n,Z_s^n)dB_s-\int_t^TZ_s^ndW_s,\q\ 0\leq t\leq T.
		\end{align}
		Next, we divide the proof into two steps.
		\ms
		
		{\bf\it Step 1}: Assume  that there is a positive  constant $c$ such that
		\begin{align*}
			\|\xi\|_{\infty}+\esssup\limits_{(t,\omega)}|f(t,0,0)|
		\leq c.
		\end{align*}
		
		Let \( \bar{\gamma}=\frac{\gamma -1-\frac{C^2}{1-\alpha}}{2e^{\gamma T}c^2} >0 \) with
		 \(\gamma\in \mathbb{R}\) that has yet to be chosen,
		  define
		\begin{align*}
			\hat{X}_t=e^{\gamma T}\|\xi\|_{\infty}^2+\int_t^Te^{\gamma s}\big(\|f(s,0,0)\|_{\infty}^2+|\lambda(s)|^2 \big)ds.
		\end{align*}
		Consider  a non-negative convex function $\Phi(u) : \mathbb{R} \longmapsto \mathbb{R}$ as follows:
		\begin{equation}\label{t3.6phi}
			\Phi(u)=\left\{
			\begin{aligned}
				\frac{1}{\bar{\gamma}}(e^{\bar{\gamma}u}-1),\qquad& u>0; \\
				\quad 0,\qquad  	&u\leq 0.
			\end{aligned}
			\right.
		\end{equation}
	Note that  \(\Phi(u)=0\) if and only if \(u\leq 0\). For any \(u>0\), \(\Phi'(u)\geq 0\),  \(\Phi''(u)\geq 0\),  and 
	\begin{align*}
	 2e^{\gamma T}c^2\Phi''(u)
	 =\Big( \gamma -1 - \frac{C^2}{1-\alpha} \Big)\Phi'(u).
	\end{align*}
		To make the following argument rigorous, let \((\Phi_\varepsilon)_{\varepsilon>0}\subset C^2(\mathbb{R})\) be a sequence of non-negative convex functions such that
		\begin{align*}
		 \Phi_\varepsilon(u)\longrightarrow \Phi(u),\q\
		 \Phi_\varepsilon'(u)\longrightarrow \Phi'(u),\q\
		 \Phi_\varepsilon''(u)\longrightarrow \Phi''(u),\q\ 
		  u\neq0,
		\end{align*}
		and such that the above differential inequality holds up to an error term which vanishes as \(\varepsilon\longrightarrow 0\) on compact subsets of \(\mathbb R\). Applying It\^o's formula to
		\( 
		\Phi_\varepsilon (e^{\gamma t}|Y_t^n|^2-\hat X_t),
		\)
		taking expectations, and then letting \(\varepsilon\longrightarrow 0\),   by the monotone convergence theorem and the dominated convergence theorem, we have
		\begin{align*}
		 	\mathbb{E}\left[\Phi(e^{\gamma t}|Y_t^{n}|^2-\hat{X}_t)\right]
		\leq  &\
		\mathbb{E}\int_t^T \Phi'(e^{\gamma s}|Y_s^{n}|^2-\hat{X}_s)e^{\gamma s}(1+\frac{C^2}{1-\alpha}-\gamma)|Y_s^{n}|^2ds\\
		&+\mathbb{E}\int_t^T \Big( 
		2\Phi''(e^{\gamma s}|Y_s^{n}|^2-\hat{X}_s) e^{2\gamma s}|Y_s^{n}|^2- \Phi'(e^{\gamma s}|Y_s^{n}|^2-\hat{X}_s)e^{\gamma s} \Big)
		|\lambda(s)|^2 ds,
		\end{align*}
		where we used the following estimate
		\begin{align*}
			2Y_t^{n}f_n(t,Y_t^{n},Z_t^n)
			&\leq 2Y_t^nf_n(t,0,Z_t^n)\leq 2Y_t^n\left(|f(t,0,0)|+C|Z_t^n|\right)\\
			&\leq (1+\frac{C^2}{1-\alpha})|Y_t^n|^2+|f(t,0,0)|^2+(1-\alpha)|Z_t^n|^2.
		\end{align*}
		Since \(\Phi'\) and \(\Phi''\) vanish on the set \((-\infty,0)\), 
		choosing \(\gamma > 1+ \frac{C^2}{1-\alpha}\),
		we obtain 
		\begin{align*}
			\mathbb{E}\left[\Phi(e^{\gamma t}|Y_t^{n}|^2-\hat{X}_t)\right]
			\leq 0.
		\end{align*}
		Since \(\Phi \geq 0\), we deduce that \(\Phi=0\), \(\mathbb{P}\)-a.s., and then for any \(t\in [0,T]\),
		\begin{align*}
			e^{\gamma t}|Y_t^{n}|^2
			\leq \hat{X}_t, \q\     \hbox{a.s.}
		\end{align*}
		Note that the sequence $(f_n)_{n\geq C}$ is non-increasing in $n$, it follows from the comparison theorem (\autoref{c3.7}) that,
		for any $t\in [0,T]$, \(Y_t^n\geq Y_t^{n+1}\).
	Moreover,
		for every fixed \((t,\omega)\in [0,T]\times \Omega\), 
		the sequence
		$f_n(t,\omega,\cdot,\cdot)$ is continuous and decreases pointwise to
		$f(t,\omega,\cdot,\cdot)$. Since \(f(t,\omega,\cdot,\cdot)\) is continuous,
		  Dini's theorem yields that    $f_n(t,\omega,y,z)$ converges  to $f(t,\omega,y,z)$ uniformly on the compact sets of $\mathbb{R}\times
		\mathbb{R}^d$.
		Let \(H\) denote the uniform bound for \((Y^n)_{n\geq C}\) obtained above, and set
		\begin{align*}
	 \eta_t\triangleq |f(t,0,0)|+\varphi(H).
		\end{align*}
		 Then 
		\begin{align*}
		 |f_n(t,Y_t^n,Z_t^n)|
		 \leq \eta_t+C|Z_t^n|.
		\end{align*}
		Applying It\^o's formula to \(|Y_t^n|^2\), using condition \textbf{(G2)} and   Young's inequality, yields that
		\begin{align*}
	 \sup\limits_{n\geq C}\mathbb{E}\int_0^T |Z_s^n|^2ds <\infty.
		\end{align*}
		Since  \(Y^n\) decreases pointwise to some process \(Y\) and \(|Y^n|\leq H\), the dominated convergence theorem gives
		\begin{align*}
\mathbb{E}\int_0^T |Y_s^n-Y_s|^2ds \longrightarrow 0.
		\end{align*}
		For \(n, m\geq C\),  applying It\^o's formula to \(|Y_t^n-Y_t^m|^2\),
		using condition \textbf{(G1)} together with the above growth estimate, we deduce that
		\begin{align*}
		 (1-\alpha)\mathbb{E}\int_0^T |Z_s^n-Z_s^m|^2ds 
		\longrightarrow 0.
		\end{align*}
		Consequently, \((Z^n)\) is a Cauchy sequence in \(L^2_{\mathbb{F}}([0,T];\mathbb{R}^d)\),  hence   $Z^n$ strongly converges to $Z$ in $L^2_{\mathbb{F}}([0,T];\mathbb{R}^d)$.
		Furthermore, since \(f_n\) locally uniformly converges to \(f\) and   \( 	 |f_n(t,Y_t^n,Z_t^n)|
		\leq \eta_t+C|Z_t^n|\). By the uniform integrability of \(|Z^n|^2\)  and
		 applying Vitali's convergence theorem, we have
		\begin{align*}
		 f_n(\cdot,Y^n,Z^n)\longrightarrow f(\cdot,Y,Z)\q\ 
		 \hbox{in} \ L^2_{\mathbb{F}}([0,T];\mathbb{R}).
		\end{align*}
		Passing to the limit in \eqref{3.7}, and then using the Burkholder--Davis--Gundy inequality, we obtain 
	 \(Y^n\) converges to  \(Y\) in \(S^2_{\mathbb{F}}([0,T];\mathbb{R})\).
		  Thus,  the pair $(Y,Z) \in S^2_{\mathbb{F}}([0,T];\mathbb{R})\times L^2_{\mathbb{F}}([0,T];\mathbb{R}^d)$ is a solution of BDSDE \eqref{1}. Next,
		we  need to  verify that $(Y,Z)$ is a maximal solution. Assume that $(Y',Z')\in S^2_{\mathbb{F}}([0,T];\mathbb{R})\times L^2_{\mathbb{F}}([0,T];\mathbb{R}^d)$ is another solution to BDSDE \eqref{1}. Since $f_n$ is decreasing in $n$ and Lipschitz in 
		$z$, and  we have $f\leq f_n$. It follows from the comparison theorem (\autoref{c3.7}) that $Y'\leq Y^n$.
		Taking the limit as $n\longrightarrow \infty$, we conclude that $Y'\leq Y$.
		
		\ms
		
		{\bf\it Step 2}: The general case where $\xi\in L^2_{\mathscr{F}_T}(\Omega;\mathbb{R})$ and $f(\cdot,0,0)\in L^2_{\mathbb{F}}([0,T];\mathbb{R})$.

		In order to transition from bounded data to square-integrable data, we adopt a two-sided localization,
		let
		\begin{align*}
			&	\xi_{n,m}=-m\vee( \xi\wedge n) \q\  \hbox{and} \q\  f_{n,m}(t,y,z)=f(t,y,z)-f(t,0,0)+f_t^{n,m}, 
		\end{align*}
		where \(f_t^{n,m}=-m\vee( f(t,0,0)\wedge n).  \)
		For fixed \(m\in \mathbb{N}\), we first pass to the limit as 
		\(n\longrightarrow\infty\). Indeed,
  it follows from Step 1 that,
   BDSDE \eqref{1} with parameters $(\xi_{n,m},f_{n,m},g)$ has a maximal solution
		$(Y^{n,m},Z^{n,m})\in S^2_{\mathbb{F}}([0,T];\mathbb{R})\times L^2_{\mathbb{F}}([0,T];\mathbb{R}^d)$ satisfying
		\begin{align} \label{3.8}
			Y_t^{n,m}=\xi_{n,m}+\int_t^Tf_{n,m}(s,Y_s^{n,m},Z_s^{n,m})ds+\int_t^Tg(s,Y_s^{n,m},Z_s^{n,m})dB_s-\int_t^TZ_s^{n,m}dW_s,\q\ 0\leq t\leq T,
		\end{align}
		where the sequence \((\xi_{n,m},f_{n,m})_n\) is non-decreasing in \(n\),
	and \(\xi_{n,m}\longrightarrow (-m\vee \xi)\) in \(L^2_{\mathscr{F}_T}(\Omega;\mathbb{R})\),
		\(f_t^{n,m}\longrightarrow (-m\vee f(t,0,0))\) in 
		\(L^2_{\mathbb{F}}([0,T];\mathbb{R})\), the same
		stability estimate as in Step 1 yields that
		\begin{align*}
		\mathbb{E}\left[ \sup\limits_{0\leq t\leq T}\big| 
		Y_t^{n,m}-Y_t^{k,m}\big|^2+\int_{0}^{T}\big|Z_t^{n,m}-Z_t^{k,m}\big|^2dt
		\right]\longrightarrow 0.
		\end{align*}
		Hence,  there exists  a limiting pair \((Y^m,Z^m) \in S^2_{\mathbb{F}}([0,T];\mathbb{R})\times L^2_{\mathbb{F}}([0,T];\mathbb{R}^d)\) such that
		\(Y^{n,m}\longrightarrow Y^m\) in \(  S^2_{\mathbb{F}}([0,T];\mathbb{R}) \) and
		\(Z^{n,m}\longrightarrow Z^m\) in \( L^2_{\mathbb{F}}([0,T];\mathbb{R}^d)\),
		which is a solution of BDSDE \eqref{1} with parameters \((\xi_m,f_m,g)\).
		Then,
		taking the limit as $m\longrightarrow \infty$, 
		since \(\xi_m\longrightarrow \xi\) in \(L^2_{\mathscr{F}_T}(\Omega;\mathbb{R})\), and \(f^m(\cdot,0,0)\longrightarrow f(\cdot,0,0)\) in 	\(L^2_{\mathbb{F}}([0,T];\mathbb{R})\), the stability estimate gives
			\(Y^{m}\longrightarrow Y\) in \(  S^2_{\mathbb{F}}([0,T];\mathbb{R}) \) and
		\(Z^{m}\longrightarrow Z\) in \( L^2_{\mathbb{F}}([0,T];\mathbb{R}^d) \).
		The solution
		\((Y,Z) \in S^2_{\mathbb{F}}([0,T];\mathbb{R})\times L^2_{\mathbb{F}}([0,T];\mathbb{R}^d)\) solves BDSDE \eqref{1}. 
		 To prove maximality,  let  \((Y', Z')\) be any solution of  BDSDE \eqref{1}. Since \(\xi_m\geq \xi\) and \(f_m\geq f\),
		the  comparison principle  gives \(Y_t'\leq Y_t^m\),
		 \(\forall t\in [0,T]\), a.s.  Passing to the limit \(m\longrightarrow\infty\),    we obtain \(Y'\leq Y\).
	\end{proof}

	\subsection{Reflected BDSDE: existence and uniqueness}\label{sec3.3}

	To establish the existence and uniqueness of reflected BDSDEs, we adopt the same conditions as in \autoref{t3.2}. In contrast to the work of Bahlali et al. \cite{09} and Li and Wei \cite{21}, our main result in this part (\autoref{t3.8}) extends their results  to
	generators that are
	   monotone and of  general growth in \(y\).
	Given that the collection \(\mathbb{F} = \{\mathscr{F}_t; t \in [0,T]\}\) is not a filtration, the usual techniques used in classical reflected BSDEs (see El Karoui et al. \cite{97r}) are not applicable. To overcome this difficulty, we first consider a simple case in which \(f\) is independent of \( z\) and \(g\) is independent of \((y,z)\). Then, using a penalization method and a truncation technique, we construct a sequence of  approximating solutions
  converging to the solution of reflected BDSDE.

	\begin{theorem}\label{t3.8} \sl
		Let the conditions \(\textbf{(A1)}\), \(\textbf{(F1)}\), \(\textbf{(G1)}\), \(\textbf{(G3)}
		\)  and \(\textbf{(S1)}\) be satisfied. Moreover, assume that $f$ is Lipschitz in $z$, i.e.,
		\begin{align*}
			\big| f(t,y,z)-f(t,y,z')\big|
			\leq C |z-z'|.
		\end{align*}
		Then  reflected BDSDE \eqref{r1} admits a unique solution $(Y,Z,K)\in S^2_{\mathbb{F}}([0,T];\mathbb{R})\times L^2_{\mathbb{F}}([0,T];\mathbb{R}^d)\times A^2_{\mathbb{F}}([0,T];\mathbb{R}_+)$.
	\end{theorem}

	Before proving  \autoref{t3.8}, we first present a result that addresses the existence and uniqueness of solutions to RBDSDE \eqref{r1} when the generator \(f\) is independent of \(z\) and the generator \(g\) is independent of \((y,z)\).

	\begin{proposition}\label{p3.10} \sl
		Under the assumptions of \autoref{t3.8},    the following reflected BDSDE	 
		\begin{equation}\label{3.9}
			\left\{	\begin{aligned}
				Y_t&=\xi+\int_t^Tf(s,Y_s)ds+\int_t^Tg(s)dB_s+K_T-K_t-\int_t^TZ_sdW_s, \q\ 0\leq t\leq T,\\
				Y_t&\geq S_t\q\ \hbox{and} \q\ \int_0^T (Y_s-S_s)dK_s=0,
			\end{aligned}
			\right.\end{equation}
		admits a unique solution $(Y,Z,K)\in S^2_{\mathbb{F}}([0,T];\mathbb{R})\times L^2_{\mathbb{F}}([0,T];\mathbb{R}^d)\times A^2_{\mathbb{F}}([0,T];\mathbb{R}_+).$
	\end{proposition}

	\begin{proof}
		{\bf Existence.} 
		We denote by $c>0$ a constant whose value may vary. The following proof is divided into three steps.
		\ms
		
		{\bf\it Step 1}: Assume that, for some constant $c>0,$
		\begin{align*}
			\|\xi\|_{\infty}+\esssup\limits_{(t,\omega)}|f(t,0)|
			+\big\|\sup\limits_{0\leq t\leq T}S_t^+ \big\|_{\infty}
		\leq c.
		\end{align*} 
		
		We notice that
		$(Y,Z,K)\in S_\mathbb{F}^{2}([0,T];\mathbb{R})\times L_\mathbb{F}^2([0,T];\mathbb{R}^d)\times A_\mathbb{F}^{2}([0,T];\mathbb{R}_+)$ is a solution of  reflected BDSDE \eqref{3.9} with parameters $(\xi,f,g)$ and obstacle $S$ if and only if 
		$(Y^c,Z^c,K^c)\in S_\mathbb{F}^{2}([0,T];\mathbb{R})\times L_\mathbb{F}^2([0,T];\mathbb{R}^d)\times A_\mathbb{F}^{2}([0,T];\mathbb{R}_+)$ is a solution of  reflected BDSDE \eqref{3.9} with parameters $(\xi^c,f^c,g^c)$ and obstacle $S^c$, where
		\begin{equation*}
			\left\{
			\begin{aligned}
				&	(Y^c,Z^c,K^c)=(Y-c,Z,K),\\
				&	(\xi^c,f^c,g^c,S^c)=(\xi-c,f(\cdot,y+c),g,S-c).
			\end{aligned} \right.
		\end{equation*}
		For each \(n\in\mathbb{N}\), we consider 
		\begin{align*}
			f^c_n(t,y)=f^c(t,y)+n(y-S^c_t)^-.
		\end{align*}
		It is easy to check that 
		\begin{align*}
			\sup\limits_{0\leq t\leq T}S_t^c \leq 0
			\q\ \hbox{and} \q\ 
			\|\xi^c\|_{\infty}+\sup\limits_{n\geq 1, 0\leq t\leq T}|f_n^c(t,0)|& \leq \|\xi\|_{\infty}+c+\sup\limits_{0\leq t\leq T}|f(t,0)|
			+\varphi(c)\leq 3c+\varphi(c).
		\end{align*}
		Recall Step 1 of \autoref{p3.6},
		for each \(n\in\mathbb{N}\),  	 BDSDE \eqref{1} with parameters $(\xi^c,f^c_n,g^c)$ admits a   solution $(Y^{n,c},Z^{n,c})\in S^2_{\mathbb{F}}([0,T];\mathbb{R})\times L^2_{\mathbb{F}}([0,T];\mathbb{R}^d)$ satisfying
		\begin{align}\label{3.10}
			Y_t^{n,c}=\xi^c+\int_t^T f^c_n(s,Y^{n,c}_s)ds+\int_t^T g^c(s)dB_s-\int_t^T Z_s^{n,c}dW_s,\q\ 0\leq t\leq T.
		\end{align}
		Moreover, the  comparison theorem (\autoref{c3.7}) yields that 
		$  Y^{n,c}_t\leq Y^{n+1,c}_t$.
		Applying It\^o's formula to $|Y^{n,c}_t|^2$ and then taking  expectations on both sides, for any  $\varepsilon>0$, we  derive
		\begin{align}
			\mathbb{E}\left[|Y^{n,c}_t|^2+\int_t^T|Z^{n,c}_s|^2ds\right]&=  \mathbb{E}\left[|\xi^c|^2
			+2\int_t^TY_s^{n,c}f_n^c(s,Y_s^{n,c})ds +\int_t^T |g^c(s)|^2ds\right]
			\nn \\
			&\leq c\left\{1+ 2\mathbb{E}\left[\sup\limits_{0\leq t\leq T}|Y^{n,c}_t|\left(\int_0^T\big|f^c(s,Y_s^{n,c})\big| ds
			+n\int_0^T\left(Y^{n,c}_s-S^c_s\right)^-ds\right)
			\right] \right\}  \nn \\
			&\leq c\left(1+\varepsilon \mathbb{E}\left[\left(n\int_0^T\left(Y^{n,c}_s-S^c_s\right)^-ds\right)^2
			\right]\right), \label{3.13}
		\end{align}
		where Young's inequality has been used. Then,  in view of the following equality,
		\begin{align*}
			n\int_0^T\left(Y^{n,c}_t-S^c_t\right)^-dt
			=Y^{n,c}_0-\xi^c-\int_0^Tf^c(s,Y_s^{n,c})ds-\int_0^Tg^c(s)dB_s+\int_0^TZ^{n,c}_sdW_s.
		\end{align*}
		Using the Burkholder--Davis--Gundy inequality, we have
		\begin{align*}
			\mathbb{E}\left[\left(n\int_0^T\left(Y^{n,c}_s-S^c_s\right)^-ds\right)^2
			\right]
			\leq c\left(1+\mathbb{E}\int_0^T|Z^{n,c}_t|^2dt\right).
		\end{align*}
		Returning to  inequality \eqref{3.13},	and choosing $\varepsilon>0$ small enough, we obtain
		\begin{align*}
			\mathbb{E}\int_0^T|Z^{n,c}_s|^2ds\leq c.
		\end{align*}
		As a result, $\mathbb{E}\left[\big|K^{n,c}_T\big|^2\right]\leq c$,  
		where for any \(t\in [0,T]\),
		\begin{align*}
			K^{n,c}_t=n\int_0^t\left(Y^{n,c}_s-S^c_s\right)^-ds.
		\end{align*}
		Next, we will verify that the convergence of the sequence   \((Y^{n,c},Z^{n,c})\) in $S^2_{\mathbb{F}}([0,T];\mathbb{R})\times L^2_{\mathbb{F}}([0,T];\mathbb{R}^d)$.
		For each $n,m\in \mathbb{N}$, applying It\^o's formula to $|Y^{n,c}_t-Y^{m,c}_t|^2$, and taking expectations on both sides, we get
		\begin{align}\label{k}
			\mathbb{E}\left[\big|Y^{n,c}_t-Y^{m,c}_t\big|^2+\int_t^T\big|Z^{n,c}_s-Z^{m,c}_s\big|^2ds\right]
			&\leq 2\mathbb{E} \int_t^T\left(Y^{n,c}_s-Y^{m,c}_s\right)d\left(K^{n,c}_s-K^{m,c}_s\right)\nn \\
			&\leq 2\mathbb{E}\int_t^T\left(Y^{n,c}_s-S^c_s\right)^-dK^{m,c}_s
			+ 2\mathbb{E}\int_t^T\left(Y^{n,c}_s-S^c_s\right)^-dK^{n,c}_s,
		\end{align}
		where we have used the fact that $f$ is decreasing in $y$. 
		Following,
		we  consider
		\begin{align*}
			\hat{f}^c_n(t,y)=f^c(t,y)+n(S^c_t-y).
		\end{align*}
		For each $n\in \mathbb{N}$,  BDSDE \eqref{1} with parameters $(\xi^c,\hat{f}^c_n,g^c)$ admits a unique solution $(\hat{Y}^{n,c},\hat{Z}^{n,c})\in S^2_{\mathbb{F}}([0,T];\mathbb{R})\times L^2_{\mathbb{F}}([0,T];\mathbb{R}^d)$ satisfying
		\begin{align*}
			\hat{Y}^{n,c}_t=\xi^c+\int_t^T\hat{f}^c_n(s,\hat{Y}^{n,c}_s)ds+\int_t^Tg^c(s)dB_s-\int_t^T\hat{Z}^{n,c}_sdW_s,\q\ 0\leq t\leq T.
		\end{align*}
		For any \((t,y)\in [0,T]\times \mathbb{R}\),
		since \(f_n^c(t,y)\geq \hat{f}_n^c(t,y)\),
		from the  comparison theorem (\autoref{c3.7}),   we have $Y^{n,c}_t\geq \hat{Y}^{n,c}_t$,   \(\forall t\in [0,T]\). Let $v$ be a $(\mathscr{G}_t)_{t\geq 0}$-stopping time, then
		\begin{align*}
			\hat{Y}^{n,c}_v=
			\mathbb{E}\left[e^{-n(T-v)}\xi^c+\int_v^Te^{-n(s-v)}f^c(s,\hat{Y}^{n,c}_s)ds
			+n\int_v^Te^{-n(s-v)}S^c_sds  +\int_v^Te^{-n(s-v)}g^c(s)dB_s~ \Big|~ \mathscr{G}_v \right].
		\end{align*}
		Applying H\"older's inequality and the Burkholder--Davis--Gundy inequality, we deduce that
		\begin{align*}
			\mathbb{E}\left[\int_v^Te^{-n(s-v)}f^c(s,\hat{Y}^{n,c}_s)ds
			+\int_v^Te^{-n(s-v)}g^c(s)dB_s~ \Big|~ \mathscr{G}_v \right]\longrightarrow 0, \q\   as ~ n\longrightarrow \infty.
		\end{align*}
		Consequently, $\hat{Y}^{n,c}_v\longrightarrow \xi^c\mathbf{1}_{\{v=T\}}+S^c_v\mathbf{1}_{\{v<T\}}$ in  mean square. Let \(\bar{Y}_t^c\triangleq \sup\limits_{n\geq 1}Y_t^{n,c}\),
		since \(Y_v^{n,c}\geq \hat{Y}^{n,c}_v\) and \(\hat{Y}^{n,c}_v\longrightarrow S^c_v\) on the set \(\{v<T\}\).
		For any $\dbG$-stopping time \(v\), we get $\bar{Y}_v^c\geq S^c_v$, a.s. Since each process \(Y^{n,c}\) is continuous and \(\mathbb{G}\)-adapted, and is therefore \(\mathbb{G}\)-optional. Hence \(\bar{Y}^c\) is also optional. Since \(S^c\) is optional, the set \(\{(t,\omega);
		\bar{Y}^c_t<S^c_t\}\) is optional.  
		 From the optional section theorem of Dellacherie and Meyer \cite{1975}, this set is evanescent. Consequently,
		we obtain $\bar{Y}_t^c\geq S^c_t$,   \(\forall t\in [0,T]\). Then, 
		since \(Y^{n,c} \uparrow \bar{Y}^c\), \( \big( Y^{n,c} -S^c \big)^-   \downarrow 0\),
		 the sequence \(\left\{\big( Y^{n,c} -S^c \big)^- \right\}_{n\geq 1}\) is continuous, non-increasing  in \(n\), and the process
		\( \big(Y^{1,c}-S^c\big)^-\in S^2_{\mathbb{F}}([0,T];\mathbb{R}).\)
		By Dini's  theorem, we derive
		\begin{align*}
			\mathbb{E}\left[\sup\limits_{0\leq t\leq T}\Big|\big(Y^{n,c}_t-S^c_t\big)^-\Big|^2\right]\longrightarrow 0,\q
			\hbox{as} ~n\longrightarrow \infty.
		\end{align*}
		Returning to inequality \eqref{k}, and combining it with 
		the Burkholder--Davis--Gundy inequality, we conclude that 
		$(Y^{n,c},Z^{n,c})_{n\geq 1}$ is a Cauchy sequence in $S^2_{\mathbb{F}}([0,T];\mathbb{R})\times L^2_{\mathbb{F}}([0,T];\mathbb{R}^d)$.
		We denote the limit process as \((Y^c,Z^c)\).
		Using the Burkholder--Davis--Gundy inequality once more to the following inequality:
		\begin{align*}
			&\	\mathbb{E}\left[\sup\limits_{0\leq t\leq T}\big|K^{n,c}_t-K^{m,c}_t\big|^2\right]
			\leq  4\Bigg\{\mathbb{E}\left[\big|Y^{n,c}_0-Y^{m,c}_0\big|^2\right]  
			+\mathbb{E}\left[\sup\limits_{0\leq t\leq T}\big|Y^{n,c}_t-Y^{m,c}_t\big|^2\right] \\
			&\q\	+ \mathbb{E}\left[\left(\int_0^T\big| f^c(s,Y_s^{n,c})-f^c(s,Y_s^{m,c})\big| ds\right)^2\right]
			+\mathbb{E}\left[\sup\limits_{0\leq t\leq T}\Big|\int_0^t\left(Z^{n,c}_s-Z^{m,c}_s\right)dW_s \Big|^2\right]\Bigg\},
		\end{align*}
		which implies that   the sequence \((K^{n,c})_{n\geq 1}\) converges to a limit process \(K^c\)  in \( S^2_{\mathbb{F}}([0,T];\mathbb{R})\), i.e.,  
		\begin{align*}
			\mathbb{E}\left[\sup\limits_{0\leq t\leq T}\big|K^{n,c}_t-K^c_t\big|^2\right]\longrightarrow 0,\q \hbox{as} ~n\longrightarrow \infty.
		\end{align*}
		Let $n\longrightarrow\infty$ in \eqref{3.10},
		we get a triple  $(Y^c,Z^c,K^c)\in  S^2_{\mathbb{F}}([0,T];\mathbb{R})\times L^2_{\mathbb{F}}([0,T];\mathbb{R}^d)\times A^2_{\mathbb{F}}([0,T];\mathbb{R}_+)$ satisfies 
		\begin{align*}
			Y^c_t=\xi^c+\int_t^Tf^c(s,Y^c_s)ds+\int_t^Tg^c(s)dB_s+K^c_T-K^c_t-\int_t^T Z^c_s dW_s,\q\ 0\leq t\leq T,
		\end{align*}
		where \(Y_t^c\geq S_t^c\) for any \(t\in [0,T]\). Furthermore,
		since \(Y^{n,c}\longrightarrow Y^c\) in \(S^2_{\mathbb{F}}([0,T];\mathbb{R})\) and \(K^{n,c}\longrightarrow K^c\) uniformly in probability, we  have the following convergence:
		\begin{align*}
			\lim\limits_{n\rightarrow \infty}\int_0^T (Y_s^{n,c}-S^c_s)dK_s^{n,c}=
			\int_0^T(Y^c_s-S^c_s)dK^c_s,
		\end{align*}
		this yields
		$	\int_0^T (Y^c_s-S^c_s)dK^c_s=0,$ \(\mathbb{P}\)-a.s. Therefore,  the solution \((Y,Z,K)\in S^2_{\mathbb{F}}([0,T];\mathbb{R})\times L^2_{\mathbb{F}}([0,T];\mathbb{R}^d)\times A^2_{\mathbb{F}}([0,T];\mathbb{R}_+)\) solves reflected BDSDE \eqref{3.9}.

		\ms
		
		{\bf\it Step 2}: Assume that \(S^+\in S^2_{\mathbb{F}}([0,T];\mathbb{R})\),  $\mathbb{E}\left[\sup\limits_{0\leq t\leq T} \varphi^2(S_t^+)\right]<\infty$ and 
		\begin{align*}
			\|\xi\|_{\infty}+\esssup\limits_{(t,\omega)}|f(t,0)|
		\leq c.
		\end{align*}

		Let $S_t^n=S_t \wedge n$,	then \(S_t^n  \uparrow S\). Following Step 1,
		for each $n\in \mathbb{N}$,   there exists a solution $(Y^n,Z^n,K^n)\in S^2_{\mathbb{F}}([0,T];\mathbb{R})\times L^2_{\mathbb{F}}([0,T];\mathbb{R}^d)\times A^2_{\mathbb{F}}([0,T];\mathbb{R}_+)$ satisfying
		\begin{equation}\label{3.14}
			\left\{
			\begin{aligned}
				Y^n_t&=\xi+\int_t^T f(s,Y^n_s)ds+\int_t^Tg(s)dB_s+K^n_T-K^n_t-\int_t^T Z^n_sdW_s,\q\ 0\leq t\leq T,\\
				Y^n_t&\geq S^n_t\q\ \hbox{and}\q\ \int_0^T (Y^n_s-S^n_s)dK^n_s=0.
			\end{aligned}\right.
		\end{equation}
		Since for any \(t\in [0,T]\), $S_t^n\leq S_t^{n+1}$, it follows from \autoref{c3.11} that $ Y_t^n\leq Y_t^{n+1}$.
		As in Step 1, applying It\^o's formula, the dominated convergence theorem and the Burkholder--Davis--Gundy inequality,  we obtain a limiting process $(Y,Z,K)\in S^2_{\mathbb{F}}([0,T];\mathbb{R})\times L^2_{\mathbb{F}}([0,T];\mathbb{R}^d)\times A^2_{\mathbb{F}}([0,T];\mathbb{R}_+)$ satisfying \eqref{3.9}. Additionally, we can establish that for any \(t\in [0,T]\), \( Y_t \geq S_t \), which follows  from \( Y^n_t \geq S^n_t \).
		Furthermore, from dominated convergence theorem, \( S^n\) converges to \( S\) in   \( S^2_{\mathbb{F}}([0,T];\mathbb{R}) \) as \( n \longrightarrow \infty \). 
	 \( (Y^n_t, K^n_t) \) converges uniformly in probability to \( (Y_t, K_t) \)  on \( t\in[0,T] \), 
		the measure \( dK^n \) converges weakly in probability to \( dK \), moreover, \(\sup\limits_{n}\mathbb{E}\left[(K_T^n)^2 \right]<\infty\).  Hence,
		\begin{align*}
			\mathbb{E}\int_0^T(Y^n_s-S^n_s)dK^n_s\longrightarrow \mathbb{E} \int_0^T(Y_s-S_s)dK_s.
		\end{align*}
		Since   $Y_t\geq S_t$ for any \(t\in [0,T]\) and \(K\) is non-decreasing,
		$\int_0^T(Y_s-S_s)dK_s\geq 0$, a.s. The expectation of this non-negative random variable is zero. Therefore,
	 $\int_0^T(Y_s-S_s)dK_s=0$, $ \dbP $-a.s.
		
		\ms
		{\bf\it Step 3}: The general case where  $\xi\in L^2_{\mathscr{F}_T}(\Omega;\mathbb{R})$, $f(\cdot,0)\in L^2_{\mathbb{F}}([0,T];\mathbb{R})$, \(S^+\in S^2_{\mathbb{F}}([0,T];\mathbb{R})\)  and  $\mathbb{E}\left[\sup\limits_{0\leq t\leq T} \varphi^2(S_t^+)\right]<\infty$.

		Let us set
		\begin{align*}
			\xi_n=\rho_n(\xi),\q\ S_t^n=S_t \wedge n, \q\ f_n(t,y)=f(t,y)-f(t,0)+\rho_n(f(t,0)),
		\end{align*}
		where $\rho_n$ is defined in \eqref{3.4},
		\(S_T^n\leq \xi_n\) and  \((S^n)^+\leq n\).
		It follows from Step 2 that,
		for each $n\in\mathbb{N}$, reflected BDSDE \eqref{3.9} with parameter $(\xi_n,f_n,g)$ and obstacle $S^n$ admits a solution $(Y^n,Z^n,K^n)\in S^2_{\mathbb{F}}([0,T];\mathbb{R})\times L^2_{\mathbb{F}}([0,T];\mathbb{R}^d)\times A^2_{\mathbb{F}}([0,T];\mathbb{R}_+)$. 
		Since
		\begin{align*}
			\mathbb{E}\left[|\xi_n-\xi|^2\right]\longrightarrow 0\q\ \hbox{and}  \q\
			\mathbb{E}\int_{0}^{T}\big|f_n(t,0)-f(t,0)\big|^2dt\longrightarrow 0,\q\ \hbox{as} ~n\longrightarrow \infty.
		\end{align*}
		Using the same methodology as in Step 1, we can show that $(Y^n,Z^n,K^n)\longrightarrow (Y,Z,K)$ in $S^2_{\mathbb{F}}([0,T];\mathbb{R})\times L^2_{\mathbb{F}}([0,T];\mathbb{R}^d)\times S^2_{\mathbb{F}}([0,T];\mathbb{R}_+)$. And the limit process $(Y,Z,K)\in S^2_{\mathbb{F}}([0,T];\mathbb{R})\times L^2_{\mathbb{F}}([0,T];\mathbb{R}^d)\times A^2_{\mathbb{F}}([0,T];\mathbb{R}_+)$ is a solution of reflected BDSDE \eqref{3.9}.
		\ms
		
		{\bf Uniqueness.} If there exist two solutions $(Y^1,Z^1,K^1), (Y^2,Z^2,K^2)\in S^2_{\mathbb{F}}([0,T];\mathbb{R})\times L^2_{\mathbb{F}}([0,T];\mathbb{R}^d)\times A^2_{\mathbb{F}}([0,T];\mathbb{R}_+)$, which solve reflected BDSDE \eqref{3.9}, one can apply It\^o's formula
		to $|Y_t^1-Y_t^2|^2$,  this gives 
		\begin{align*}
			\big|Y_t^1-Y_t^2\big|^2=&\ 2\int_t^T(Y_s^1-Y_s^2)\big(f(s,Y_s^1)-f(s,Y_s^2)\big)ds
			-2\int_t^T(Y_s^1-Y_s^2)(Z_s^1-Z_s^2)dW_s\\
			&	+2\int_t^T(Y_s^1-Y_s^2)d(K_s^1-K_s^2)
			-\int_t^T\big|Z_s^1-Z_s^2\big|^2ds.
		\end{align*}
		Note that $f$ is monotone in \(y\), and
		\begin{align*}
			\int_t^T (Y_s^1-Y_s^2)d(K_s^1-K_s^2)=&\ \int_t^T (Y_s^1-S_s)dK_s^1+\int_t^T(Y_s^2-S_s)dK_s^2\\
			&-\int_t^T(Y_s^1-S_s)dK_s^2-\int_t^T(Y_s^2-S_s)dK_s^1\leq0.
		\end{align*} 
		Taking  expectations on both sides we obtain
		\begin{align*}
			\mathbb{E}\left[\big|Y_t^1-Y_t^2\big|^2\right]+\mathbb{E}\int_t^T\big|Z_s^1-Z_s^2\big|^2ds
			\leq 0.
		\end{align*}
		This completes the proof.
	\end{proof}

	Now we give the proof of the main result in this part.

	\begin{proof}[Proof of \autoref{t3.8}]
		The uniqueness can be established similarly to \autoref{p3.10}. Here, we will focus on proving the existence.
		Define	$(Y^0,Z^0)=(0,0)$,  then,
		from \autoref{p3.10},  for each \(n\in \mathbb{N}\), the sequence \((Y^n,Z^n,K^n)\in S^2_{\mathbb{F}}([0,T];\mathbb{R})\times L^2_{\mathbb{F}}([0,T];\mathbb{R}^d)\times A^2_{\mathbb{F}}([0,T];\mathbb{R}_+)\) exists and   satisfies the following reflected BDSDE:
		\begin{equation*}
			\left\{	\begin{aligned}
				Y^n_t&=\xi+\int_t^Tf(s,Y^n_s,Z^{n-1}_s)ds+\int_t^Tg(s,Y^{n-1}_s,Z_s^{n-1})dB_s+K^n_T-K^n_t-\int_t^TZ^n_sdW_s, \q\ 0\leq t\leq T,\\
				Y^n_t&\geq S_t\q\ \hbox{and} \q\ \int_0^T (Y^n_s-S_s)dK^n_s=0.
			\end{aligned}
			\right.\end{equation*}
		For convenience, let \(n\geq 2\), we set
		\begin{align*}
			\triangle Y^{(n)}=Y^n-Y^{n-1},\q\ \triangle Z^{(n)}=Z^n-Z^{n-1},\q\ 
			\triangle K^{(n)}=K^n-K^{n-1}.
		\end{align*} 
		For any $\gamma\in \mathbb{R}$, applying It\^o's formula to $e^{\gamma t}\big|\triangle Y^{(n)}_t\big|^2$, we have
		\begin{align*}
			&\	e^{\gamma t}\big|\triangle Y^{(n)}_t\big|^2+\int_t^Te^{\gamma s}\left(\gamma \big|\triangle Y^{(n)}_s\big|^2+\big|\triangle Z^{(n)}_s\big|^2\right)ds\\
			& =\int_t^T 2e^{\gamma s} \triangle Y^{(n)}_s\Big
			(f(s,Y^n_s,Z^{n-1}_s)-f(s,Y^{n-1}_s,Z^{n-2}_s)\Big)ds+2\int_t^T e^{\gamma s}\triangle Y^{(n)}_sd \triangle K^{(n)}_s\\
			&	+\int_t^T 2e^{\gamma s} \triangle Y^{(n)}_s\Big(g(s,Y^{n-1}_s,Z^{n-1}_s)-g(s,Y^{n-2}_s,Z^{n-2}_s)\Big)dB_s\\
			& -\int_t^T  2e^{\gamma s} \triangle Y^{(n)}_s \triangle Z^{(n)}_s dW_s
			+\int_t^T e^{\gamma s} \Big |g(s,Y^{n-1}_s,Z^{n-1}_s)-g(s,Y^{n-2}_s,Z^{n-2}_s)\Big|^2 ds.
		\end{align*}
		By the Skorokhod condition and the fact that 
		\(Y^n\geq S\), \(Y^{n-1}\geq S\),  we get
		 $\int_t^T e^{\gamma s}\triangle Y^{(n)}_s d\triangle K^{(n)}_s \leq 0$, and  then taking expectations on both sides, we obtain
		\begin{align*}
			&\	\mathbb{E}\int_t^Te^{\gamma s}\left(\gamma \big|\triangle Y^{(n)}_s\big|^2+\big|\triangle Z^{(n)}_s\big|^2\right)ds\\
			&	\leq	  \frac{2C^2}{1-\alpha}\mathbb{E}\int_t^T e^{\gamma s}\big|\triangle Y^{(n)}_s\big|^2ds
			+\mathbb{E}\int_t^Te^{\gamma s}\left(C\big|\triangle Y^{(n-1)}_s\big|^2+ \frac{1+\alpha}{2} 
			\big|\triangle Z^{(n-1)}_s\big|^2 \right)ds.
		\end{align*}
		Choosing  $\gamma=\frac{2C}{1+\alpha}+ \frac{2C^2}{1-\alpha}$.  By iteratively applying the above inequality,
		\begin{align*}
			\mathbb{E}\int_t^Te^{\gamma s}\left(\frac{2C}{1+\alpha} \big|\triangle Y^{(n)}_s\big|^2+\big|\triangle Z^{(n)}_s\big|^2\right)ds
			\leq \left(\frac{1+\alpha}{2}\right)^n
			\mathbb{E}\int_t^Te^{\gamma s}\left(\frac{2C}{1+\alpha} | Y^{1}_s|^2+|Z^{1}_s|^2\right)ds.
		\end{align*}
		Since \(\frac{1+\alpha}{2}<1\), the sequence \((Y^n,Z^n)\) converges in  \( L^2_{\mathbb{F}}([0,T];\mathbb{R})\times L^2_{\mathbb{F}}([0,T];\mathbb{R}^d)\). From the
		Burkholder--Davis--Gundy inequality,
		it is easy to derive the convergence of   \((Y^n)\) in $ S^2_{\mathbb{F}}([0,T];\mathbb{R})$.
		For the convergence of \((K^n)\), we note that
		\begin{align*}
			K^n_t=Y^n_0-Y^n_t-\int_0^tf(s,Y^n_s,Z^{n-1}_s)ds -\int_0^tg(s,Y^{n-1}_s,Z^{n-1}_s)dB_s +\int_0^tZ^{n}_sdW_s,\q\ 0\leq t\leq T.
		\end{align*}
		Now we  know  \((Y^n)\) converges in $ S^2_{\mathbb{F}}([0,T];\mathbb{R})$ and  \((Z^n)\)
		converges in \(L^2_{\mathbb{F}}([0,T];\mathbb{R}^d)\).
		By applying the Burkholder--Davis--Gundy inequality once more,  we obtain
		\((K^n)\) converges in \(S^2_{\mathbb{F}}([0,T];\mathbb{R})\), i.e.,
		\begin{align*}
			\mathbb{E}\left[
			\sup\limits_{0\leq t\leq T}\big|\triangle K^{(n)}_t\big|^2\right]\longrightarrow 0,
			\q\ \hbox{as}~ n\longrightarrow \infty.
		\end{align*}
		Since \(K\) is the uniform limit of continuous non-decreasing sequences \((K^n)\), it is therefore continuous and non-decreasing, and \(K_0=0\).
	We derive a limiting process
		\((Y,Z,K)\in S^2_{\mathbb{F}}([0,T];\mathbb{R})\times L^2_{\mathbb{F}}([0,T];\mathbb{R}^d)\times A^2_{\mathbb{F}}([0,T];\mathbb{R}_+)\). 
			Since   \(Y^n\longrightarrow Y\) and \(K^n\longrightarrow K\)   in $ S^2_{\mathbb{F}}([0,T];\mathbb{R})$.  After passing to a subsequence, we may assume that both convergences are uniform in \(t\), almost surely.
		Moreover,  \((K^n)\)
		is non-decreasing, the uniform convergence entails the weak convergence of the measures \(dK^n\).
	 Consequently,
	 \begin{align*}
	  \Big|  \int_0^T (Y^n_s-S_s)dK^n_s- \int_0^T (Y_s-S_s)dK_s\Big|
	  \leq \sup\limits_{0
	  	\leq s\leq T} |Y^n_s-Y_s|\ K_T^n+
	  	  \Big|  \int_0^T (Y_s-S_s)dK^n_s- \int_0^T (Y_s-S_s)dK_s\Big|.
	 \end{align*}
	 The first term on the right-hand side tends to zero; the second term tends to zero by the continuity of \(Y-S\) and the weak convergence 
	of \(dK^n\).
	 Hence the Skorokhod condition holds for the limit:
		\begin{align*}
			Y_t\geq S_t, \forall t\in [0,T]\q\  \hbox{and} \q\  \int_0^T (Y_s-S_s)dK_s=0.
		\end{align*}
		We  conclude that  \((Y,Z,K)\in S^2_{\mathbb{F}}([0,T];\mathbb{R})\times L^2_{\mathbb{F}}([0,T];\mathbb{R}^d)\times A^2_{\mathbb{F}}([0,T];\mathbb{R}_+)\)  is a solution of 
		reflected BDSDE \eqref{r1}.
	\end{proof}

	As a result of \autoref{t3.8}, we have the following comparison theorem.

	\begin{proposition}[Comparison]\label{c3.11}\sl
			For \(i=1,2\), assume that the parameters \((\xi^i,f^i,g^i)\) satisfy
		\textbf{(A1)}, \textbf{(F1)}, \textbf{(G1)}  and \textbf{(G3)}  with common structural constants, and that \(f^i\) is Lipschitz continuous in \(z\). Assume additionally that \(S^i\) satisfies 
		\textbf{(S1)} and \(S_T^i\leq \xi^i\), \(i=1,2\).
		 Let $(Y^i,Z^i,K^i)\in S^2_{\mathbb{F}}([0,T];\mathbb{R})\times L^2_{\mathbb{F}}([0,T];\mathbb{R}^d)\times A^2_{\mathbb{F}}([0,T];\mathbb{R}_+)$  be the unique solution of the following reflected BDSDE:
		\begin{equation*}
			\left\{
			\begin{aligned}
				&	Y^i_t=\xi^i+\int_{t}^{T}f^i(s,Y^i_s,Z^i_s)ds+\int_{t}^{T}g^i(s,Y^i_s,Z^i_s)dB_s+K^i_T-K^i_t-\int_t^TZ^i_sdW_s,\q\ 0\leq t\leq T,\\
				&	Y_t^i\geq S_t^i\q\ \hbox{and} \q\ \int_0^T (Y_s^i-S_s^i)dK_s^i=0.
			\end{aligned}\right.
		\end{equation*}
		If, for any \(t\in [0,T]\),
		\begin{align*}
			\xi^1\leq \xi^2,\q\ S^1_t\leq S_t^2,\q\ f^1(t,Y^2_t,Z_t^2)\leq f^2(t,Y^2_t,Z_t^2), \q\ g^1(t,Y^2_t,Z_t^2)= g^2(t,Y^2_t,Z_t^2),  \q \hbox{a.s.}, 
		\end{align*}
		\begin{align*}
			\Big( \hbox{alternatively,}\q~ 
			f^1(t,Y^1_t,Z_t^1)\leq f^2(t,Y^1_t,Z_t^1), \q~ g^1(t,Y^1_t,Z_t^1)= g^2(t,Y^1_t,Z_t^1), \q \hbox{a.s.}, \Big)
		\end{align*}
		then we have $Y_t^1\leq Y_t^2$, \(\forall t\in [0,T]\),  $ \dbP $-a.s.
	\end{proposition}

	\begin{proof}
		Applying It\^o's formula to $\phi_{\varepsilon}(Y_t^1-Y_t^2)$, where the function $\phi_{\varepsilon}$ is defined in \eqref{itoy}.  Let $\varepsilon\longrightarrow 0$,  we have
		\begin{align*}
			&\left((Y_t^1-Y_t^2)^+\right)^2
			=\left((\xi^1-\xi^2)^+\right)^2
			+2\int_t^T(Y^1_s-Y^2_s)^+\left(f^1(s,Y^1_s,Z^1_s)-f^2(s,Y_s^2,Z_s^2)\right)ds\\
			&\q	+2\int_t^T(Y^1_s-Y^2_s)^+\Big(g^1(s,Y^1_s,Z^1_s)-g^2(s,Y_s^2,Z_s^2)\Big)dB_s
			-2\int_t^T(Y^1_s-Y^2_s)^+(Z_s^1-Z_s^2)dW_s\\
			&\q	+\int_t^T\left(\big|g^1(s,Y^1_s,Z^1_s)-g^2(s,Y_s^2,Z_s^2)\big|^2-\big|Z_s^1-Z_s^2\big|^2\right)\mathbf{1}_{\{Y_s^1\geq Y_s^2\}}ds+2\int_t^T(Y^1_s-Y^2_s)^+d(K^1_s-K^2_s).
		\end{align*}
		For any \(t\in [0,T]\),
		since $Y^1_t\geq Y^2_t\geq S^2_t\geq S^1_t$ on the set $\{Y^1\geq Y^2\}$,  and \( dK^i\) is supported on the set  $\{Y^i= S^i\}$,
		we deduce that 
		\begin{align*}
			\int_t^T(Y^1_s-Y^2_s)^+d(K^1_s-K^2_s)
			\leq&\ \int_t^T(Y^1_s-S^1_s)dK^1_s-\int_t^T(Y^2_s-S_s^1)dK^1_s-\int_t^T(Y^1_s-Y^2_s)^+dK_s^2\\
			\leq& -\int_t^T(Y^1_s-Y^2_s)^+dK_s^2\leq 0.
		\end{align*}
		The remaining proof is similar to that in \autoref{c3.7}. Combining this with Gronwall's inequality, we can conclude that
		\begin{align*}
			\mathbb{E}\left[ \left((Y_t^1-Y_t^2)^+\right)^2\right]=0.
		\end{align*}
		Therefore, for any \(t\in [0,T]\), $Y_t^1\leq Y_t^2$, $ \dbP$-a.s.
	\end{proof}

	\subsection{Reflected BDSDE: existence of maximal solutions}\label{sec3.4}
	In this part, using the comparison theorem (\autoref{c3.11}), we approximate \(f\) by a monotone sequence of Lipschitz functions \((f_n)\), thereby establishing the existence of the maximal solutions for reflected BDSDEs.  In contrast to the result of Bahlali et al. \cite[Theorem 3.3]{09}, we consider a weaker condition: the generator is continuous and is of general growth in \(y\).

	\begin{theorem}\label{t3.9} \sl
		Let the conditions \(\textbf{(A1)}\), \(\textbf{(F1)}\), \(\textbf{(G1)}\), \(\textbf{(G2)}\), \(\textbf{(G3)}\) and \(\textbf{(S1)}\) be satisfied.  Then  reflected BDSDE \eqref{r1}
		has a maximal solution $(Y,Z,K)\in S^2_{\mathbb{F}}([0,T];\mathbb{R})\times L^2_{\mathbb{F}}([0,T];\mathbb{R}^d)\times A^2_{\mathbb{F}}([0,T];\mathbb{R}_+)$.
	\end{theorem}

	\begin{proof}  For  each \(n\in\mathbb{N}\) with  $n\geq C$, let
		\begin{align*}
			f_n(t,y,z)=\sup\limits_{q\in \mathbb{Q}^d}\{f(t,y,q)-n|z-q|\}.
		\end{align*}
		It is straightforward to verify that, for each \(n\geq C\), $f_n$ satisfies the assumptions of \autoref{t3.8}. As a result, there exists a   triple
		$(Y^n,Z^n,K^n)\in S^2_{\mathbb{F}}([0,T];\mathbb{R})\times L^2_{\mathbb{F}}([0,T];\mathbb{R}^d)\times A^2_{\mathbb{F}}([0,T];\mathbb{R}_+)$ which is the unique solution of reflected BDSDE \eqref{r1} with parameters $(\xi,f_n,g)$ and obstacle $S$. Next, we divide the proof into two steps,
		and  we point out that 
		we always denote by $c>0$ a constant whose value may vary from line to line.
		\ms
		
		{\bf\it Step 1}: Assume that there exists  a positive constant $c$ such that
		\begin{align*}
			\|\xi\|_{\infty}+\esssup\limits_{(t,\omega)}|f(t,0,0)|
		+\big\|\sup\limits_{0\leq t\leq T}S_t^+ \big\|_{\infty}
		\leq c.
		\end{align*}

	Since the sequence $(f_n)_{n\geq C}$ is non-increasing in $n$, it follows from \autoref{c3.11} that  $Y^{n+1}_t\leq Y_t^n$. On one hand, for \(n\geq n_0\geq C\),
		recall the proof of Step 1 in \autoref{p3.10},  we have $|Y_t^{n_0}|\in S^2_{\mathbb{F}}([0,T];\mathbb{R})$. On the other hand, 
		$$f_n(t,y,z)+n(y-S_t)^-\geq f(t,y,z).$$
		It follows	from \autoref{t3.3},  BDSDE \eqref{1} with parameters $(\xi,f,g)$ admits a maximal solution $(Y^*,Z^*)\in S^2_{\mathbb{F}}([0,T];\mathbb{R})\times L^2_{\mathbb{F}}([0,T];\mathbb{R}^d)$.
		According to \autoref{c3.7}, we also have $Y_t^n\geq Y_t^*$.  
		Therefore,  for any \(t\in [0,T],\) \(Y_t^* \leq Y_t^n\leq Y_t^{n_0}\),   taking the supremum,  
		 we can conclude that
		\begin{align*}
			\sup\limits_{0\leq t\leq T} |Y_t^{n}|\leq \max\{\sup\limits_{0\leq t\leq T}|Y_t^{n_0}|,\sup\limits_{0\leq t\leq T}|Y^*_t|   \}
			\in S^2_{\mathbb{F}}([0,T];\mathbb{R}).
		\end{align*}
		Moreover, 
		for every fixed \((t,\omega)\in [0,T]\times \Omega\),
		the sequence \(f_n(t,\omega,\cdot,\cdot)\) is continuous and decreases pointwise to \(f(t,\omega,\cdot,\cdot)\).
		Applying Dini's theorem, $f_n(t,\omega,y,z)$ converges   to $f(t,\omega,y,z)$ uniformly on the compact sets of 
		$\mathbb{R}\times\mathbb{R}^d$. 
		Let 
		\begin{align*}
H \triangleq \max\{\sup\limits_{0\leq t\leq T}|Y_t^{n_0}|,\sup\limits_{0\leq t\leq T}|Y^*_t|\}\q\ 
\hbox{and} \q\ \eta_t\triangleq |f(t,0,0)|+\varphi(H).
		\end{align*}
		Then 
		\begin{align*}
		 |f_n(t,Y_t^n,Z_t^n)|\leq 
		 \eta_t+C|Z_t^n|.
		\end{align*}
		Repeating the a priori estimates used in  \autoref{p3.10}, with the additional term \(|Z_t^n|\) absorbed by  Young's inequality, we obtain
		\begin{align*}
	 \sup\limits_{n\geq n_0}\mathbb{E}\left[ \int_0^T
	   |Z_s^n|^2ds +(K_T^n)^2\right] <\infty.
		\end{align*}
		For \(n\geq m\),  we have \(Y^n\leq Y^m\). Moreover, by the Skorokhod condition, 
		\begin{align*}
 \int_0^T (Y_s^n-Y_s^m)d(K_s^n-K_s^m)
 =\int_0^T(S_s-Y_s^m)dK_s^n-\int_0^T(Y_s^n-S_s)dK_s^m\leq 0.
		\end{align*}
		Therefore, applying It\^o's formula to \(|Y_t^n-Y_t^m|^2\), 
		the reflection term has the favorable sign,  and the same estimate as in the non-reflected case gives $Z^n$ converges to $Z$ in $L^2_{\mathbb{F}}([0,T];\mathbb{R}^d)$.  
		Moreover, by the local uniform convergence of \(f_n\), the preceding uniform growth estimate, and Vitali's convergence theorem, 
		\begin{align*}
		 f_n(\cdot,Y^n,Z^n)\longrightarrow f(\cdot,Y,Z)
		 \q\ \hbox{in}\ L^2_{\mathbb{F}}([0,T];\mathbb{R}).
		\end{align*}
		The Burkholder--Davis--Gundy inequality then gives  $Y^n$ converges   to $Y$ in $S^2_{\mathbb{F}}([0,T];\mathbb{R})$.
		Finally,  from the following equation:
		\begin{align*}
			K_t^n=Y_0^n-Y_t^n-\int_0^t f_n(s,Y_s^n,Z_s^n)ds
			-\int_0^t g(s,Y_s^n,Z_s^n)dB_s +\int_0^t Z_s^ndW_s,\q\
			0\leq t\leq T,
		\end{align*}
	 we obtain 
		$K^n$ converges to \(K\)  in  $S^2_{\mathbb{F}}([0,T];\mathbb{R})$.
		Since for any \(t\in [0,T]\), $Y^n_t\geq S_t$ and $\int_0^T(Y_s^n-S_s)dK_s^n=0$,  the Skorokhod condition is preserved in the limit. Indeed, since
		\(Y^n\longrightarrow Y\) and \(K^n\longrightarrow K\)
		in \(   S^2_{\mathbb{F}}([0,T];\mathbb{R})\), up to a subsequence both convergences hold uniformly in \(t\),
		\(\mathbb{P}\)-a.s. Moreover, the measures \(dK^n\) converge weakly to \(dK\) on \([0,T]\). Therefore,
		\begin{align*}
			Y_t\geq S_t, \forall t\in [0,T] \q\hbox{and}\q  \int_0^T(Y_s-S_s)dK_s=0.
		\end{align*}
		The triple $(Y,Z,K)\in S^2_{\mathbb{F}}([0,T];\mathbb{R})\times L^2_{\mathbb{F}}([0,T];\mathbb{R}^d) \times A^2_{\mathbb{F}}([0,T];\mathbb{R}_+)$ is a solution of reflected BDSDE \eqref{r1}.
		Next, we prove the solution is maximal. Let $(Y',Z',K')\in S^2_{\mathbb{F}}([0,T];\mathbb{R})\times L^2_{\mathbb{F}}([0,T];\mathbb{R}^d) \times A^2_{\mathbb{F}}([0,T];\mathbb{R}_+)$ be another solution of reflected BDSDE \eqref{r1}, since $f_n$ is Lipschitz in 
		$z$ and decreasing in $y$ with $f\leq f_n$. It follows from \autoref{c3.11} that  $Y'\leq Y^n$. Let $n\longrightarrow \infty$, we get $Y'\leq Y$.
		\ms

		{\bf\it Step 2}: The general case where
		$\xi\in L^2_{\mathscr{F}_T}(\Omega;\mathbb{R})$,
		$f(\cdot,0,0)\in L^2_{\mathbb{F}}([0,T];\mathbb{R})$, \(S^+\in S^2_{\mathbb{F}}([0,T];\mathbb{R})\)  and  $\mathbb{E}\left[\sup\limits_{0\leq t\leq T} \varphi^2(S_t^+)\right]<\infty$.

		Let us set
		\begin{align*}
			&	\xi_n=\rho_n(\xi), \q\   S_t^n=S_t\wedge n,   \q\ 
			f_n(t,y,z)=f(t,y,z)-f(t,0,0)+\rho_n(f(t,0,0)), 
		\end{align*}
		where $\rho_n$ is defined in \eqref{3.4}. 
		The obstacle condition satisfies
	\(S_T^n\leq \xi_n\) and \((S^n)^+\leq n\).
		It follows from Step 1 that, reflected BDSDE \eqref{r1} with parameters $(\xi_n,f_n,g)$ and obstacle \(S^n\) has a maximal solution
		$(Y^n,Z^n,K^n)\in S^2_{\mathbb{F}}([0,T];\mathbb{R})\times L^2_{\mathbb{F}}([0,T];\mathbb{R}^d) \times A^2_{\mathbb{F}}([0,T];\mathbb{R}_+)$ satisfying
		\begin{align}\label{3.16}
			Y_t^n=\xi_n+\int_t^Tf_n(s,Y_s^n,Z_s^n)ds+\int_t^Tg(s,Y_s^n,Z_s^n)dB_s+K^n_T-K^n_t-\int_t^TZ_s^ndW_s,\q\ 0\leq t\leq T.
		\end{align}
		Since
		\begin{align*}
		 \sup\limits_{0\leq t\leq T}|S_t^n-S_t|
		 =\sup\limits_{0\leq t\leq T}\big(S_t-n \big)^+
		 \leq \sup\limits_{0\leq t\leq T}\big(S^+_t-n \big)^+.
		\end{align*}
		The right-hand side converges to zero almost surely and is dominated by \(\sup\limits_{0\leq t\leq T}S_t^+\).
		Hence, $S^n\longrightarrow S$ in $S^2_{\mathbb{F}}([0,T];\mathbb R)$. Moreover,
		$\xi_n\longrightarrow \xi$ in $L^2_{\mathscr{F}_T}(\Omega; \mathbb{R})$ and $f_n(\cdot,0,0)\longrightarrow f(\cdot,0,0)$ in
		$L^2_{\mathbb{F}}([0,T];\mathbb R)$,  the same estimate as in Step 1 yields that
		\begin{align*}
		 \mathbb E\Bigg[\sup_{0\leq t\leq T}\big|Y_t^n-Y_t^m\big|^2
		 +\int_0^T\big |Z_t^n-Z_t^m\big|^2dt
		 +\sup_{0\leq t\leq T}\big|K_t^n-K_t^m\big|^2\Bigg]\longrightarrow 0.
		\end{align*}
		Thus,  there exists a triple \((Y,Z,K)\) such that
		\(Y^n\longrightarrow Y\) in \( S^2_{\mathbb{F}}([0,T];\mathbb{R})\), \(Z^n\longrightarrow Z\) in \(L^2_{\mathbb{F}}([0,T];\mathbb{R}^d)\), and 
		\(K^n\longrightarrow K\) in \(S^2_{\mathbb{F}}([0,T];\mathbb{R})\).
		Passing to the limit in \eqref{3.16} gives the desired reflected BDSDE.
		Moreover, since $Y^n\geq S^n$ and $S^n\longrightarrow S$ in $S^2_{\mathbb{F}}([0,T];\mathbb R)$, we have $Y\geq S$.
		Finally, the Skorokhod condition follows from
		\begin{align*}
	&\	 \int_0^T(Y_s^n-S_s^n)dK_s^n-\int_0^T(Y_s-S_s)dK_s
		 =\int_0^T\big( (Y_s^n-S_s^n) - (Y_s-S_s) \big)dK_s^n
		+\int_0^T (Y_s-S_s) d(K_s^n-K_s),
		\end{align*}
		where the first term converges by the uniform convergence in probability of
		$(Y^n,S^n)$ and the uniform $L^2$-boundedness of $(K_T^n)$; the second term converges by the weak convergence of the measure \(dK^n\).
	\end{proof}

	\section{Linear growth in $y$ and
		quadratic growth in $z$}\label{section4}
	
	In this section, we study the existence and uniqueness for BDSDEs and reflected BDSDEs when the generator
	$f$ is of quadratic growth with respect to $z$, and the terminal value \(\xi\) is bounded. For the sake of clarity, we will divide our results into two subsections.

	\subsection{BDSDE: existence and uniqueness}\label{sec4.1}
	
	In this part, we study the existence and uniqueness of BDSDEs when the coefficient $f$ is of linear growth with respect to $y$ and quadratic growth with respect to $z$.
	Additionally, we consider the condition   \(\textbf{(G2)}\) on  the coefficient $g$, which is weaker than the condition     proposed by   Hu, Wen, and Xiong \cite{wen1}.
	We begin by generalizing the result of a priori estimate for the solution \((Y, Z)\) of BDSDE  \eqref{1} (see Hu, Wen, and Xiong \cite[Proposition 3.7]{wen1}). This step is crucial for proving the  existence of solutions.

	\begin{lemma}[A priori estimate]\label{l4.1}\sl
		Let the conditions \(\textbf{(A2)}\) and \(\textbf{(G2)}\) be satisfied. Moreover, let \( C >0 \) be a constant, \(a:[0,T] \longmapsto \mathbb{R}_+ \) and \( b: [0,T]\longmapsto \mathbb{R}_+\) be two non-negative functions, such that for all \((t,y,z)\in [0,T]\times\mathbb{R}\times\mathbb{R}^d\), 
		\begin{align*}
			|f(t,y,z)|\leq a_t+b_t|y|+C|z|^2.
		\end{align*}
		If \((Y,Z)\in S^\infty_{\mathbb{F}}([0,T];\mathbb{R})\times L^2_{\mathbb{F}}([0,T];\mathbb{R}^d)\) is a solution of BDSDE \eqref{1},
		then for any \(t\in[0,T]\),
		\begin{align}\label{y4.17}
			Y_t\leq  \|\xi\|_{\infty}\exp\left\{\int_t^T b_sds\right\}+\int_t^T a_s \exp\left\{\int_t^s b_rdr\right\}ds+\frac{4C}{1-\alpha}\int_t^T|\lambda(s)|^2\exp\left\{\int_t^s b_rdr\right\}ds,\q\ \hbox{a.s.}
		\end{align}
		respectively,
		\begin{align}\label{y4.18}
			\Big(	Y_t\geq   -\|\xi\|_{\infty}\exp\left\{\int_t^T b_sds\right\}-\int_t^T a_s \exp\left\{\int_t^s b_rdr\right\}ds-\frac{4C}{1-\alpha}\int_t^T|\lambda(s)|^2\exp\left\{\int_t^s b_rdr\right\}ds,\q\ \hbox{a.s.} \Big)
		\end{align}
		Furthermore, there exists a constant $\hat{K}>0$ which depends on $\alpha,C,T, \|a\|_{L^1([0,T];\mathbb{R}_+)}, \|b\|_{L^1([0,T];\mathbb{R}_+)}$ and $\|Y\|_{S_{\mathbb{F}}^\infty([0,T];\mathbb{R})}$, such that
		\begin{align}\label{z4.19}
			\mathbb{E}\int_0^T|Z_s|^2ds\leq\hat{ K}.
		\end{align}
	\end{lemma}
	\begin{proof}
		Consider a linear ordinary differential equation, 
		\begin{align}\label{x}
			X_t=\|\xi\|_{\infty}+\int_t^T\left(b_s X_s+a_s+\frac{4C}{1-\alpha}|\lambda(s)|^2\right)ds.
		\end{align}
		For $M=\|Y\|_{S_{\mathbb{F}}^\infty([0,T];\mathbb{R})}+\|X\|_{S_{\mathbb{F}}^\infty([0,T];\mathbb{R})}$, define the function $\Phi$ on the interval $[-M,M]$ by
		\begin{equation}\label{Phi}
			\Phi(u)=\left\{
			\begin{aligned}
				\frac{1-\alpha}{4C}(e^{\frac{4C}{1-\alpha}u}-1),\qquad& u\in[0,M], \\
				\quad 0,\qquad  	&u\in\left[-M,0\right).
			\end{aligned}
			\right.
		\end{equation}
		Let \((\Phi_\varepsilon)_{\varepsilon>0}\subset C^2(\mathbb R)\) be a sequence of non-negative convex functions, such that 
		\begin{align*}
			\Phi_\varepsilon(u)\longrightarrow \Phi(u),\q\
			\Phi_\varepsilon'(u)\longrightarrow \Phi'(u),\q\
			\Phi_\varepsilon''(u)\longrightarrow \Phi''(u), \q\ u\neq 0.
		\end{align*}
		Moreover,   
		\begin{align*}
			u\Phi_\varepsilon'(u)\leq \Phi_\varepsilon(u)\left(1+\frac{4C M}{1-\alpha}\right)+r_\varepsilon,\q\
			\Phi_\varepsilon'(u)C-\frac{1-\alpha}{2}\Phi_\varepsilon''(u)\leq r_\varepsilon,\q\
			\frac{1}{2}\Phi_\varepsilon''(u)-\frac{4C}{1-\alpha}\Phi_\varepsilon'(u)\leq r_\varepsilon,
		\end{align*}
		where \(r_\varepsilon\longrightarrow 0\) uniformly on \([-M,M]\).
		Applying  It\^o's formula to $\Phi_\varepsilon(Y_t-X_t)$,  taking expectations,
and then letting \(\varepsilon\longrightarrow 0\), using monotone convergence theorem and the dominated convergence theorem, we obtain
		\begin{align*}
		\mathbb{E}\left[ 	\Phi(Y_t-X_t) \right]\leq &\ \mathbb{E}\left[ \Phi(Y_T-X_T) \right]+\mathbb{E}\int_{t}^{T}\Phi'(Y_s-X_s)\left(f(s,Y_s,Z_s)-b_s X_s-a_s-\frac{4C}{1-\alpha}|\lambda(s)|^2\right)ds\\
			&+\frac{1}{2}\mathbb{E}\int_t^T\Phi''(Y_s-X_s)\Big(|g(s,Y_s,Z_s)|^2-|Z_s|^2\Big)ds.
		\end{align*}
	By the assumptions on $f$ and $g$, we get
		\begin{equation}\label{lemma Phi}
			\begin{aligned}
				\mathbb{E}\left[ 	\Phi(Y_t-X_t) \right]\leq&\ 	\mathbb{E}\left[ \Phi(Y_T-X_T)\right]+\mathbb{E}\int_{t}^{T}b_s\big|\Phi'(Y_s-X_s)\left(Y_s-X_s\right)\big|ds\\
				&+\mathbb{E}\int_t^T\left(\Phi'C-\frac{1-\alpha}{2}\Phi''\right)(Y_s-X_s)|Z_s|^2ds.
			\end{aligned}
		\end{equation}
	From the estimates for  \(\Phi, \Phi'\) and \(\Phi''\), we obtain
		\begin{align*}
			0\leq \mathbb{E}\left[\Phi(Y_t-X_t) \right]
			\leq \int_t^T\widetilde{b}_s \mathbb{E}\left[\Phi(Y_s-X_s)\right]ds,
		\end{align*}
		where $\widetilde{b}_s=b_s \big(1+\frac{4C M}{1-\alpha}\big)$ is a positive, deterministic function.   Gronwall's inequality then yields that
		$\mathbb{E}\left[\Phi(Y_t-X_t)\right]=0$, which  implies that, for all $t\in[0,T]$, $\Phi(Y_t-X_t)=0$, $\mathbb{P}$-a.s. Therefore, we get
		$Y_t\leq X_t,$ for any \(t\in [0,T]\).   The inequality \eqref{y4.18} can be proved similarly.\\
		Now we prove the inequality \eqref{z4.19}. For any $u\in[-M,M]$,  define
		\begin{align}\label{Phi1}
			\widetilde{\Phi}(u)=\frac{1-\alpha}{4C}\big(e^{\frac{4C}{1-\alpha}(u+M+1)}-1\big).
		\end{align} 
Using the same  convex smoothing technique   as above, we    approximate \(\widetilde{\Phi}\) by a smooth convex function such that 
		\begin{align*}
			\widetilde{\Phi}'(u)C-\frac{1-\alpha}{2}\widetilde{\Phi}''(u)=-C\widetilde{\Phi}'(u)\leq -C\q\
			\hbox{and} \q\ 
			\frac{1}{2}\widetilde{\Phi}''(u)-\frac{4C}{1-\alpha}\widetilde{\Phi}'(u)\leq 0.
		\end{align*}
		Replacing the function $\Phi$ with $\widetilde{\Phi}$ in inequality \eqref{lemma Phi}, and letting $t=0$, we obtain
		\begin{align*}
			C\mathbb{E}\left[\int_0^T|Z_s|^2ds\right]
			\leq \widetilde{\Phi}(M)+ e^{\frac{4C}{1-\alpha}(1+2M)}\|b\|_{L^1([0,T];\mathbb{R}_+)}.
		\end{align*}
		This completes the proof.
	\end{proof}

	To ensure uniqueness, we adopt an assumption similar to that of Hu, Wen, and Xiong \cite{wen1}.
	These strengthened assumptions state that the coefficient $f$ is locally Lipschitz continuous and is of  quadratic growth in $z$ in a strong sense. Additionally, the coefficient \(g\) has  bounded partial derivatives with respect to both \(y\) and \(z\).
	Specifically, we will strengthen assumptions \(\textbf{(F2)}\) and \(\textbf{(G1)}\) to the following conditions:

	\begin{description}
		\item[(F4)] There exist two positive constants \(\epsilon\), \(C\),	and
		three functions $k_1, k_2, k_{\epsilon} :[0,T] \longmapsto \mathbb{R}$, such that for all $t\in[0,T], y\in[-M,M]$ and $z\in\mathbb{R}^d$, \(\mathbb{P}\)-a.s., $f$ satisfies the following:\\
		(i) $|f(t,y,z)|\leq k_1(t)+C|z|^2,$\\
		(ii) $\frac{\partial f}{\partial y}(t,y,z)\leq k_{\epsilon}(t)+\epsilon |z|^2,$\\
		(iii) $\Big|\frac{\partial f}{\partial z}(t,y,z)\Big|\leq k_2(t)+C|z|.$
		
		\item [(G4)] There exist two constants \(C>0, \alpha \in (0,1)\) such that,  \(\mathbb{P}\)-a.s.,
		g satisfies: 
		\begin{align*}
			\Big|\frac{\partial g}{\partial y}\Big|\leq C\q\ \hbox{and} \q\ \Big|\frac{\partial g}{\partial z}\Big|\leq \alpha.
		\end{align*}
	\end{description}

 The following lemma slightly generalizes the comparison result of Hu, Wen, and Xiong \cite[Theorem 4.4]{wen1} and directly yields uniqueness for BDSDE. The proof is identical to that of the reflected case, so we omit it here.
 
\begin{lemma}\label{qbdsdecom}\sl 
		Let the conditions \(\textbf{(A2)}\), \(\textbf{(G2)}\) and \(\textbf{(G4)}\)  be satisfied, let $(Y^1,Z^1)\in S^\infty_{\mathbb{F}}([0,T];\mathbb{R})\times L^2_{\mathbb{F}}([0,T];\mathbb{R}^d)$ be a solution of   BDSDE \eqref{1} with parameters $(\xi^1,f^1,g)$, $(Y^2,Z^2)\in S^\infty_{\mathbb{F}}([0,T];\mathbb{R})\times L^2_{\mathbb{F}}([0,T];\mathbb{R}^d)$ be a solution of   BDSDE \eqref{1} with parameters
	$(\xi^2,f^2,g)$.  Assume that
	\(	\xi^1\leq \xi^2,\) \(\mathbb{P}\)-a.s.
	  For all \(t\in[0,T]\), if $f^1(t,Y_t^1,Z_t^1)\leq f^2(t,Y_t^1,Z_t^1)$, $ \dbP$-a.s. and $f^2$ satisfies \(\textbf{(F4)}\) (or if  $f^1(t,Y_t^2,Z_t^2)\leq f^2(t,Y_t^2,Z_t^2)$, $ \dbP $-a.s.
	and  $f^1$ satisfies \(\textbf{(F4)}\)). 
	Then  $Y_t^1\leq Y_t^2$ for all \( t\in [0,T]\),  $  \dbP $-a.s. 
\end{lemma}
	
	Now we give the main result of this part.

	\begin{theorem}\label{tb4.2}\sl
		Let the conditions \(\textbf{(A2)}\), \(\textbf{(F2)}\), \(\textbf{(G1)}\) and \(\textbf{(G2)}\) be satisfied. 
		Then the BDSDE \eqref{1} admits at least one solution $(Y,Z)\in S_\mathbb{F}^{\infty}([0,T];\mathbb{R})\times L_\mathbb{F}^2([0,T];\mathbb{R}^d)$. Moreover, if the conditions \(\textbf{(F4)}\) and \(\textbf{(G4)}\) are satisfied, then the solution is unique.
	\end{theorem}
	
	\begin{proof}
		The uniqueness is a straightforward result of \autoref{qbdsdecom}.
		We provide a brief proof for the existence, which
		will be divided into two steps.
		\ms
		
		{\bf\it Step 1} : $f$ is non-negative.

		For each integer \(n\geq C\), we define the sequence of functions as follows:
		\begin{align*}
			f_n(t,y,z)=\inf\limits_{(p,q)\in\mathbb{Q}\times\mathbb{Q}^d}\{f(t,p,q)+n|y-p|+n|z-q|\},
		\end{align*}
		which is well defined for each $n\geq C$. It follows from the classical existence and uniqueness result for BDSDE (see Pardoux and Peng \cite{94}) that  BDSDE \eqref{1} with parameters $(\xi,f_n,g)$ admits a unique solution $(Y^n,Z^n)$ such that  the sequence $(Y^n)$ is non-decreasing.
		For the Lipschitz approximating BDSDE, the solution is bounded by Hu, Wen, and Xiong's boundedness result, hence  \autoref{l4.1} is applicable, \((Y^n)\) is uniformly bounded.
	Furthermore,   applying the monotone stability theorem of Hu, Wen, and Xiong \cite[Proposition 3.9]{wen1},  there exists a pair of processes $(Y,Z)\in S_\mathbb{F}^{\infty}([0,T];\mathbb{R})\times L_\mathbb{F}^2([0,T];\mathbb{R}^d) $ such that the sequence $(Y^n)_{n\geq C}$ converges uniformly to $Y$ on \([0,T]\) and
		the sequence $(Z^n)_{n\geq C}$ converges to $Z$ in $L_\mathbb{F}^2([0,T];\mathbb{R}^d)$. Passing to the limit in approximating BDSDE, we conclude that  $(Y,Z)\in  S_\mathbb{F}^{\infty}([0,T];\mathbb{R})\times L_\mathbb{F}^2([0,T];\mathbb{R}^d)$ is a solution of  BDSDE \eqref{1}.
		
		\ms
		
		{\bf\it Step 2}: the general case.

		Define
		\begin{align*}
			f_{n,m}(t,y,z)=&\inf\limits_{(p,q)\in\mathbb{Q}\times\mathbb{Q}^d}\{f^+(t,p,q)+n|y-p|+n|z-q|\}\\
			&-\inf\limits_{(p,q)\in\mathbb{Q}\times\mathbb{Q}^d}\{f^-(t,p,q)+m|y-p|+m|z-q|\}.
		\end{align*}
		The generators \(f_{n,m}\) satisfy the same quadratic growth bound as \(f\), uniformly in \(n,m\). Hence,
		the a priori estimate in  \autoref{l4.1} is uniform in \(n,m\).  Since the sequence $(f_{n,m})_{n\geq C}$ is monotone
		and locally uniformly converges to
		\( f_{\infty,m}\triangleq
		f^+(t,y,z)
		-\inf\limits_{(p,q)\in\mathbb Q\times\mathbb Q^d}
		\{f^-(t,p,q)+m|y-p|+m|z-q|\}.
		\) 
		 Applying the monotone stability theorem
		first with $m$ fixed and \(n\longrightarrow \infty\), \(f_{n,m} \uparrow f_{\infty,m}\).
		Then, as $m\longrightarrow \infty$,
			\( f_{\infty,m}  \downarrow f
		\), 
		the same stability argument as Step 1 gives a solution of BDSDE \eqref{1}. We complete the proof.
	\end{proof}

	\subsection{Reflected BDSDE: existence and uniqueness}\label{sec4.3}

	In this part, we study the existence and uniqueness of quadratic BDSDEs with reflection, where the generators meet the same assumptions as stated in \autoref{tb4.2}. To establish the uniqueness, we first demonstrate a comparison theorem (\autoref{t5.2}) based on the structural condition, and then we complete the proof using a change of variables. For  the existence, we can still apply the monotone stability method along with double approximation techniques to reflected BDSDEs.
	
	\begin{theorem}[Comparison]\label{t5.2}\sl
		Let the conditions \(\textbf{(A2)}\), \(\textbf{(G2)}\), \(\textbf{(G4)}\) and \(\textbf{(S2)}\) be satisfied, let $(Y^1,Z^1,K^1)\in S^\infty_{\mathbb{F}}([0,T];\mathbb{R})\times L^2_{\mathbb{F}}([0,T];\mathbb{R}^d)\times A^2_{\mathbb{F}}([0,T];\mathbb{R}_+)$ be a solution of reflected BDSDE \eqref{r1} with parameters $(\xi^1,f^1,g)$ and obstacle process $S^1$, $(Y^2,Z^2,K^2)\in S^\infty_{\mathbb{F}}([0,T];\mathbb{R})\times L^2_{\mathbb{F}}([0,T];\mathbb{R}^d)\times A^2_{\mathbb{F}}([0,T];\mathbb{R}_+)$ be a solution of reflected BDSDE \eqref{r1} with parameters
		$(\xi^2,f^2,g)$ and obstacle process $S^2$. For all \(t\in[0,T]\), if $f^1(t,Y_t^1,Z_t^1)\leq f^2(t,Y_t^1,Z_t^1)$, $ \dbP$-a.s. and $f^2$ satisfies \(\textbf{(F4)}\) (or if  $f^1(t,Y_t^2,Z_t^2)\leq f^2(t,Y_t^2,Z_t^2)$, $ \dbP $-a.s.
		and  $f^1$ satisfies \(\textbf{(F4)}\)). Moreover, assume that,
		for any \(t\in [0,T]\),
		\begin{align*}
			S_t^1\leq S_t^2 \q \hbox{and} \q 
			\xi^1\leq \xi^2, \q\
			\hbox{a.s.}
		\end{align*}
		Then  $Y_t^1\leq Y_t^2,$  \(\forall t\in [0,T]\),  $  \dbP $-a.s. 
	\end{theorem}
	
	\begin{proof}
		Following the approach used by Barles and Murat \cite{95B} for PDEs, we divide the proof into two steps. First, we establish the comparison principle under a structural condition denoted as \(\textbf{(STR)}\). Next,  we complete the proof by applying a change of variable that transforms a
		reflected BDSDE with the coefficient $f$ satisfying \(\textbf{(F4)}\) into a reflected BDSDE with the coefficient $f$ that meets the structural condition \(\textbf{(STR)}\).
		\ms
		
		{\bf\it Step 1}. Assume that $f^2$ satisfies the structural condition \(\textbf{(STR)}\) with a positive constant $a$ and a function $b: [0,T]\longmapsto\mathbb{R}$ such that
		\begin{align}\label{str}
			\frac{\partial f}{\partial y}(t,y,z)+a\Big|\frac{\partial f}{\partial z}
			\Big|^2(t,y,z)
			\leq b(t).
		\end{align}
		
		Applying It\^o's formula to $\left((Y_t^1-Y_t^2)^+\right)^p$ with $p\geq 2$, we get
		\begin{equation}
			\begin{aligned}\label{5.2}
				&	\left((Y_t^1-Y_t^2)^+\right)^p\leq p\int_t^T \left((Y_s^1-Y_s^2)^+\right)^{p-1}\Big(f^1(s,Y_s^1,Z_s^1)-f^2(s,Y_s^2,Z_s^2)\Big)ds\\
				&\q +p\int_t^T \left((Y_s^1-Y_s^2)^+\right)^{p-1}\Big(g(s,Y_s^1,Z_s^1)-g(s,Y_s^2,Z_s^2)\Big)dB_s-p\int_t^T\left((Y_s^1-Y_s^2)^+\right)^{p-1}\left(Z_s^1-Z_s^2\right)dW_s\\
				&\q +\frac{p(p-1)}{2}\int_t^T\left((Y_s^1-Y_s^2)^+\right)^{p-2}\left(\big|g(s,Y_s^1,Z_s^1)-g(s,Y_s^2,Z_s^2)\big|^2-\big|Z_s^1-Z_s^2\big|^2\right)ds,
			\end{aligned}
		\end{equation}
		where we have used the fact that $K^2$ is a non-decreasing process and
		\begin{align*}
			\int_t^T\left((Y_s^1-Y_s^2)^+\right)^{p-1}d\left(K_s^1-K_s^2\right)\leq 0.
		\end{align*}
		In view of the condition on the coefficient $f^2$, we have
		\begin{align*}
			&\ \left((Y_s^1-Y_s^2)^+\right)^{p-1}\Big(f^1(s,Y_s^1,Z_s^1)-f^2(s,Y_s^2,Z_s^2)\Big)\\
			&\leq
			\left((Y_s^1-Y_s^2)^+\right)^{p}\int_0^1\left(\frac{\partial f^2}{\partial y}+a\Big|\frac{\partial f^2}{\partial z}\Big |^2\right)(*)d\lambda+\frac{1}{4a}\left((Y_s^1-Y_s^2)^+\right)^{p-2}\big| Z_s^1-Z_s^2\big|^2\mathbf{1}_{\{Y_s^1\geq Y_s^2\}},
		\end{align*}
		where
		\begin{align*}
			(*)=\big(s,\lambda Y_s^1+(1-\lambda)Y_s^2,\lambda Z_s^1+(1-\lambda)Z_s^2\big).
		\end{align*}
		Since \( g \) has bounded partial derivatives with respect to \( y \) and \( z \), for any \(\varepsilon>0\), we have
		\begin{align*}
			\big| g(s,Y_s^1,Z_s^1)-g(s,Y_s^2,Z_s^2)\big|^2
			\leq (1+\varepsilon ^{-1})C^2 \big|Y_s^1-Y_s^2\big|^2+(1+\varepsilon)\alpha^2 \big|Z_s^1-Z_s^2\big|^2.
		\end{align*}
		Let $\varepsilon =\frac{1-\alpha^2}{2\alpha^2}$ and 
		returning to   inequality  \eqref{5.2},  we obtain
		\begin{align*}
			&\ \mathbb{E}\left[\left((Y_t^1-Y_t^2)^+\right)^p +\frac{p}{4}\left((p-1)(1-\alpha^2)-\frac{1}{a}\right)\int_t^T \left((Y_s^1-Y_s^2)^+\right)^{p-2}\big|Z_s^1-Z_s^2\big|^2ds\right]\\	&\leq
			p\mathbb{E}\int_t^T\left(b(s)+\frac{(p-1)(1+\alpha^2)}{2(1-\alpha^2)}C^2\right)\left((Y_s^1-Y_s^2)^+\right)^pds.
		\end{align*}
		Now  we choose $p$ large enough such that 
		\begin{align*}
			(p-1)(1-\alpha^2)-\frac{1}{a}\geq 0.
		\end{align*}
		Using	Gronwall's inequality, we get for any \(t\in [0,T]\),
		$
		Y_t^1\leq Y_t^2,
		$ \(\mathbb{P}\)-a.s.
		\ms
		
		{\bf\it Step 2}. Assume that $f^2$ satisfies \(\textbf{(F4)}\).

		Let $M\in\mathbb{R}$ such that $M\geq \|Y^1\|_{\infty}\vee \|Y^2\|_{\infty}$. For two constants \(A, B > 0\) that are yet to be chosen, consider a change of variable \(y = \phi(\widetilde{y})\), where \(\phi : \mathbb{R} \longrightarrow \mathbb{R}\) is defined as follows:
		\begin{align*}
			\phi(\widetilde{y})=\frac{1}{B}ln\left(\frac{e^{AB\widetilde{y}}+1}{A}\right)-M.
		\end{align*}
		For $i=1,2$, we set
		\begin{align*}
			Y_t^i=\phi(\widetilde{Y}_t^i),\q\ Z_t^i=\phi'(\widetilde{Y}_t^i)\widetilde{Z}_t^i,\q\
			K_t^i=\int_0^t\phi'(\widetilde{Y}_s^i)d\widetilde{K}_s^i,
		\end{align*}
		where $(\widetilde{Y}^i,\widetilde{Z}^i,\widetilde{K}^i)$ is a solution of the following reflected BDSDE with parameters $(\widetilde{\xi}^i,\widetilde{f}^i,\widetilde{g})$ and obstacle $\widetilde{S}^i$, such that
		\begin{equation*}
			\left\{
			\begin{aligned}
				\widetilde{Y}_t^i=&\widetilde{\xi}^i+\int_t^T\widetilde{f}^i(s,\widetilde{Y}_s^i,\widetilde{Z}_s^i)ds
				+\int_t^T\widetilde{g}(s,\widetilde{Y}_s^i,\widetilde{Z}_s^i)dB_s+\widetilde{K}_T^i-\widetilde{K}_t^i-\int_t^T\widetilde{Z}_s^idW_s,\q\ 0\leq t\leq T,\\
				\widetilde{Y}_t^i\geq &\widetilde{S}_t^i\q\ \hbox{and} \q\ \int_0^T \left(\widetilde{Y}_s^i-\widetilde{S}_s^i\right)d\widetilde{K}_s^i=0.
			\end{aligned}
			\right.
		\end{equation*}
		We note that	
		\begin{equation*}
			\left\{
			\begin{aligned}
				&	\widetilde{\xi}^i=\phi^{-1}(\xi^i),\q\ \widetilde{S}_t^i=\phi^{-1}(S_t^i), \q\
				\widetilde{g}(t,\widetilde{y},\widetilde{z})=\frac{g(t,\phi(\widetilde{y}),\phi'(\widetilde{y})\widetilde{z})}{\phi'(\widetilde{y})},\\
				&	\widetilde{f}(t,\widetilde{y},\widetilde{z})=
				\frac{1}{\phi'(\widetilde{y})}\left(f(t,\phi(\widetilde{y}),\phi'(\widetilde{y})\widetilde{z})
				+\frac{1}{2}\phi''(\widetilde{y})\left(|\widetilde{z}|^2-\big|\widetilde{g}(t,\widetilde{y},\widetilde{z})\big|^2\right)\right).
			\end{aligned}\right.
		\end{equation*}
		Since \(\phi^{-1}\) is strictly increasing, \(S_T^{i}\leq \xi^i\) implies
		\(\widetilde{S}_T^i \leq \widetilde{\xi}^i\).
		From Hu, Wen, and Xiong \cite[Theorem 4.4]{wen1}, 
		the transformed terminal values, obstacles, and solution processes
		remain in the same integrability classes, and the transformed coefficient \(\widetilde{g}\) satisfies the corresponding assumptions.
	Choose suitable constants $A, B$ such that $\widetilde{f}^2$ satisfies the structural
		condition \(\textbf{(STR)}\).  Since for any \(t\in[0,T]\),  
		\begin{align*}
			\widetilde{S}_t^1\leq \widetilde{S}_t^2,  \q\ 
			\widetilde{\xi}^1\leq \widetilde{\xi}^2, \q\ 
			\widetilde{f}^1(t,\widetilde{Y}_t^1,\widetilde{Z}_t^1)\leq 	\widetilde{f}^2(t,\widetilde{Y}_t^1,\widetilde{Z}_t^1),
			\q\ \hbox{a.s.}
		\end{align*}
		In view of step 1, we obtain $\widetilde{Y}_t^1 \leq \widetilde{Y}_t^2$.  
		Moreover,  since
		\(d\widetilde{K}_t=\frac{1}{\phi' (\widetilde{Y}_t)}dK_t\),
		where	 \(dK\) is supported on \(\{Y=S\}\), and since
		  $\phi$ is a  strictly increasing function  with \(\phi' >0\),  the obstacle condition is preserved. Therefore, 
 \(d\widetilde{K}\) is supported on \(\{\widetilde{Y}=\widetilde{S}\}\) and
		  the Skorokhod condition still holds. We conclude that for any \(t \in [0, T]\), \(Y_t^1 \leq Y_t^2\) almost surely.
	\end{proof}
	
	\begin{lemma}[A priori estimate]\label{l4.9}\sl
		Let the conditions \(\textbf{(A2)}\),  \(\textbf{(G2)}\) and \(\textbf{(S2)}\) be satisfied. Moreover, let   \(\hat{M}\triangleq \|\xi\|_{\infty}\vee \|S\|_{\infty} \),
		\( C >0 \) be a constant, \(a:[0,T] \longmapsto \mathbb{R}_+ \) and \( b: [0,T]\longmapsto \mathbb{R}_+\) be two non-negative functions, such that for all \((t,y,z)\in [0,T]\times\mathbb{R}\times\mathbb{R}^d\), 
		\begin{align*}
			|f(t,y,z)|\leq a_t+b_t|y|+C|z|^2,\q\ \hbox{a.s.}
		\end{align*}
		If \((Y,Z,K)\in S^\infty_{\mathbb{F}}([0,T];\mathbb{R})\times L^2_{\mathbb{F}}([0,T];\mathbb{R}^d)\times A^2_{\mathbb{F}}([0,T];\mathbb{R}_+)\) is a solution of RBDSDE \eqref{r1},
		then for any \(t\in[0,T]\),
		\begin{align}\label{ly4.17}
			Y_t\geq  -\|\xi\|_{\infty}\exp\left\{\int_t^T b_sds\right\}-\int_t^T a_s \exp\left\{\int_t^s b_rdr\right\}ds-\frac{4C}{1-\alpha}\int_t^T|\lambda(s)|^2\exp\left\{\int_t^s b_rdr\right\}ds,\q\ \hbox{a.s.}
		\end{align}
		respectively,
		\begin{align}\label{ly4.18}
			\Big(		Y_t\leq \hat{M}\exp\left\{\int_t^T b_sds\right\}+\int_t^T a_s \exp\left\{\int_t^s b_rdr\right\}ds+\frac{4C}{1-\alpha}\int_t^T|\lambda(s)|^2\exp\left\{\int_t^s b_rdr\right\}ds,\q\ \hbox{a.s.} \Big)
		\end{align}
		
		Furthermore, there exists a constant $\hat{K}>0$ which depends on $\alpha,C,T, \|a\|_{L^1([0,T];\mathbb{R}_+)}, \|b\|_{L^1([0,T];\mathbb{R}_+)}$, \(\|S\|_{\infty}\) and $\|Y\|_{S_{\mathbb{F}}^\infty([0,T];\mathbb{R})}$, such that
		\begin{align}\label{lz4.19}
			\mathbb{E}\int_0^T|Z_s|^2ds\leq \hat{K}.
		\end{align}
	\end{lemma}
	\begin{proof}
	Replace the terminal value \(\|\xi\|_{\infty}\) of ODE  \eqref{x} 
	by \(X_T=\hat{M}\). Since \(\hat{M}\geq \|S\|_{\infty}\), we have \(X_t\geq S_t\) for all \(t\in [0,T]\).  Hence, by the
	Skorokhod condition, \(dK_t\) is supported by the set \(\{Y_t=S_t\}\), and \(Y_t-X_t=S_t-X_t\leq 0\) on this set.
	Since \(\Phi'(u)=0\) for \(u\leq 0\), it follows that 
	\begin{align*}
	\int_t^T \Phi'(Y_s-X_s)dK_s=0.
	\end{align*}
	The remainder of the proof of the upper bound 
		 is similar to that of \autoref{l4.1}.  We now prove the lower bound. 
		Let us consider a linear ordinary differential equation,
		\begin{align}\label{xx}
			\hat{X}_t=-\|\xi\|_{\infty}-\int_t^T \left( b_s\hat{X}_s+a_s+\frac{4C}{1-\alpha}|\lambda(s)|^2\right)ds.
		\end{align}
		For \(M=\|Y\|_{S^{\infty}_{\mathbb{F}}([0,T];\mathbb{R})}+\|\hat{X}\|_{S^{\infty}_{\mathbb{F}}([0,T];\mathbb{R})}\), replace the function 
		\(\Phi\) in \eqref{Phi} with the following:
		\begin{equation*}
			\hat{\Phi}(u)=\left\{
			\begin{aligned}
				0&,\qquad u\in\left(0,M\right], \\
				\frac{1-\alpha}{4C}(e^{-\frac{4C}{1-\alpha}u}-1)&,\qquad  	u\in[-M,0].
			\end{aligned}
			\right.
		\end{equation*}
			Let \((\hat{\Phi}_\varepsilon)_{\varepsilon>0}\subset C^2(\mathbb R)\) be a sequence of smooth convex approximations of  \(\hat{\Phi}\), such that 
		\begin{align*}
			\hat{\Phi}_\varepsilon(u)\longrightarrow\hat{\Phi}(u),\q\
			\hat{\Phi}_\varepsilon'(u)\longrightarrow \hat{\Phi}'(u),\q\
			\hat{\Phi}_\varepsilon''(u)\longrightarrow \hat{\Phi}''(u), \q\ u\neq 0.
		\end{align*}
		Moreover,
			\begin{align*}
			u\hat{\Phi}_\varepsilon'(u)\leq \hat{\Phi}_\varepsilon(u)\left(1+\frac{4C M}{1-\alpha}\right)+r_\varepsilon,\q\
			-\hat{\Phi}_\varepsilon'(u)C-\frac{1-\alpha}{2}\hat{\Phi}_\varepsilon''(u)\leq r_\varepsilon,\q\
			\frac{1}{2}\hat{\Phi}_\varepsilon''(u)+\frac{4C}{1-\alpha}\hat{\Phi}_\varepsilon'(u)\leq r_\varepsilon,
		\end{align*}
	and \(\hat{\Phi}'_\varepsilon\leq r_\varepsilon\), 	where \(r_\varepsilon \longrightarrow 0\) uniformly on \([-M,M]\).
		Applying It\^o's formula to $\hat{\Phi}_\varepsilon(Y_t-\hat{X}_t)$,  we have
		\begin{equation}\label{hatphi}
			\begin{aligned}
				\hat{\Phi}_\varepsilon(Y_t-\hat{X}_t)=&\ \hat{\Phi}_\varepsilon(Y_T-\hat{X}_T)+\int_{t}^{T}\hat{\Phi}_\varepsilon'(Y_s-\hat{X}_s)\left(f(s,Y_s,Z_s)+b_s \hat{X}_s+a_s+\frac{4C}{1-\alpha}|\lambda(s)|^2\right)ds\\
				&+\frac{1}{2}\int_t^T\hat{\Phi}_\varepsilon''(Y_s-\hat{X}_s)\Big(\big|g(s,Y_s,Z_s)\big|^2-|Z_s|^2\Big)ds
				+\int_t^T\hat{\Phi}_\varepsilon'(Y_s-\hat{X}_s)dK_s     \\
				&+\int_t^T\hat{\Phi}_\varepsilon'(Y_s-\hat{X}_s)g(s,Y_s,Z_s)dB_s-\int_t^T\hat{\Phi}_\varepsilon'(Y_s-\hat{X}_s)Z_sdW_s.
		\end{aligned}	\end{equation}
	Using the assumptions  on the coefficients, taking  expectations on both sides, and letting \(\varepsilon\longrightarrow 0\), we obtain
		\begin{align*}
			0\leq \mathbb{E}\left[\hat{\Phi}(Y_t-\hat{X}_t) \right]
			\leq \int_t^T\hat{b}_s \mathbb{E}\left[\hat{\Phi}(Y_s-\hat{X}_s)\right]ds,
		\end{align*}
		where $\hat{b}_s=b_s \big(1+\frac{4C M}{1-\alpha}\big)$ is a positive, deterministic function.  Gronwall's inequality then yields that
		$\mathbb{E}\left[\hat{\Phi}(Y_t-\hat{X}_t)\right]=0$, which implies that
	  for all $t\in[0,T]$, $\hat{\Phi}(Y_t-\hat{X}_t)=0$, $\mathbb{P}$-a.s.  	Since \(\hat{\Phi}=0\) if and only if \(u\geq 0\),   we obtain
		$Y_t\geq \hat{X}_t$. \\
		To derive the inequality \eqref{lz4.19}, we will replace the function \(\hat{\Phi}\) in \eqref{hatphi} with the following:
		\begin{align*}
			\bar{\Phi}(u)=\frac{1-\alpha}{4C}\big(e^{-\frac{4C}{1-\alpha}(u-M-1)}-1\big),\q\ u\in [-M,M].
		\end{align*} 
		The remaining proof is the same as in \autoref{l4.1}.
	\end{proof}
	
	Next,  we present  a monotone stability theorem 
	for reflected BDSDEs, which serves as a  crucial tool in proving the existence of  solutions to   reflected BDSDEs.

	\begin{proposition}[Monotone Stability]\label{p4.9}\sl
		Let the conditions \(\textbf{(G1)}\), \(\textbf{(G2)}\) and \(\textbf{(S2)}\) be satisfied, and let \((\xi,f,g)\) denote a set of parameters
		for reflected BDSDE \eqref{r1} with obstacle \(S\). Assume that
		 \(f^n\) satisfies \(\textbf{(F2)}\) with some positive constant \(C\) which is independent of \(n\), 
		  and for every fixed \((t,\omega), \) the sequence
		\((f^n)\) converges to \(f\) locally uniformly in \((y,z)\) on compact sets.
		Additionally,   \(\xi^n\) converges to \(\xi\) as \(n\longrightarrow \infty\),
		 \(\mathbb{P}\)-a.s.
		If, for every \(n\in \mathbb{N}\), reflected BDSDE \eqref{r1} with parameters \((\xi^n,f^n,g)\) and obstacle \(S\) admits a solution \((Y^n,Z^n,K^n)\in
		S_\mathbb{F}^{\infty}([0,T];\mathbb{R})\times L_\mathbb{F}^2([0,T];\mathbb{R}^d)\times A^2_{\mathbb{F}}([0,T];\mathbb{R}_+)\), where
the sequence \((K^n)\) satisfies	\(\sup\limits_{n\geq 1}\mathbb{E}\left[ (K_T^n)^2\right]<\infty\), 	the sequence \((Y^n)\) is monotone, and  for some  positive constant \(\hat{K}\),  \(\sup\limits_{n\geq 1}\|Y^n\|_{\infty}\leq \hat{ K}\).
		Then there exists a triple \((Y,Z,K)\in
		S_\mathbb{F}^{\infty}([0,T];\mathbb{R})\times L_\mathbb{F}^2([0,T];\mathbb{R}^d)\times A^2_{\mathbb{F}}([0,T];\mathbb{R}_+)\) such that \((Y^n,K^n)\) converges uniformly in probability to \((Y,K)\)  on  \( [0,T]\), \((Z^n)\) converges to \(Z\) in \(L_\mathbb{F}^2([0,T];\mathbb{R}^d)\),
		where \((Y,Z,K)\in
		S_\mathbb{F}^{\infty}([0,T];\mathbb{R})\times L_\mathbb{F}^2([0,T];\mathbb{R}^d)\times A^2_{\mathbb{F}}([0,T];\mathbb{R}_+)\)
		is a solution of reflected BDSDE \eqref{r1} with parameters \((\xi,f,g)\) and obstacle \(S\).
	\end{proposition}
	\begin{proof} 
		In the It\^o estimates used to prove the convergence of \((Y^n,Z^n)\), after interchanging \(n\) and \(m\), if necessary, we may assume that \(Y^n\geq Y^m\).
		By the monotonicity of \(Y^n\)
		 and the Skorokhod condition,
		 the additional reflection term satisfies
		\begin{align*}
		 \int_0^T \Phi'(Y_s^n-Y_s^m)d(K_s^n-K_s^m)\leq 0.
		\end{align*}
		Hence the   argument  of Hu, Wen, and Xiong \cite[Proposition 3.9]{wen1} remains valid.
	For completeness, we indicate
	the only additional point caused by the reflection.
	Once \(Y^n\longrightarrow Y\) uniformly on \([0,T]\), a.s.,
	and \(Z^n\longrightarrow Z\) strongly in \(L_\mathbb{F}^2([0,T];\mathbb{R}^d)\)
	have been obtained as in Hu, Wen, and Xiong \cite[Proposition 3.9]{wen1}, we define
	\begin{align*}
 K_t=Y_0-Y_t-\int_0^t f(s,Y_s,Z_s)ds
 -\int_0^tg(s,Y_s,Z_s)dB_s +\int_0^t Z_s dW_s, \q\ 
 0\leq t\leq T.
	\end{align*}
	For the convergence of the generator,  note that the sequence 	\((f^n)\) converges locally
	uniformly to  \(f\). Using the uniform boundedness of \(Y^n\) and the uniform integrability of  \(|Z^n|^2\) in \(L^1_{\mathbb{F}}([0,T];\mathbb{R}^d)\), 
	  Vitali's convergence theorem gives
	\begin{align*}
		\mathbb{E}\int_0^T \big|f^n(s,Y_s^n,Z_s^n)-
		 f(s,Y_s,Z_s)\big| ds \longrightarrow 0.
	\end{align*}
	Consequently,
 the equation satisfied by \(K^n\), 
	 the convergence of \((Y^n,Z^n)\) and \(f^n(\cdot,Y^n,Z^n)\), together  with  the	Burkholder--Davis--Gundy  inequality, we obtain
	\(K^n\longrightarrow K\) uniformly in probability  on \([0,T]\).
	Passing to a subsequence if necessary, we may assume that 
	\(K_T^n\longrightarrow K_T\), \(\mathbb{P}\)-a.s.
	Hence, by Fatou's lemma, 
	\begin{align*}
	\mathbb{E}\left[
	(K_T)^2\right]\leq \liminf\limits_{n\rightarrow\infty}\mathbb{E}\left[ (K_T^n)^2\right]<\infty.
	\end{align*}
	Since each \(K^n\) is continuous and non-decreasing,
	and \(K^n\longrightarrow K\) uniformly on \([0,T]\), \(\mathbb{P}\)-a.s.,
	 along
	 a subsequence, 
	 the limiting process
	\(K\) is also continuous and non-decreasing. Moreover,
	since \(\int_0^T \big(Y_t^n-S_t \big)dK_t^n=0\), we write
	\begin{align*}
	\Big| \int_0^T \big( Y_t^n-S_t \big)dK_t^n-\int_0^T\big( Y_t-S_t \big)dK_t\Big|
	\leq \sup\limits_{0\leq t\leq T}|Y_t^n-Y_t|\ K_T^n+
\Big|	\int_0^T\big(Y_t-S_t \big)\big(dK_t^n-dK_t \big)\Big|,
	\end{align*}
		Since  \(\sup\limits_{n}\mathbb{E}\left[ (K_T^n)^2\right]<\infty\),
	the first term on the right-hand side converges to zero in probability. For the second term, 
	since \(Y-S\) is continuous, passing to a subsequence, we may assume that 
	   \(K^n\longrightarrow K \) uniformly on \([0,T]\), \(\mathbb{P}\)-a.s.  Hence the measure \(dK^n\) weakly converges to 
	  \(dK\) and the second term converges to zero.
This yields the desired Skorokhod condition and completes the proof.
	\end{proof}

	Now, we give the main result in this part.

	\begin{theorem}\label{t5.4} \sl
		Let the conditions \(\textbf{(A2)}\), \(\textbf{(F2)}\), \(\textbf{(G1)}\),  \(\textbf{(G2)}\) and \(\textbf{(S2)}\) be satisfied. 
		Then reflected BDSDE \eqref{r1} has at least one solution $(Y,Z,K)\in S_\mathbb{F}^{\infty}([0,T];\mathbb{R})\times L_\mathbb{F}^2([0,T];\mathbb{R}^d)\times A^2_{\mathbb{F}}([0,T];\mathbb{R}_+) $. Moreover, if the conditions \(\textbf{(F4)}\) and \(\textbf{(G4)}\) are satisfied, then the solution is unique.
	\end{theorem}
	
	\begin{proof}
		Uniqueness follows immediately from 
		\autoref{t5.2}. Therefore,  it remains to establish the existence. We begin by assuming that the coefficient $f$ is 
		non-negative.  Using the continuous  approximation lemma of
		Lepeltier and San Mart\'in \cite[Lemma 1]{97},
		for each \(n\in \mathbb{N}\) with  \(n\geq C\), let $(f_n)$ be a sequence of Lipschitz functions that approximates $f$ from below. Specifically, we define:
		\begin{align*}
			f_n(t,y,z)=\inf\limits_{(y',z')\in\mathbb{Q}^{1+d}}\{f(t,y',z')+n|y-y'|+n|z-z'|\}.
		\end{align*}
		By  the existence and uniqueness result  of  reflected BDSDE with Lipschitz coefficients (see Bahlali et al. \cite[Theorem 3.1]{09}), for each \(n\in \mathbb{N}\), 
		there exists a unique solution $(Y^n,Z^n,K^n)$ which solves
		 reflected BDSDE \eqref{r1} with parameters $(\xi,f_n,g)$ and obstacle $S$.   	Since \(f_n\) increases pointwise to \(f\) and the reflected
		 comparison theorem in the  Lipschitz case applies 
		 (see Bahlali et al. \cite[Theorem 3.2]{09}), the corresponding sequence  \((Y^n)\) is non-decreasing. 
		 Moreover, \((Y^n)_{n\geq C}\) is uniformly bounded.
		By \autoref{l4.9}, the uniform boundedness of  \((Y^n)\) is independent of \(n\).
The uniform quadratic growth estimate 
		for \((f_n)\) and the Burkholder--Davis--Gundy inequality give
		 \(\sup\limits_{n} \mathbb{E}[(K_T^n)^2]<\infty\).
		All constants are independent of the Lipschitz constants of \((f_n)\)
	and 	 the assumptions of
		 \autoref{p4.9} are satisfied.  We can first establish the convergence of the 
		pair of sequences
		\((Y^n,Z^n)\), the convergence of the sequence \((K^n)\) follows by
		\begin{align*}
		K_t^n=Y_0^n-Y_t^n-\int_0^tf_n(s,Y_s^n,Z_s^n)ds
		-\int_0^tg(s,Y_s^n,Z_s^n)dB_s+\int_0^tZ_s^ndW_s,\q\ 0\leq t\leq T.
		\end{align*}
		The obstacle condition and the Skorokhod condition 
		follow from \autoref{p4.9}; in particular, 
		\(Y\geq S\) and \(\int_0^T(Y_s-S_s)dK_s=0\).
		We conclude that the  triple 
		$(Y,Z,K)\in S_\mathbb{F}^{\infty}([0,T];\mathbb{R})\times L_\mathbb{F}^2([0,T];\mathbb{R}^d)\times A^2_{\mathbb{F}}([0,T];\mathbb{R}_+)$ is a solution of reflected BDSDE \eqref{r1}.\\
		Finally, for the general case, we consider
		\begin{align*}
			f_{n,m}(t,y,z)=&\inf\limits_{(y',z')\in\mathbb{Q}^{1+d}}\{f^+(t,y',z')+n|y-y'|+n|z-z'|\}\\
			&	-\inf\limits_{(y',z')\in\mathbb{Q}^{1+d}}\{f^-(t,y',z')+m|y-y'|+m|z-z'|\}.
		\end{align*}
		For each fixed \(m\in \mathbb{N}\),  we 
		 first pass to the limit as $n\longrightarrow \infty$.
		 Since \(f_{n,m} \uparrow f_{\infty,m}\triangleq
		 f^+(t,y,z)
		 -\inf\limits_{(y',z')\in  \mathbb{Q}^{1+d}}
		 \{f^-(t,y',z')+m|y-y'|+m|z-z'|\}\) locally uniformly as \(n\longrightarrow\infty\).
		 The reflected comparison theorem implies that
		 the corresponding sequence \((Y^{n,m})_n\) is non-decreasing. By \autoref{l4.9}, it is uniformly bounded
		 by a constant   independent of \(n\) and  \(m\).  This is because 
		 \(|f_{n,m}(t,y,z)| \leq
		 C(1+|y|+|z|^2)\). The uniform estimate of \(K^{n,m}\) follows as above.
		 Thus  \autoref{p4.9} 
		 applies and yields a reflected solution.
	 Next,   since \(f_{\infty,m}\downarrow f\) locally uniformly as  $m\longrightarrow\infty$,  repeating the above argument   yields the desired solution of  reflected BDSDE \eqref{r1}. The convergence
	 of \((K^{n,m})\) and the preservation of the Skorokhod condition are included in \autoref{p4.9}.
	\end{proof}

	\section{General growth in  $y$ and quadratic growth in  $z$} \label{sectionnew5}
	
		In this section, we study the maximal solutions for BDSDEs and reflected BDSDEs when the generator
	$f$ has one-sided quadratic growth with respect to $z$, and the terminal value \(\xi\) is bounded. For the sake of clarity, we will divide our results into two subsections.

\subsection{BDSDE: existence of maximal solutions}\label{sec4.2}

The main results presented in this part are the following two theorems (\autoref{t4.2} and \autoref{t4.3}), which are obtained as applications of the comparison theorem (\autoref{l4.6}). In contrast to the work of Shi, Gu, and Liu \cite{05}, we obtain the existence of maximal solutions for quadratic BDSDEs.   Our approach relies on truncation arguments and a monotone convergence procedure.
Using the same method, we can also derive the comparison theorem (\autoref{c4.10}) for quadratic BDSDEs, where the generator \(f\) is of general growth in \(y\). This further extends the work of Hu, Wen, and Xiong \cite{wen1}.

To construct a monotone approximation for a quadratic generator,
we suppose the following one-sided quadratic condition. Notice that no normalization \(\varphi(0)=0\) is required under \(\textbf{(F3')}\).
\begin{description}
	\item[(F2')]  
	\( -C(1+|y|+|z|^2)\leq f(t,y,z)\leq C(1+|y|).
	\)
	\item [(F3')] \( -\varphi(|y|)- C |z|^2\leq   f(t,y,z) \leq \varphi(|y|) \q\hbox{and} \q\
	\left(y-y'\right) \big(f(t,y,z)-f(t,y',z)\big)\leq \mu \big|y-y'\big|^2.\)
\end{description}

\begin{remark}\label{rm_qua}\sl 
	Conditions \textbf{(F2)} and \textbf{(F3)} are two-sided quadratic growth assumptions, whereas \textbf{(F2')} and \textbf{(F3')} are one-sided quadratic assumptions introduced for the construction of maximal solutions.
	In general, neither \textbf{(F2)} nor \textbf{(F3)} directly implies its one-sided counterpart.
	In the special case where \(g\) is positively homogeneous
	in \(z\), an increasing exponential change of variables
	transforms  
	the two-sided quadratic problem into a one-sided one. 
	Moreover,  since the transformation is increasing, maximality is preserved under the inverse transformation.
	Therefore, under the additional assumption that \(g\) is positively homogeneous in \(z\),
	\autoref{t4.2} and \autoref{t4.3} 
	hold under the condition \(\textbf{(F2)}\) and \(\textbf{(F3)}\), respectively.
\end{remark}

\begin{theorem}\label{t4.2} \sl
	Assume that the conditions \(\textbf{(A2)}\), \(\textbf{(F2')}\), \(\textbf{(G1)}\), and \(\textbf{(G2)}\) hold with \(\lambda(t) = 0\) for all \(t \in [0, T]\).
	Then  BDSDE \eqref{1} has a maximal bounded solution $(Y,Z)\in S_\mathbb{F}^{\infty}([0,T];\mathbb{R})\times L_\mathbb{F}^2([0,T];\mathbb{R}^d)$, and for  any $t\in[0,T]$, $U_0\leq U_t\leq Y_t\leq V_t\leq V_0$, $ \dbP $-a.s., where
	$(U,V)\in S_\mathbb{F}^{2}([0,T];\mathbb{R})\times S_\mathbb{F}^2([0,T];\mathbb{R})$ are the unique solutions of the following pair of ODEs:
	\begin{align}
		U_t=-\|\xi\|_{\infty}-\int_t^T C(1-U_s)ds,\q\ 0\leq t\leq T, \label{4.1}
	\end{align}
	and 	\begin{align}
		V_t=\|\xi\|_{\infty}+\int_t^T C(1+V_s)ds,\q\ 0\leq t\leq T.\label{4.2}
	\end{align}
\end{theorem}

\begin{theorem}\label{t4.3} \sl
	Assume that the
	conditions \(\textbf{(A2)}\), \(\textbf{(F3')}\),  \(\textbf{(G1)}\) and \(\textbf{(G2)}\)  hold  with \(\lambda(t)=0\) for all \(t\in [0,T]\). Then  BDSDE \eqref{1} has a maximal bounded solution $(Y,Z)\in S_\mathbb{F}^{\infty}([0,T];\mathbb{R})\times L_\mathbb{F}^2([0,T];\mathbb{R}^d)$.
\end{theorem}

\begin{remark}\label{lam}\sl 
	The restriction \(\lambda(t)=0\) in the maximal solution results implies \(g(t,y,0)=0\). It allows the upper and lower
	comparison equations to be chosen as deterministic ordinary
	differential equations.  When \(\lambda(t)\neq 0\), an additional non-homogeneous backward stochastic term remains, and the deterministic ODE comparison used below is no longer directly available.
\end{remark}

The lemma presented below will play a crucial role in  our analysis.

\begin{lemma}[Comparison]\label{l4.6} \sl
	Assume that the conditions \(\textbf{(A1)}\) and \(\textbf{(G1)}\)  hold.  For any \((t,y,z)\in [0,T]\times\mathbb{R}\times\mathbb{R}^d\), we have
	\begin{align*}
		\xi^1\leq \xi^2	\q\   \hbox{and}\q\ f^1(t,y,z)\leq f^2(t,y,z)	 \q\ \hbox{a.s.}
	\end{align*}
	Moreover,	assume that $f^2(t,y,z)$ is continuous and is of linear growth in \(y\) and \(z\), i.e., there exists a finite constant $C>0$ such that
	\begin{align*}
		|f^2(t,y,z)|\leq C(1+|y|+|z|),\q\ \hbox{ a.s.}
	\end{align*}
	Then for any \(t\in [0,T]\), $Y_t^1\leq Y_t^2,  \dbP $-a.s., where $(Y^1,Z^1)\in S_\mathbb{F}^{2}([0,T];\mathbb{R})\times L_\mathbb{F}^2([0,T];\mathbb{R}^d)$ is a solution of BDSDE \eqref{1} with parameters $(\xi^1,f^1,g)$, $(Y^2,Z^2)\in S_\mathbb{F}^{2}([0,T];\mathbb{R})\times L_\mathbb{F}^2([0,T];\mathbb{R}^d)$ is the maximal solution of BDSDE \eqref{1} with parameters $(\xi^2,f^2,g)$.
\end{lemma}

\begin{proof}
	Since $f^2$ is continuous and is of linear growth in \(y\) and \(z\), it follows from Shi, Gu, and Liu \cite[Theorem 4.1]{05} that BDSDE \eqref{1} with parameters $(\xi^2,f^2,g)$  has a maximal solution $(Y^2,Z^2)\in S_\mathbb{F}^{2}([0,T];\mathbb{R})\times L_\mathbb{F}^2([0,T];\mathbb{R}^d)$.
	For each \(n\in\mathbb{N}\) with $n\geq C$, let 
	\begin{align*}
		f^2_n(t,y,z)=\sup\limits_{y',z'\in\mathbb{Q}^{1+d}}\left\{f^2(t,y',z')-n|y-y'|-n|z-z'|\right\}.
	\end{align*}
	It is easy to verify  that $f^2_n\downarrow f^2$ and  \((f_n^2)\) is a sequence of Lipschitz continuous functions. Let $(Y_n^2,Z_n^2)$ 
	denote the unique solution of  BDSDE \eqref{1} with parameters $(\xi^2,f^2_n,g)$. By the comparison theorem for  BDSDEs in the Lipschitz case (refer to Shi, Gu, and Liu \cite[Theorem 3.1]{05}), we have  $Y^1\leq Y_n^2$ for all $n\geq C$.  Following the decreasing Lipschitz approximation and the monotone approximation argument in Shi, Gu, and Liu \cite[Theorem 4.1]{05},
	$Y_n^2\downarrow Y^2$, where $Y^2$ is the maximal solution of BDSDE \eqref{1} with parameters $(\xi^2,f^2,g)$. Therefore, we obtain $Y_t^1\leq Y_t^2,  \forall t\in [0,T], \dbP $-a.s.
\end{proof}

Now we give the proof of the main results in this part.

\begin{proof}[Proof of \autoref{t4.2}]
	Let
	$A=C(1+\|U_0\|_{\infty}+\|V_0\|_{\infty})$,  enlarging \(C\)
	if necessary, we may assume \(C\geq 1\).
	Define 
	a continuous truncation function \(\rho^A: \mathbb{R} \longmapsto [-A,A]\) with \(\rho^A(y)=y\) if \(|y|\leq A\).
	For each \(n\in\mathbb{N}\), let $\kappa_n :\mathbb{R}^d\longmapsto [0,1]$ be a sequence of smooth functions such that \(0\leq \kappa_n\leq \kappa_{n+1}\leq  1\), \(\kappa_n(z) \uparrow 1\) and
	\begin{equation*}
		\kappa_n(z)=
		\left\{\begin{aligned}
			1 & \qquad |z| \leq n, \\
			0 &\qquad  |z| \geq n+1.
		\end{aligned}\right.
	\end{equation*}
	Let
	\begin{align*}
		f_n(t,y,z)=C(1+|\rho^A(y)|)(1-\kappa_n(z))+\kappa_n(z)\widetilde{f}(t,y,z),
	\end{align*}
	where \(\widetilde{f}(t,y,z)=f(t,\rho^A(y),z)\). 
	Since \(\kappa_n \uparrow 1\) and 
	\(
	\widetilde{f}(t,y,z)\leq C(1+|\rho^A(y)|),
	\) then
	\(f_n \downarrow \widetilde{f}\). Moreover, in view of the condition on \(f\), we have
	\begin{align*}
		-h_n(y,z)\leq f_n(t,y,z) \leq h_n(y,0),
	\end{align*}
	where \(h_n(y,z)=C(1+|\rho^A(y)|+\kappa_n(z)|z|^2)\). 
	From the result of Shi, Gu, and Liu  \cite[Theorem 4.1]{05},
	for each \(n\in\mathbb{N}\),  there exists
	a  maximal solution \((Y^n,Z^n)\in S^2_{\mathbb{F}}([0,T];\mathbb{R})\times L^2_{\mathbb{F}}([0,T];\mathbb{R}^d)\) which solves the following BDSDE:
	\begin{align*}
		Y_t^n=\xi+\int_t^T f_n(s,Y_s^n,Z_s^n)ds +\int_t^T g(s,Y_s^n,Z_s^n)dB_s-\int_t^TZ_s^ndW_s, \q\ 0\leq t\leq T.
	\end{align*}
	Since for any \(t\in [0,T]\), \(U_t\leq 0\leq V_t\), 
	\(-h_n(U_t,0)=-C(1-U_t)\) and \(h_n(V_t,0)=C(1+V_t)\),
	where \((U,V)\in S^2_{\mathbb{F}}([0,T];\mathbb{R})\times S^2_{\mathbb{F}}([0,T];\mathbb{R})\) are the solutions of  ODEs \eqref{4.1} and \eqref{4.2} respectively. 
	Condition \((\textbf{G2})\) with \(\lambda=0\) implies that
	\(g(t,\cdot,0)=0\).
	Hence,  \((U,0)\) can be regarded as the following BDSDE:
	\begin{align*}
		U_t=-\|\xi\|_{\infty}-\int_t^Th_n(U_s,Z_s)ds+\int_t^Tg(s,U_s,Z_s)dB_s-\int_t^TZ_sdW_s, \q\ 0\leq t\leq T,
	\end{align*}
	and	\((V,0)\) can be regarded as the following BDSDE:
	\begin{align*}
		V_t=\|\xi\|_{\infty}+\int_t^Th_n(V_s,Z_s)ds+\int_t^Tg(s,V_s,Z_s)dB_s-\int_t^TZ_sdW_s, \q\ 0\leq t\leq T.
	\end{align*}
	From \autoref{l4.6}, we derive that for any \(t\in [0,T]\), 
	\begin{align*}
		U_t\leq Y_t^n\leq V_t,\q\ \hbox{a.s.}
	\end{align*}
	For the decreasing bounded sequence $(Y^n)$, there exists a process $Y$ such that $Y_t^n\downarrow Y_t$, \(\forall t\in [0,T]\), a.s. and \(Y^n\longrightarrow Y\) in $L^2_{\mathbb{F}}([0,T];\mathbb{R})$, and the following inequalities hold:   $U_0\leq U_t\leq Y_t\leq V_t\leq V_0$, for any \(t\in [0,T]\). 
	We claim that $(Z^n)$ admits a convergent  subsequence in
	$L^2_{\mathbb{F}}([0,T];\mathbb{R}^d)$. Therefore, we can verify that the limit process $(Y,Z)\in S^\infty_{\mathbb{F}}([0,T];\mathbb{R})\times L^2_{\mathbb{F}}([0,T];\mathbb{R}^d)$ is a solution of  BDSDE \eqref{1}.\\
	Next, we will prove this claim, dividing the proof into three steps.
	\ms
	
	{\bf\it Step 1}. The sequence $(Z^n)$ is bounded in $L^2_{\mathbb{F}}([0,T];\mathbb{R}^d)$.

	Applying It\^o's formula to $\phi(Y^n_t)$, where \(\phi : \mathbb{R}\longmapsto \mathbb{R}\) is defined by $\phi(x)=e^{\frac{4C}{1-\alpha}x}$.  Taking expectations, we get
	\begin{align*}
		\mathbb{E}\left[\phi(Y_0^n)\right]+\frac{1}{2}\mathbb{E}\int_0^T\phi''(Y_s^n)|Z_s^n|^2ds
		=&\ \mathbb{E}\left[\phi(\xi)\right]+\mathbb{E}\int_0^T\phi'(Y_s^n)f_n(s,Y_s^n,Z_s^n)ds\\
		&	+\frac{1}{2}\mathbb{E}\int_0^T\phi''(Y_s^n)\big|g(s,Y_s^n,Z_s^n)\big|^2ds.
	\end{align*}
	In view of the conditions on \(f\) and \(g\), we have
	\begin{align*}
		|f_n(t,y,z)|\leq A+C|z|^2\q\ \hbox{and} \q\ |g(t,y,z)|^2\leq \alpha |z|^2,
	\end{align*}
	where $A=C(1+\|U_0\|_{\infty}+\|V_0\|_{\infty})$.  Then, by a straightforward calculation,
	\begin{align*}
		\frac{1-\alpha}{2}\phi''-C\phi'=\frac{4C^2}{1-\alpha}\phi.
	\end{align*}
	We obtain
	\begin{align*}
		\frac{4C^2}{1-\alpha}\mathbb{E}\int_0^T\phi(Y_s^n)|Z_s^n|^2ds
		\leq \mathbb{E}\left[\phi(\xi)\right]+\frac{4C}{1-\alpha}e^{\frac{4CA}{1-\alpha}}AT.
	\end{align*}
	Consequently,	 there exists a constant $\hat{K}>0$ which is  independent of $n$, such that $\|Z^n\|_{L^2_{\mathbb{F}}([0,T];\mathbb{R}^d)}\leq \hat{K}<\infty$.  As a result, 
	there exists a 
	subsequence, still denoted by $(Z^n)$, such that $Z^n \longrightarrow Z$ weakly in $L^2_{\mathbb{F}}([0,T];\mathbb{R}^d)$.
	\ms
	
	{\bf\it Step 2}. The sequence $(Z^n)$ strongly converges to $Z$ in $L^2_{\mathbb{F}}([0,T];\mathbb{R}^d)$.

	For each \(n,m\in\mathbb{N}\) with $n<m$, apply It\^o's formula to $\Phi(Y_t^n-Y_t^m)$, where \(\Phi : \mathbb{R} \longmapsto \mathbb{R}\) is defined as $\Phi(x)=\frac{1-\alpha}{200C^2}\left(e^{\frac{20C}{1-\alpha}x}-\frac{20C}{1-\alpha}x-1\right)$.   Then, taking expectations on both sides, we get
	\begin{align*}
		\mathbb{E}\left[\Phi(Y^n_0-Y^m_0)\right]
		=&\ \mathbb{E}\int_0^T\Phi'(Y_s^n-Y_s^m)\Big(f_n(s,Y_s^n,Z_s^n)-f_m(s,Y_s^m,Z_s^m)\Big)ds\\
		&+\frac{1}{2}\mathbb{E}\int_0^T\Phi''(Y_s^n-Y_s^m)\left(\big|g(s,Y_s^n,Z_s^n)-g(s,Y_s^m,Z_s^m)\big|^2-\big|Z_s^n-Z_s^m\big|^2\right)ds.
	\end{align*}
	Note that
	\begin{align*}
		f_n(s,Y_s^n,Z_s^n)-f_m(s,Y_s^m,Z_s^m)\leq&\
		2A+C(|Z_s^n|^2+|Z_s^m|^2)\\
		\leq &\ 2A+5C\left(\big|Z_s^n-Z_s^m\big|^2+\big|Z_s^n-Z_s\big|^2+|Z_s|^2\right),
	\end{align*}
	and	consider the  condition of  Lipschitz on \(g\),  
	we have
	\begin{align*}
		&\	\mathbb{E}\int_0^T\left(\frac{1-\alpha}{2}\Phi''(Y_s^n-Y_s^m)-5C\Phi'(Y_s^n-Y_s^m)\right)\big|Z_s^n-Z_s^m\big|^2ds\\
		&	\leq
		\mathbb{E}\int_0^T\Phi'(Y_s^n-Y_s^m)\Big(2A+5C\left(\big|Z_s^n-Z_s\big|^2+|Z_s|^2\right)\Big)ds
		+\frac{C}{2}\mathbb{E}\int_0^T\Phi''(Y_s^n-Y_s^m)\big|Y_s^n-Y_s^m\big|^2ds.
	\end{align*}
	By a simple calculation, 
	\begin{align*}
		\frac{1-\alpha}{2}\Phi''-5C\Phi'=\frac{1}{2}e^{\frac{20C}{1-\alpha}x}+\frac{1}{2}.
	\end{align*} 
	From \autoref{l4.1},
	$\Big(\left(\frac{1-\alpha}{2}\Phi''-5C\Phi'\right)\left(Y_s^n-Y_s^m\right)\Big)^{\frac{1}{2}}$ has uniform bound, 
	we can use the  dominated convergence theorem to get
	\begin{align*}
		\left(\big(\frac{1-\alpha}{2}\Phi''-5C\Phi'\big)\left(Y^n-Y^m\right)\right)^{\frac{1}{2}}
		\longrightarrow \left(\big(\frac{1-\alpha}{2}\Phi''-5C\Phi'\big)\left(Y^n-Y\right)\right)^{\frac{1}{2}}\q
		\hbox{in}~ L^2_{\mathbb{F}}([0,T];\mathbb{R}) ~\hbox{as}~m\longrightarrow\infty.
	\end{align*}
	Let
	\begin{align*}
		Q_{n,m}\triangleq	\left(\big(\frac{1-\alpha}{2}\Phi''-5C\Phi'\big)\left(Y^n-Y^m\right)\right)^{\frac{1}{2}}.
	\end{align*}
	Since \(Q_{n,m}\longrightarrow Q_n\) strongly in \(L^2_{\mathbb{F}}([0,T];\mathbb{R})\), and the sequence \((Q_{n,m})_m\) is uniformly bounded. Moreover,
	by Step 1, up to a subsequence,   $Z^m \rightharpoonup Z$ weakly   in $L^2_{\mathbb{F}}([0,T];\mathbb{R}^d)$.  Hence we have
	\begin{align*}
		Q_{n,m}\big(Z^n-Z^m\big)
		\rightharpoonup Q_n\big(Z^n-Z\big) \q
		\hbox{weakly in}~ L^2_{\mathbb{F}}([0,T];\mathbb{R}^d)~\hbox{as}~m\longrightarrow\infty.
	\end{align*}
	Therefore,
	by the weak lower semicontinuity of the norm in the Hilbert space \(L^2_{\mathbb{F}}([0,T];\mathbb{R}^d)\) (or by the generalized  Fatou lemma in Zhang  \cite[Problem 1.4.11]{zhang_backward_2017}), we deduce that
	\begin{align*}
		&\ \mathbb{E}\int_0^T\left(\frac{1-\alpha}{2}\Phi''-5C\Phi'\right)\left(Y_s^n-Y_s\right)\big|Z_s^n-Z_s\big|^2ds\\
		&\leq \liminf\limits_{m\rightarrow\infty}
		\mathbb{E}\int_0^T\left(\frac{1-\alpha}{2}\Phi''-5C\Phi'\right)\left(Y_s^n-Y_s^m\right)\big|Z_s^n-Z_s^m\big|^2ds.
	\end{align*}
	Similarly, we have
	\begin{align*}
		&\ \limsup\limits_{m\rightarrow\infty}\mathbb{E}\int_0^T\Phi'(Y_s^n-Y_s^m)\Big(2A+5C\left(\big|Z_s^n-Z_s\big|^2+|Z_s|^2\right)\Big)ds
		\\
		&\leq \mathbb{E}\int_0^T\Phi'(Y_s^n-Y_s)\Big(2A+5C\big(\big|Z_s^n-Z_s\big|^2+|Z_s|^2\big)\Big)ds
	\end{align*}
	and
	\begin{align*}
		&\ \limsup\limits_{m\rightarrow\infty}\mathbb{E}\int_0^T\Phi''(Y_s^n-Y_s^m)\big|Y_s^n-Y_s^m\big|^2ds\\
		&\leq \mathbb{E}\int_0^T \Phi''(Y_s^n-Y_s)\big|Y_s^n-Y_s\big|^2ds.
	\end{align*}
	Therefore, we get
	\begin{align*}
		&\ \mathbb{E}\int_0^T\left(\frac{1-\alpha}{2}\Phi''-10C\Phi'\right)\left(Y_s^n-Y_s\right)\big|Z_s^n-Z_s\big|^2ds\\
		&	\leq\mathbb{E}\int_0^T\Phi'(Y_s^n-Y_s)\left(2A+|Z_s|^2\right)ds
		+\frac{C}{2}\mathbb{E}\int_0^T\Phi''(Y_s^n-Y_s)\big|Y_s^n-Y_s\big|^2ds,
	\end{align*}
	where
	$\frac{1-\alpha}{2}\Phi''-10C\Phi'=1.$
	Letting $n\longrightarrow\infty$ in  the above inequality,  the dominated convergence theorem yields that
	\begin{align*}
		\lim\limits_{n\rightarrow\infty}\mathbb{E}\int_0^T\big|Z_s^n-Z_s\big|^2ds=0.
	\end{align*}
	\ms
	
	{\bf\it Step 3}. The sequence  $(Y^n)$ converges to $Y$ uniformly in $t$.

	Note that for each \(n \in \mathbb{N}\), $f_n$ is continuous and the sequence $f_n\downarrow f$. By Dini's theorem, for every fixed \((t,\omega),\)  \(f_n\)  converges  locally uniformly to \(f\) with respect to \((y,z)\).
	Moreover, in view of Step 2,  since \(Z^n\longrightarrow Z\) strongly in \(L^2_{\mathbb{F}}([0,T];\mathbb R^d)\),
	and
	from the quadratic growth condition, we have
	\begin{align*}
		\Big|f_n(s,Y_s^n,Z_s^n)-f_m(s,Y_s^m,Z_s^m)\Big|\leq 2A+C\big(|Z_s^n|^2+|Z_s^m|^2\big),
	\end{align*}
	and the right-hand side can be controlled by the uniform integrability of 
	the family
	\(\{|Z^n|^2;n\geq 1\}\) in \(L^1_{\mathbb{F}}([0,T];\mathbb{R}^d)\).
	Applying Vitali's convergence  theorem, we obtain
	\begin{align*}
		\lim\limits_{n,m\rightarrow\infty}\mathbb{E}\int_0^T	\Big|f_n(s,Y_s^n,Z_s^n)-f_m(s,Y_s^m,Z_s^m)\Big|ds=0.
	\end{align*}
	By the Burkholder--Davis--Gundy inequality, we have
	\begin{align*}
		\lim\limits_{n,m\rightarrow\infty}\mathbb{E}\left[\sup\limits_{0\leq t\leq T}\Big| \int_t^T\left( Z_s^n-Z_s^m\right)dW_s\Big|^2\right]=0,
	\end{align*}
	and
	\begin{align*}
		\lim\limits_{n,m\rightarrow\infty}\mathbb{E}\left[\sup\limits_{0\leq t\leq T}\Bigg|\int_t^T\Big(g(s,Y_s^n,Z_s^n)-g(s,Y_s^m,Z_s^m)\Big)dB_s\Bigg|^2\right]=0.
	\end{align*}
	Using the following inequality,
	\begin{align*}
		\mathbb{E}\left[\sup\limits_{0\leq t\leq T}\big|Y_t^n-Y_t^m\big|\right]
		\leq&\ \mathbb{E}\int_0^T	\Big|f_n(s,Y_s^n,Z_s^n)-f_m(s,Y_s^m,Z_s^m)\Big|ds
		+\mathbb{E}\left[\sup\limits_{0\leq t\leq T}\Big|\int_t^T\left(Z_s^n-Z_s^m\right)dW_s\Big|\right]\\
		&	+\mathbb{E}\left[\sup\limits_{0\leq t\leq T}\Bigg|\int_t^T\Big(g(s,Y_s^n,Z_s^n)-g(s,Y_s^m,Z_s^m)\Big)dB_s\Bigg|\right].
	\end{align*}
	We conclude that $Y^n\longrightarrow Y$ uniformly in $t$, $ \dbP $-a.s., and $Y$ is a continuous process.

	Passing  to the limit in the BDSDE satisfied by \((Y^n,Z^n)\), we obtain that \((Y,Z)\) solves the BDSDE
	\eqref{1} with generator \(\widetilde{f}\). Since 
	\(U_t\leq Y_t\leq V_t\) for all \(t\in [0,T]\), by the choice of
	\(A\), we have \(|Y_t|\leq A\), \(\forall t\in[0,T]\). Hence \(\rho^A(Y_t)=Y_t\)
	and \(\widetilde{f}(t,Y_t,Z_t)=f(t,Y_t,Z_t). \) Therefore,
	$(Y,Z)\in S^\infty_{\mathbb{F}}([0,T];\mathbb{R})\times L^2_{\mathbb{F}}([0,T];\mathbb{R}^d)$ is a solution of  BDSDE \eqref{1}. Finally,  the maximality follows from applying 
	\autoref{l4.6}.
\end{proof}

\begin{proof}[Proof of \autoref{t4.3}]
	Let  $\bar{C}>0$ be fixed,  define a continuous function $\rho^{\bar{C}} :\mathbb{R}\longmapsto [0,1]$ such that
	\begin{align}\label{4.12}
		\rho^{\bar{C}}(y)=\begin{cases}
			1 & ~\text{if}~|y|\leq \bar{C}, \\
			0 &~\text{if}~ |y|\geq 2\bar{C}.
		\end{cases}
	\end{align}
	Let 
	\begin{align}\label{4.13}
		f^{\bar{C}}(t,y,z)=\rho^{\bar{C}}(y)f(t,y,z).
	\end{align}
	Then,  in view of the condition on \(f\), we have
	\begin{align*}
		|f^{\bar{C}}(t,y,z)| \leq \rho^{\bar{C}}(y)\left(\varphi(|y|)+C|z|^2\right)
		\leq \varphi(2\bar{C})+C|z|^2.
	\end{align*}
	It follows from \autoref{t4.2} that there exists a maximal solution $(Y^{\bar{C}},Z^{\bar{C}})\in S_\mathbb{F}^{\infty}([0,T];\mathbb{R})\times L_\mathbb{F}^2([0,T];\mathbb{R}^d)$ which solves BDSDE \eqref{1} with parameters
	$(\xi,f^{\bar{C}},g)$, such that 
	\begin{align}\label{4.14}
		Y_t^{\bar{C}}=\xi+\int_t^Tf^{\bar{C}}(s,Y_s^{\bar{C}},Z_s^{\bar{C}})ds+\int_t^Tg(s,Y_s^{\bar{C}},Z_s^{\bar{C}})dB_s
		-\int_t^TZ_s^{\bar{C}}dW_s,\q\ 0\leq t\leq T.
	\end{align}
	For any even $n\geq 2$ and $a\in \mathbb{R}$, apply It\^o's formula to $e^{at}(Y_t^{\bar{C}})^n$, we have
	\begin{align*}
		e^{at}(Y_t^{\bar{C}})^n
		=&\ e^{aT}\xi^n-a\int_t^Te^{as}(Y_s^{\bar{C}})^nds +n\int_t^Te^{as}(Y_s^{\bar{C}})^{n-1}f^{\bar{C}}(s,Y_s^{\bar{C}},Z_s^{\bar{C}})ds\\
		&+n\int_t^Te^{as}(Y_s^{\bar{C}})^{n-1}g(s,Y_s^{\bar{C}},Z_s^{\bar{C}})dB_s
		-n\int_t^Te^{as}(Y_s^{\bar{C}})^{n-1}Z_s^{\bar{C}}dW_s\\
		&+\frac{n(n-1)}{2}\int_t^T e^{as}(Y_s^{\bar{C}})^{n-2}\left( \big|g(s,Y_s^{\bar{C}},Z_s^{\bar{C}})\big|^2-\big|Z_s^{\bar{C}}\big|^2\right)ds.
	\end{align*}
	To estimate the bound of the solution \(Y^{\bar{C}}\),	 we define
	\begin{align}\label{4.15}
		X_t^n=\|\xi\|_{\infty}^ne^{aT}+n\int_t^Te^{as}\varphi(0)ds.
	\end{align}
	For any $x\in\mathbb{R}$, consider a non-negative $C^2(\mathbb{R})$ function $\Phi$ with $\Phi'(x)\geq 0, \Phi''(x)\geq 0$, and $\Phi(x)=0$ if and only if $x\leq 0$.
	Applying It\^o's formula to $\Phi(e^{at}(Y_t^{\bar{C}})^n-X_t^n)$, we get
	\begin{align*}
		\Phi(e^{at}(Y_t^{\bar{C}})^n-X_t^n)=&\
		\Phi(0)+\int_t^T\Phi'(e^{as}(Y_s^{\bar{C}})^n-X_s^n)
		\Big(-ae^{as}(Y_s^{\bar{C}})^n+ne^{as}(Y_s^{\bar{C}})^{n-1}f^{\bar{C}}(s,Y_s^{\bar{C}},Z_s^{\bar{C}})\\
		&-ne^{as}\varphi(0)
		+\frac{n(n-1)}{2}e^{as}(Y_s^{\bar{C}})^{n-2}\big(\big|g(s,Y_s^{\bar{C}},Z_s^{\bar{C}})\big|^2-|Z_s^{\bar{C}}|^2\big)\Big)ds\\
		&+\int_t^T \Phi'(e^{as}(Y_s^{\bar{C}})^n-X_s^n) ne^{as}(Y_s^{\bar{C}})^{n-1}g(s,Y_s^{\bar{C}},Z_s^{\bar{C}})dB_s\\
		&-\int_t^T \Phi'(e^{as}(Y_s^{\bar{C}})^n-X_s^n) ne^{as}(Y_s^{\bar{C}})^{n-1}Z_s^{\bar{C}}dW_s\\
		&+\frac{1}{2}\int_t^T \Phi''(e^{as}(Y_s^{\bar{C}})^n-X_s^n) \Big( n^2e^{2as}(Y_s^{\bar{C}})^{2n-2}\big(\big|g(s,Y_s^{\bar{C}},Z_s^{\bar{C}})\big|^2-|Z_s^{\bar{C}}|^2\big)\Big)ds.
	\end{align*}
	Note that
	\begin{align*}
		yf(t,y,z)\leq yf(t,0,z)+\mu y^2\q\ \hbox{and} \q\ y^{n-2}\geq 0,
	\end{align*}
	it follows that
	\begin{align*}
		y^{n-1}f^{\bar{C}}(t,y,z)=&\ y^{n-1}\rho^{\bar{C}}(y)f(t,y,z)\\
		\leq&\  |y|^{n-1}|f(t,0,z)|\rho^{\bar{C}}(y)+\mu y^n\\
		\leq&\  (1+y^n)\varphi(0)+2C\bar{C}|z|^2y^{n-2}+\mu y^n.
	\end{align*}
	Choose $a=n(\varphi(0)+\mu)$ with  $n-1\geq \frac{4C\bar{C}}{1-\alpha}$, and recall the condition on \(g\),
	\begin{align*}
		\frac{1}{2}\int_t^T \Phi''(e^{as}(Y_s^{\bar{C}})^n-X_s^n) \Big( n^2e^{2as}(Y_s^{\bar{C}})^{2n-2}\big(\big|g(s,Y_s^{\bar{C}},Z_s^{\bar{C}})\big|^2-|Z_s^{\bar{C}}|^2\big)\Big)ds\leq 0.
	\end{align*}
	Therefore, taking expectations on both sides,
	we obtain 
	\begin{align*}
		\mathbb{E}\left[\Phi(e^{at}(Y_t^{\bar{C}})^n-X_t^n)\right]\leq \mathbb{E}\left[\Phi(0)\right]=0,
	\end{align*}
	which yields that
	\begin{align*}
		\Phi(e^{at}(Y_t^{\bar{C}})^n-X_t^n)=0\q\  \hbox{and} \q\
		e^{at}(Y_t^{\bar{C}})^n\leq X_t^n,\q\ \hbox{a.s.},
	\end{align*}
	hence,
	\begin{align}\label{4.16}
		|Y_t^{\bar{C}}|\leq 
		\left(e^{\left(\varphi(0)+\mu\right)T}\vee 1\right)\left(\|\xi\|_{\infty}+1\right).
	\end{align}
	If $\bar{C}$ is chosen as follows
	\begin{align*}
		\bar{C}\geq \left(e^{\left(\varphi(0)+\mu\right)T}\vee 1\right)\left(\|\xi\|_{\infty}+1\right).
	\end{align*}
	We get $|Y^{\bar{C}}|\leq \bar{C}$ and $f^{\bar{C}}(t,Y_t^{\bar{C}},Z_t^{\bar{C}})=f(t,Y_t^{\bar{C}},Z_t^{\bar{C}})$.
	For any bounded solution \((Y,Z)\), a priori estimate yields that 
	\(|Y| \leq \bar{C}\), which implies that  \((Y,Z)\) is also
	a solution of   the truncated BDSDE \eqref{4.14}. By the maximality of  truncated BDSDE, we conclude 
	that $(Y^{\bar{C}},Z^{\bar{C}})\in S_\mathbb{F}^{\infty}([0,T];\mathbb{R})\times L_\mathbb{F}^2([0,T];\mathbb{R}^d)$ is a maximal solution of BDSDE \eqref{1}.
\end{proof}

As a consequence of the method used in the proof of \autoref{t4.3}, we obtain the following comparison result.

\begin{corollary}[Comparison]\label{c4.10} \sl
	Assume that 
	the conditions \(\textbf{(A2)}\), \(\textbf{(G1)}\) and \(\textbf{(G2)}\) hold with \(\lambda(t)=0\) for all \(t\in[0,T]\).    For any \((t,y,z)\in[0,T]\times\mathbb{R}\times\mathbb{R}^d\), we have
	\begin{align*}
		f^1(t,y,z)\leq f^2(t,y,z) \q\  \hbox{and} \q\ \xi^1\leq \xi^2, \q\ \hbox{a.s.}
	\end{align*}
	Moreover, assume that $f^2$ satisfies \(\textbf{(F3')}\).
	Then for any \(t\in [0,T]\), $Y_t^1\leq Y_t^2,  \dbP $-a.s., where $(Y^1,Z^1)\in S^\infty_{\mathbb{F}}([0,T];\mathbb{R})\times L^2_{\mathbb{F}}([0,T];\mathbb{R}^d)$ is a bounded solution of  BDSDE \eqref{1} with parameters $(\xi^1,f^1,g)$,
	$(Y^2,Z^2)\in S^\infty_{\mathbb{F}}([0,T];\mathbb{R})\times L^2_{\mathbb{F}}([0,T];\mathbb{R}^d)$ is a maximal solution of the BDSDE \eqref{1} with parameters $(\xi^2,f^2,g)$.
\end{corollary}

	\subsection{Reflected BDSDE: existence of maximal solutions}\label{sec4.4}
	
	In the final part, we prove the existence of maximal solutions to quadratic BDSDEs with reflection. Unlike the result obtained by Bahlali et al. \cite[Theorem 3.3]{09},  we consider generators with quadratic growth in \(z \). Consequently, 
	the standard monotone approximation argument used in the
	linear growth setting is no longer directly applicable.
	Therefore, the truncation technique and a comparison theorem play a crucial role in our proof.
	
	\begin{theorem}\label{t4.11} \sl
		Assume that the
		conditions \(\textbf{(A2)}\), \(\textbf{(F2')}\),  \(\textbf{(G1)}\),
		\(\textbf{(S2)}\)  and \(\textbf{(G2)}\) hold with \(\lambda(t)=0\) for all \(t\in [0,T]\). 
			Let \(\hat{M}\triangleq \|\xi\|_{\infty}\vee \|S\|_{\infty} \). Then reflected BDSDE \eqref{r1} has a maximal bounded solution $(Y,Z,K)\in S^{\infty}_{\mathbb{F}}([0,T];\mathbb{R})\times L^{2}_{\mathbb{F}}([0,T];\mathbb{R}^d)\times A^{2}_{\mathbb{F}}([0,T];\mathbb{R}_+)$, and for any $t\in[0,T]$, $\hat{U}_0\leq \hat{U}_t\leq Y_t\leq \hat{V}_t\leq \hat{V}_0$, $ \dbP $-a.s., where
		$(\hat{U},\hat{V})\in S_\mathbb{F}^{2}([0,T];\mathbb{R})\times S_\mathbb{F}^2([0,T];\mathbb{R})$ are the unique solutions of the following pair of ODEs:
		\begin{align}
			\hat{U}_t=-\hat{M}-\int_t^T C(1-\hat{U}_s)ds,\q\ 0\leq t\leq T, \label{4.1u}
		\end{align}
		and 	\begin{align}
			\hat{V}_t=\hat{M} +\int_t^T C(1+\hat{V}_s)ds,\q\ 0\leq t\leq T.\label{4.2v}
		\end{align}
	\end{theorem}
	
	\begin{theorem}\label{t4.12} \sl
		Assume that the
		conditions \(\textbf{(A2)}\), \(\textbf{(F3')}\),  \(\textbf{(G1)}\),
		\(\textbf{(S2)}\) and   \(\textbf{(G2)}\) hold with \(\lambda(t)=0\)  for all \(t\in [0,T]\).  Then  reflected BDSDE \eqref{r1} has a maximal bounded solution $(Y,Z,K)\in S^{\infty}_{\mathbb{F}}([0,T];\mathbb{R})\times L^{2}_{\mathbb{F}}([0,T];\mathbb{R}^d)\times A^{2}_{\mathbb{F}}([0,T];\mathbb{R}_+)$.
	\end{theorem}
	
		\begin{remark}\label{rm_quaref}\sl 
		In the special case where \(g\) is positively homogeneous
		in \(z\), the two-sided quadratic case can be reduced to the one-sided quadratic case by an exponential change of variables. Moreover,  since the transformation is increasing, maximality is preserved under the inverse transformation.
		 Therefore, under the additional assumption that \(g\) is positively homogeneous in \(z\),
		 \autoref{t4.11} and \autoref{t4.12} 
		hold under the condition \(\textbf{(F2)}\) and \(\textbf{(F3)}\), respectively.
	\end{remark}

	We first establish the following comparison result, which extends
	\autoref{l4.6} to the reflected setting.

	\begin{lemma}[Comparison]\label{l4.13}\sl
		Let the conditions \(\textbf{(A1)}\), \(\textbf{(G1)}\) and \(\textbf{(S2)}\) be satisfied. 	For any $(t,\omega,y,z)\in [0,T]\times\Omega\times\mathbb{R}\times\mathbb{R}^d$, we have
		\begin{align*}
			S^1_t\leq S^2_t,\q\
			\xi^1\leq \xi^2, \q\ f^1(t,y,z)\leq f^2(t,y,z) \q\ \hbox{a.s.}
		\end{align*}
		Moreover, assume that,
		$f^2(t,y,z)$ is continuous and is of linear growth in \(y\) and \(z\), i.e., there exists a finite constant $C>0$ such that
		\begin{align*}
			|f^2(t,y,z)|\leq C(1+|y|+|z|),\q\ \hbox{a.s.}
		\end{align*}
		Then for any \(t\in [0,T]\), $Y_t^1\leq Y_t^2,  \dbP $-a.s., where $(Y^1,Z^1,K^1)\in S_\mathbb{F}^{2}([0,T];\mathbb{R})\times L_\mathbb{F}^2([0,T];\mathbb{R}^d)\times A_\mathbb{F}^{2}([0,T];\mathbb{R}_+)$ is a solution of reflected BDSDE \eqref{r1} with parameters $(\xi^1,f^1,g)$ and obstacle $S^1$, $(Y^2,Z^2,K^2)\in S_\mathbb{F}^{2}([0,T];\mathbb{R})\times L_\mathbb{F}^2([0,T];\mathbb{R}^d)\times A_\mathbb{F}^{2}([0,T];\mathbb{R}_+)$ is the maximal  solution of reflected BDSDE \eqref{r1} with parameters $(\xi^2,f^2,g)$ and obstacle $S^2$.
	\end{lemma}
	\begin{proof}
		For each \(n\in \mathbb{N}\) with $n\geq C$, let $(f^{2,n})$ be a sequence of Lipschitz functions that approximate $f^2$ from above (as discussed in \autoref{l4.6}).  There
		exists a unique  solution $(Y^{2,n},Z^{2,n},K^{2,n})\in S_\mathbb{F}^{2}([0,T];\mathbb{R})\times L_\mathbb{F}^2([0,T];\mathbb{R}^d)\times A_\mathbb{F}^{2}([0,T];\mathbb{R}_+)$ to reflected BDSDE \eqref{r1} with parameters $(\xi^2,f^{2,n},g)$ and obstacle $S^2$.
		By   the comparison theorem for  reflected BDSDE in the Lipschitz case (refer to Bahlali et al. \cite[Theorem 3.2]{09}),  we have $Y^1\leq Y^{2,n}$, for all
		$n\geq C$. Furthermore, by the existence theorem for  
		maximal solutions in the linear-growth case (see Bahlali et al. \cite[Theorem 3.3]{09}),
		we note that $Y^{2,n}\downarrow Y^2$, where $Y^2$ is the maximal  solution of reflected BDSDE \eqref{r1} with parameters $(\xi^2,f^2,g)$ and obstacle $S^2$. Therefore,
		we conclude that $Y_t^1\leq Y_t^2$, \(\forall t\in [0,T]\), $ \dbP$-a.s.
	\end{proof}

	We will now provide the proofs of   the main results in this part. Since the proofs are similar to those of \autoref{t4.2} and \autoref{t4.3}, we will only highlight the differences.

	\begin{proof}[Proof of \autoref{t4.11}]
		For each \(n\in\mathbb{N}\), $\kappa_n$, $f_n$ and
		$h_n$ are defined as in \autoref{t4.2}. 
			Let
		$\hat{A}=C(1+\|\hat{U}_0\|_{\infty}+\|\hat{V}_0\|_{\infty})$,  enlarging \(C\)
		if necessary, we may assume \(C\geq 1\).
		From the result of Bahlali et al.  \cite[Theorem 3.3]{09}, there exists a maximal solution $(Y^n,Z^n,K^n)\in S_\mathbb{F}^{2}([0,T];\mathbb{R})\times L_\mathbb{F}^2([0,T];\mathbb{R}^d)\times A_\mathbb{F}^{2}([0,T];\mathbb{R}_+)$ which solves the following reflected BDSDE:
		\begin{align*}
			Y_t^n=\xi+\int_t^T f_n(s,Y_s^n,Z_s^n)ds +\int_t^T g(s,Y_s^n,Z_s^n)dB_s+K^n_T-K^n_t-\int_t^TZ_s^ndW_s,\q\ 0\leq t\leq T,
		\end{align*}
		with $\int_0^T\left(Y_t^n-S_t\right)dK_t^n=0.$
		On one hand,  by the choice of \(\hat{M}\), one has \(S_t\leq \hat{V}_t\) for all \(t\in [0,T]\). Therefore
		 \((\hat{V},0,0)\) is a solution of the following  reflected BDSDE with obstacle \(S\):
		\begin{align*}
			\hat{V}_t= \hat{M}+\int_t^Th_n(\hat{V}_s,Z_s)ds+\int_t^Tg(s,\hat{V}_s,Z_s)dB_s+K_T-K_t-\int_t^TZ_sdW_s,\q\ 0\leq t\leq T,
		\end{align*}
		where, in this comparison equation,  \(Z\equiv 0\) and \(K\equiv 0\).
		Since \(\lambda(t)=0\) for all \(t\in [0,T]\), condition \textbf{(G2)}
implies that \(g(t,y,0)=0\). Hence this reflected BDSDE reduces exactly
to the ODE \eqref{4.2v}.
		On the other hand,
		\((\hat{U},0)\) is used  as a lower comparison solution
		for   the following non-reflected BDSDE:
		\begin{align*}
			\hat{U}_t=-\hat{M} -\int_t^Th_n(\hat{U}_s,Z_s)ds+\int_t^Tg(s,\hat{U}_s,Z_s)dB_s-\int_t^TZ_sdW_s,\q\ 0\leq t\leq T,
		\end{align*}
		where \(Z\equiv 0\). Hence this non-reflected BDSDE reduces exactly to the ODE \eqref{4.1u}.
		The comparison with the lower non-reflected equation is valid since  an additional increasing process \(K^n\) in   reflected BDSDE has the favorable sign in It\^o's formula for \(\big( \hat{U}_t-Y_t^n\big)^+\).
		 By \autoref{l4.13}, for any \(t\in[0,T]\), we have 
		 \begin{align*}
		 	\hat{U}_t\leq Y_t^n\leq \hat{V}_t,\q\ \hbox{a.s.}
		 \end{align*}
		For the decreasing bounded sequence $(Y^n)$, there exists a process $Y$ such that $Y^n\downarrow Y$ in $L^2_{\mathbb{F}}([0,T];\mathbb{R})$ and for any \(t\in [0,T]\),  $\hat{U}_0\leq \hat{U}_t\leq Y_t\leq \hat{V}_t\leq \hat{V}_0$.
		We next prove the strong convergence of
	 $(Z^n)$. After that,  	\autoref{p4.9} will identify the limiting reflected solution, and maximality will be verified at the end. \\
	 We divide the following proof into two steps.

		\ms
		
		{\bf\it Step 1}. The sequence \((Z^n)\) is bounded in $L^2_{\mathbb{F}}([0,T];\mathbb{R}^d)$.

		Applying It\^o's formula to $\phi(Y_t^n)$, where \(\phi : \mathbb{R} \longmapsto \mathbb{R}\) is defined by $\phi(x)=e^{-\frac{4C}{1-\alpha}x}$. Since \(\phi' \leq 0\),   taking expectations gives
		\begin{align*}
			\mathbb{E}\left[\phi(Y_0^n)\right]+\frac{1}{2}\mathbb{E}\int_0^T\phi''(Y_s^n)|Z_s^n|^2ds
			\leq &\ \mathbb{E}\left[\phi(\xi)\right]+\mathbb{E}\int_0^T\phi'(Y_s^n)f_n(s,Y_s^n,Z_s^n)ds\\
			& 	+\frac{1}{2}\mathbb{E}\int_0^T\phi''(Y_s^n)\big|g(s,Y_s^n,Z_s^n)\big|^2ds,
		\end{align*}
		where
		\begin{align*}
			f_n(t,y,z)\geq -\hat{A}-C|z|^2\q\
			\hbox{and} \q\  |g(t,y,z)|^2\leq \alpha |z|^2,
		\end{align*}
		with  $\hat{A}=C(1+\|\hat{U}_0\|_{\infty}+\|\hat{V}_0\|_{\infty})$.  By a straightforward calculation:
		\begin{align*}
			\frac{1-\alpha}{2}\phi''+C\phi'=\frac{4C^2}{1-\alpha}\phi,
		\end{align*}
		we get
		\begin{align*}
			\frac{4C^2}{1-\alpha}\mathbb{E}\int_0^T\phi(Y_s^n)|Z_s^n|^2ds
			\leq \mathbb{E}\left[\phi(\xi)\right]+\frac{4C}{1-\alpha}e^{\frac{4C\hat{A}}{1-\alpha} }\hat{A}T.
		\end{align*}
		Therefore, there exists a constant $\hat{K}>0$ which is   independent of $n$, such that $\|Z^n\|_{L^2_{\mathbb{F}}}\leq \hat{K}<\infty$. 
		By the weak compactness of bounded sets in 
		\(L^2_{\mathbb{F}}([0,T];\mathbb{R}^d)\), there
		exists a subsequence, 
		still denoted by \((Z^n)\), and a process
		$Z\in L^2_{\mathbb{F}}([0,T];\mathbb{R}^d)$ such that \(Z^n\longrightarrow Z\) weakly in 
	 $L^2_{\mathbb{F}}([0,T];\mathbb{R}^d)$.
	 
		\ms
		
		{\bf\it Step 2}. The sequence $(Z^n)$ strongly converges to $Z$ in $L^2_{\mathbb{F}}([0,T];\mathbb{R}^d)$.
		
		For each \(n,m \in \mathbb{N}\) with $n<m$, apply It\^o's formula to $\Phi(Y_t^n-Y_t^m)$, where \(\Phi : \mathbb{R} \longmapsto \mathbb{R}\) is defined by $\Phi(x)=\frac{1-\alpha}{200C^2}\left(e^{\frac{20C}{1-\alpha}x}-\frac{20C}{1-\alpha}x-1\right)$. 
		Note that \(\Phi'\geq 0\), \(dK^n\) is supported on the set 
		\(\{Y^n=S\}\) and \(dK^m\) is supported on the set \(\{Y^m=S\}\),
		 we get
		\begin{align*}
			\mathbb{E}\int_0^T\Phi'(Y_s^n-Y_s^m)d\left(K_s^n-K_s^m\right)\leq 0.
		\end{align*}
		Thus taking expectations on both sides, we derive 
		\begin{align*}
			\mathbb{E}\left[\Phi(Y^n_0-Y^m_0)\right] 
			\leq&\ \mathbb{E}\int_0^T\Phi'(Y_s^n-Y_s^m)\Big(f_n(s,Y_s^n,Z_s^n)-f_m(s,Y_s^m,Z_s^m)\Big)ds\\
			&+\frac{1}{2}\mathbb{E}\int_0^T\Phi''(Y_s^n-Y_s^m)\left(\big|g(s,Y_s^n,Z_s^n)-g(s,Y_s^m,Z_s^m)\big|^2-\big|Z_s^n-Z_s^m\big|^2\right)ds.
		\end{align*}
		The only additional term compared with the non-reflected case 
		is the \(dK^n-dK^m\) term, which is non-positive by the Skorokhod
		condition as shown above.
			Repeating the estimates from the proof of
		\autoref{t4.2}, and using the uniform boundedness of \((Y^n)\),   the uniform integrability of \(|Z^n|^2\) in \(L^1_{\mathbb{F}}([0,T];\mathbb{R}^d)\), and the weighted lower semicontinuity argument,
	we obtain the same strong convergence result:
		\begin{align*}
			\lim\limits_{n\rightarrow\infty}\mathbb{E}\int_0^T\big|Z_s^n-Z_s\big|^2ds=0.
		\end{align*}
			Then, in view of the following equality,
		\begin{align*}
			K_T^n=Y_0^n-\xi-\int_0^Tf_n(s,Y_s^n,Z_s^n)ds-\int_{0}^Tg(s,Y_s^n,Z_s^n)dB_s+\int_{0}^TZ_s^ndW_s,
		\end{align*}
		where
		\begin{align*}
			|f_n(s,Y_s^n,Z_s^n)|\leq 2\hat{A}+2C |Z_s^n|^2.
		\end{align*}
		The uniform boundedness of \((Y^n)\), the preceding estimate,
		 and the identity for \((K_T^n)\) imply that  \(\sup\limits_{n}\mathbb{E}\left[ (K_T^n)^2\right]<\infty\).
		Moreover, \(f_n\longrightarrow \widetilde{f}\) locally uniformly,  the sequence \((Y^n)\) is non-increasing, and the generators \((f_n)\) satisfy a uniform quadratic growth condition.  Therefore, all the assumptions  of
		\autoref{p4.9}
		 are satisfied. There exists  a triple 
		 $(Y,Z,K)\in S_\mathbb{F}^{\infty}([0,T];\mathbb{R})\times L_\mathbb{F}^2([0,T];\mathbb{R}^d)\times A_\mathbb{F}^{2}([0,T];\mathbb{R}_+)$ such that 
		\((Y^n,K^n)\longrightarrow (Y,K) \) uniformly in probability on  \( [0,T]\), and \(Z^n \longrightarrow Z\) in \( L_\mathbb{F}^2([0,T];\mathbb{R}^d)\).
		The limiting triple solves the reflected BDSDE \eqref{r1} with generator \(\widetilde{f}\).
		Since  \( \forall t\in [0,T]\), \( \hat{U}_t\leq Y_t\leq \hat{V}_t\),   
	 and \(\hat{A}\) was chosen so that  \(|Y_t|\leq \hat{A}\),
	 we have \(\rho^{\hat{A}}(Y_t)=Y_t\) and  \(\widetilde{f}(t,Y_t,Z_t)=f(t,Y_t,Z_t)\).
	 Hence \((Y,Z,K)\) solves the original reflected BDSDE \eqref{r1}.
	 By \autoref{p4.9} and its pathwise Stieltjes convergence
	 argument, \(Y\geq S\) and \(\int_0^T (Y_s-S_s)dK_s=0\).
Let \((Y',Z',K')\) be any other bounded solution of the original reflected BDSDE \eqref{r1}. Since \(f\leq f_n\), \autoref{l4.13} yields \(Y_t'\leq Y_t^n\), 
\(0\leq t\leq T\). Letting \(n\longrightarrow \infty\), we obtain 
\(Y'_t\leq Y_t\), for any \(t\in [0,T]\). Hence  \((Y,Z,K)\) is the maximal solution.
	\end{proof}

	\begin{proof}[Proof of \autoref{t4.12}]
		In view of the condition \(\textbf{(S2)}\), let $b\triangleq \| S \|_{S^{\infty}_{\mathbb{F}}([0,T];\mathbb{R})}<\infty,$ which is a deterministic constant. We notice that
		$(Y,Z,K)\in S_\mathbb{F}^{\infty}([0,T];\mathbb{R})\times L_\mathbb{F}^2([0,T];\mathbb{R}^d)\times A_\mathbb{F}^{2}([0,T];\mathbb{R}_+)$ is a solution of  reflected BDSDE \eqref{r1} with parameters $(\xi,f,g)$ and obstacle $S$ if and only if 
		$(Y^b,Z^b,K^b)\in S_\mathbb{F}^{\infty}([0,T];\mathbb{R})\times L_\mathbb{F}^2([0,T];\mathbb{R}^d)\times A_\mathbb{F}^{2}([0,T];\mathbb{R}_+)$ is a solution of  reflected BDSDE \eqref{r1} with parameters $(\xi^b,f^b,g^b)$ and obstacle $S^b$, where
		\begin{equation*}
			\left\{
			\begin{aligned}
				&	(Y^b,Z^b,K^b)=(Y-b,Z,K),\\
				&	(\xi^b,f^b,g^b,S^b)=(\xi-b,f(\cdot,y+b,z),g(\cdot,y+b,z),S-b).
			\end{aligned} \right.
		\end{equation*}
		It is straightforward to verify that the triple $(\xi^b,f^b,g^b)$ satisfies the conditions \(\textbf{(A2)}\), \(\textbf{(F3')}\) with \(\varphi_b(x)=
		\varphi(b+x)\) and \(\varphi_b(0)=\varphi(b)\),
		 \(\textbf{(G1)}\) and \(\textbf{(G2)}\) with \(\lambda(t)=0\) for all 
		\(t\in [0,T]\).
		Therefore, without loss of generality, we may assume that
		the obstacle process is bounded and non-positive.
		Using techniques similar to those in the
		proof of  \autoref{t4.3}, we define $\rho^{\bar{C}}$ and $f^{\bar{C}}$ as shown in equations \eqref{4.12} and \eqref{4.13}, respectively. It follows from \autoref{t4.11} that there exists a maximal solution $(Y^{\bar{C}},Z^{\bar{C}},K^{\bar{C}})\in S_\mathbb{F}^{\infty}([0,T];\mathbb{R})\times L_\mathbb{F}^2([0,T];\mathbb{R}^d)\times A_\mathbb{F}^{2}([0,T];\mathbb{R}_+)$ which solves the following  reflected
		BDSDE with parameters $(\xi,f^{\bar{C}},g)$ and obstacle $S$, such that 
		\begin{equation}\label{4.20}
			\left\{\begin{aligned}
				&	Y_t^{\bar{C}}=\xi+\int_t^Tf^{\bar{C}}(s,Y_s^{\bar{C}},Z_s^{\bar{C}})ds+\int_t^Tg(s,Y_s^{\bar{C}},Z_s^{\bar{C}})dB_s
				+K_T^{\bar{C}}-K_t^{\bar{C}}	-\int_t^TZ_s^{\bar{C}}dW_s, \q\ 0\leq t\leq T,\\
				&Y_t^{\bar{C}}\geq S_t\q\ \hbox{and} \q\ \int_{0}^{T} (Y_s^{\bar{C}}-S_s)dK_s^{\bar{C}}=0.
			\end{aligned}\right.
		\end{equation}
		For even $n\geq 2$, choose $a=n(\varphi_b(0)+\mu)$ and $n-1\geq \frac{4C\bar{C}}{1-\alpha}$. Applying It\^o's formula to 
		$\Phi(e^{at}(Y_t^{\bar{C}})^n-X_t^n)$, where $X_t^n$ is defined in \eqref{4.15} and \(\Phi\) has the same property as in \autoref{t4.3}, we have
		\begin{align*}
			\Phi(e^{at}(Y_t^{\bar{C}})^n-X_t^n)	=&\
			\Phi(0)+\int_t^T\Phi'(e^{as}(Y_s^{\bar{C}})^n-X_s^n)
			\Big(-ae^{as}(Y_s^{\bar{C}})^n+ne^{as}(Y_s^{\bar{C}})^{n-1}f^{\bar{C}}(s,Y_s^{\bar{C}},Z_s^{\bar{C}})\\
			&-ne^{as}\varphi_b(0)
			+\frac{n(n-1)}{2}e^{as}(Y_s^{\bar{C}})^{n-2}\big(\big|g(s,Y_s^{\bar{C}},Z_s^{\bar{C}})\big|^2-|Z_s^{\bar{C}}|^2\big)\Big)ds\\
			&+\int_t^T \Phi'(e^{as}(Y_s^{\bar{C}})^n-X_s^n) ne^{as}(Y_s^{\bar{C}})^{n-1}g(s,Y_s^{\bar{C}},Z_s^{\bar{C}})dB_s\\
			&-\int_t^T \Phi'(e^{as}(Y_s^{\bar{C}})^n-X_s^n) ne^{as}(Y_s^{\bar{C}})^{n-1}Z_s^{\bar{C}}dW_s\\
			&+\int_t^T \Phi'(e^{as}(Y_s^{\bar{C}})^n-X_s^n) ne^{as}(Y_s^{\bar{C}})^{n-1}dK_s^{\bar{C}}\\
			&+\frac{1}{2}\int_t^T \Phi''(e^{as}(Y_s^{\bar{C}})^n-X_s^n) \Big( n^2e^{2as}(Y_s^{\bar{C}})^{2n-2}\big(\big|g(s,Y_s^{\bar{C}},Z_s^{\bar{C}})\big|^2-|Z_s^{\bar{C}}|^2\big)\Big)ds.
		\end{align*}
		Since \(n-1\) is odd,  the measure \(dK^{\bar{C}}\) is supported on the set 
		\(\{Y^{\bar{C}}=S\}\), and 
		  $S$ is a non-positive bounded obstacle process,
		we deduce that
		\begin{align*}
			\int_t^T \Phi'(e^{as}(Y_s^{\bar{C}})^n-X_s^n) ne^{as}(Y_s^{\bar{C}})^{n-1}dK_s^{\bar{C}}
			\leq 0.
		\end{align*}
		Taking expectations on both sides, we get
		\begin{align*}
	\mathbb{E}\left[\Phi(e^{at}(Y_t^{\bar{C}})^n-X_t^n)	\right]\leq 0.	 
		\end{align*}
		Since \(\Phi \geq 0\), we obtain
		\begin{align*}
		 e^{at}(Y_t^{\bar{C}})^n\leq X_t^n,\q\ a.s.
		\end{align*}
		Then
		choosing \(\bar{C}\) such that
		\begin{align*}
			\bar{C}\geq 	\left(e^{\left(\varphi_b(0)+\mu\right)T}\vee 1\right)\left(\|\xi\|_{\infty}+1\right),
		\end{align*}
the solution of the truncated equation is also a solution of the original equation.
		Applying the aforementioned a priori estimate to any bounded solution \((Y' ,Z', K' )\),  we obtain
		\(|Y' | \leq \bar{C}\), which implies that  \((Y' , Z' , K')\)  also satisfies
	   the truncated reflected BDSDE \eqref{4.20} with obstacle \(S\).  
		 By the maximality of 
		$(Y^{\bar{C}},Z^{\bar{C}},K^{\bar{C}})$ for the truncated reflected BDSDE, we have \(Y' \leq Y^{\bar{C}}\), 
	and then conclude 
		that $(Y^{\bar{C}},Z^{\bar{C}},K^{\bar{C}})\in S_\mathbb{F}^{\infty}([0,T];\mathbb{R})\times L_\mathbb{F}^2([0,T];\mathbb{R}^d)\times A_\mathbb{F}^{2}([0,T];\mathbb{R}_+)$ is a maximal solution of reflected BDSDE \eqref{r1}.
	\end{proof}

	As a consequence of the method used in the proof of \autoref{t4.12}, we can derive the following comparison result.

	\begin{corollary}[Comparison]\label{c4.17}\sl
		Assume that
		the conditions \(\textbf{(A2)}\), \(\textbf{(G1)}\), \(\textbf{(S2)}\) and \(\textbf{(G2)}\) hold with \(\lambda(t)=0\) for all \(t\in [0,T]\).   For any \((t,y,z)\in[0,T]\times\mathbb{R}\times\mathbb{R}^d\), we have
		\begin{align*}
			S^1_t\leq S^2_t,\q\ 	\xi^1\leq \xi^2,\q\ 
			f^1(t,y,z)\leq f^2(t,y,z), \q\  \hbox{a.s.}
		\end{align*}
		Moreover, assume that  $f^2$ satisfies \(\textbf{(F3')}\).
		Then for any \(t\in [0,T]\), $Y_t^1\leq Y_t^2,  \dbP $-a.s., where $(Y^1,Z^1,K^1)\in S_\mathbb{F}^{\infty}([0,T];\mathbb{R})\times L_\mathbb{F}^2([0,T];\mathbb{R}^d)\times A_\mathbb{F}^{2}([0,T];\mathbb{R}_+)$ is a bounded solution of reflected BDSDE \eqref{r1} with parameters $(\xi^1,f^1,g)$ and obstacle $S^1$,
		$(Y^2,Z^2,K^2)\in S_\mathbb{F}^{\infty}([0,T];\mathbb{R})\times L_\mathbb{F}^2([0,T];\mathbb{R}^d)\times A_\mathbb{F}^{2}([0,T];\mathbb{R}_+)$ is a maximal solution of reflected BDSDE \eqref{r1} with parameters $(\xi^2,f^2,g)$ and obstacle $S^2$.
	\end{corollary}

	\section{Conclusion}\label{conclusion}
	
	In this paper, we have systematically studied the well-posedness of RBDSDEs under various weak conditions. Our work extends the existing theory of RBDSDEs beyond the classical Lipschitz framework and partially addresses one of the open problems in this field.  
	In particular, under the precise general growth and quadratic growth assumptions imposed in this paper,
	several of the existence and maximality results appear
	to be new even for the corresponding non-reflected BDSDEs.
	Meanwhile, several new techniques have been introduced, and the main novelty lies in  overcoming the lack of filtration structure induced by the backward It\^o integral.
	These results not only fill theoretical gaps in the literature but also pave the way for further applications of RBDSDEs in stochastic optimal control, SPDEs, and mathematical finance.
	Several related questions remain open and are left for future work. In particular, it would be interesting to relax the
	boundedness assumptions on the
	terminal value and obstacle in the quadratic growth setting to unbounded setting, such as suitable exponential integrability. The extension to multi-dimensional quadratic BDSDEs with reflection also constitutes a natural direction for subsequent work.

	\section*{Acknowledgments}\label{section7}
	The authors would like to thank Professor Shige Peng from Shandong University for his insightful comments during his  visit to SUSTech.
	%	associate editor and the anonymous referees for their insightful comments that improve the
	%	quality of this paper.
	%	

	%------------------------------------------------------------------------------------------------


\begin{thebibliography}{99}
		\addtolength{\itemsep}{-1.0ex}
		
		
		
		%	\bibitem{ao12}
		%	\rm	Aman, A. and Owo, J.M.,
		%	\it Reflected backward doubly stochastic differential equations with discontinuous generator, 
		%	\sl  Random Oper. Stoch. Equ. ,
		%	\rm 20  (2012), 119--134.
		%	
		
		\bibitem{be20}
		\rm Berrhazi, B.,   El Fatini, M.,  Hilbert, A.,  Mrhardy, N., and  Pettersson, R.,
		\it Reflected backward doubly stochastic differential equations with discontinuous barrier,
		\sl Stochastics,
		\rm 92 (2020), 1100--1124.
		
		\bibitem{06h}
		\rm Briand, P. and Hu, Y., 
		\it  BSDE with quadratic growth and unbounded terminal value, 
		\sl Probab. Theory Related Fields,
		\rm 136  (2006), 604--618.
		
		\bibitem{08h}
		\rm Briand, P. and Hu, Y.,
		\it Quadratic BSDEs with convex generators and unbounded terminal conditions,
		\sl Probab. Theory Related Fields,
		\rm 141 (2008), 543--567.
		
		\bibitem{09}
		\rm Bahlali, K., Hassani, M., Mansouri, B. and Mrhardy, N.,
		\it One barrier reflected backward doubly stochastic differential equations with continuous generator,
		\sl C. R. Math. Acad. Sci. Paris,
		\rm 347  (2009), 1201--1206.
		
		
		
		\bibitem{07o}
		\rm Briand, P., Lepeltier, J. P., and San Mart\'in, J.,
		\it One-dimensional backward stochastic differential equations whose coefficient is monotonic in y and non-Lipschitz in z,
		\sl Electron. J. Probab.,
		\rm 12   (2007), 1504--1531.
		
		
		
		%	\bibitem{01}
		%	\rm Bally, V. and Matoussi, A.,
		%	\it Weak solutions for SPDEs and backward doubly stochastic differential equations,
		%	\sl J. Theoret. Probab.,
		%	\rm 14   (2001), 125--164.
		
		\bibitem{95B}
		\rm	Barles, G. and Murat, F., 
		\it Uniqueness and the maximum principle for quasilinear elliptic equations with quadratic growth conditions, 
		\sl  Arch. Ration. Mech. Anal.,
		\rm 133 (1995), 77--101.
		%	\bibitem{07}
		%	\rm Boufoussi, B., Van Casteren, J. and Mrhardy, N.,
		%	\it Generalized backward doubly stochastic differential equations and SPDEs with nonlinear Neumann boundary conditions,
		%	\sl Bernoulli,
		%	\rm 13 (2) (2007), 423-–446.
		%	
		
		\bibitem{delbaen_backward_2011}
		\rm Delbaen, F., Hu, Y., and Bao, X.,
		\it Backward SDEs with superquadratic growth,
		\sl Probab. Theory Related Fields,
		\rm 150 (2011), 145--192.
		
		
		
		\bibitem{1975}
		\rm Dellacherie, C. and Meyer, P.-A.,
		\it Probabilities and potential, A,
		\sl North--Holland Mathematics Studies, Vol. 29,
		North--Holland Publishing Co.,
		\rm Amsterdam, 1978.
		
		
		\bibitem{92d}
		\rm Da Prato, G. and Zabczyk, J.,
		\it Stochastic equations in infinite dimensions,
		\sl Cambridge University Press,
		\rm Vol.152 (1992), 1--428.
		
		
		%	
		%	\bibitem{08}
		%	\rm Essaky, E. H.,
		%	\it Reflected backward stochastic differential equation with jumps and RCLL obstacle,
		%	\sl Bull. Sci. Math.,
		%	\rm 132  (2008), 690--710.
		
		\bibitem{97r}
		\rm El Karoui, N., Kapoudjian, C., Pardoux, E., Peng, S., and Quenez, M. C.,
		\it Reflected solutions of backward SDE's, and related obstacle problems for PDE's,
		\sl Ann. Probab.,
		\rm 25   (1997), 702--737.
		
		\bibitem{el_karoui_backward_1997}
		\rm El Karoui, N., Peng, S., and Quenez, M.C.,
		\it Backward stochastic differential equations in finance,
		\sl Math. Finance,
		\rm 7 (1997), 1--71.
		
		\bibitem{f23}
		\rm Fan, S.,  Hu, Y., and Tang, S.,
		\it Multi-dimensional backward stochastic differential equations of diagonally quadratic generators: The general result,
		\sl J. Differential Equations,
		\rm 368 (2023), 105--140.
		
		%	\bibitem{f20}
		%	\rm  Fan S., Hu  Y., and  Tang S.,
		%	\it On the uniqueness of solutions to quadratic BSDEs with non-convex generators and unbounded terminal conditions,
		%	\sl C. R. Math. Acad. Sci. Paris,
		%	\rm 358 (2020), 227--235.
		
		\bibitem{hao_mean-field_2025}
		\rm Hao, T., Hu, Y., Tang, S., and Wen, J.,
		\it Mean-field backward stochastic differential equations and nonlocal PDEs with quadratic growth,
		\sl Ann. Appl. Probab.,
		\rm 35 (2025), 1680--1715.
		
		\bibitem{hmz}
		\rm 	 Hu, Y.,  Matoussi, A., and Zhang, T.,
		\it 	Wong--Zakai approximations of backward doubly stochastic differential equations,
		\sl 	Stochastic Process. Appl.,
		\rm 	  125    (2015), 4375--4404.
		
		\bibitem{00h}
		\rm Hu, Y.,
		\it On the solution of forward-backward SDEs with monotone and continuous coefficients,
		\sl Nonlinear Anal.,
		\rm 42   (2000), 1--12.
		
		\bibitem{ho16}
		\rm 	 Hamad\`{e}ne, S. and Ouknine, Y.,
		\it 	Reflected Backward SDEs with General Jumps,
		\sl 	Theory Probab. Appl.,
		\rm 	  60     (2016), 263--280.
		
		\bibitem{10m}
		\rm Han, Y., Peng, S., and Wu, Z.,
		\it Maximum principle for backward doubly stochastic control systems with applications,
		\sl SIAM J. Control Optim.,
		\rm 48   (2010), 4224--4241.
		
		
		
		\bibitem{HuLiangTang2020}
		\rm Hu, Y.,   Liang, G., and  Tang, S.,
		\it Systems of {Ergodic} {BSDEs} {Arising} in {Regime} {Switching} {Forward} {Performance} {Processes},
		\sl SIAM J. Control Optim.,
		\rm 58 (2020), 2503--2534.
		
		
		\bibitem{HuLiangTang2024}
		\rm Hu, Y.,   Liang, G., and  Tang, S.,
		\it Utility maximization in constrained and unbounded financial markets: {Applications} to indifference valuation, regime switching, consumption and {Epstein}-{Zin} recursive utility,
		\rm arXiv:1707.00199v5, 2024.
		
		
		
		\bibitem{ht16}
		\rm  Hu, Y. and  Tang, S.,
		\it Multi-dimensional backward stochastic differential equations of diagonally quadratic generators,
		\sl Stochastic Process. Appl.,
		\rm 126 (2016), 1066--1086.
		
		
		
		\bibitem{wen1}
		\rm Hu, Y., Wen, J., and Xiong, J.,
		\it Backward doubly stochastic differential equations and SPDEs with quadratic growth,
		\sl Stochastic Process. Appl.,
		\rm 175 (2024), 104405.
		
		\bibitem{00}
		\rm Kobylanski, M.,
		\it Backward stochastic differential equations and partial differential equations with quadratic growth,
		\sl Ann. Probab.,
		\rm 28   (2000), 558--602.
		
		%	\bibitem{02}
		%	\rm Kobylanski, M., Lepeltier, J. P., Quenez, M. C., and Torres, S.,
		%	\it Reflected BSDE with superlinear quadratic coefficient,
		%	\sl Probab. Math. Statist.,
		%	\rm 22  (2002), 259--288.
		%	
		
		
		\bibitem{liluo13}
		\rm Li, Z. and Luo, J.,
		\it  Reflected backward doubly stochastic differential equations with discontinuous coefficients,
		\sl Acta Math. Sin. Engl. Ser.,
		\rm 29 (2013),  639--650.
		
		
		\bibitem{98}
		\rm Lepeltier, J. P. and San Mart\'in, J.,
		\it Existence for BSDE with superlinear-quadratic coefficient,
		\sl Stochastics,
		\rm 63   (1998), 227--240.
		
		%	\bibitem{05l}
		%	\rm Lepeltier, J.P., Matoussi, A., and Xu, M.,
		%	\it Reflected backward stochastic differential equations under monotonicity and general increasing growth conditions,
		%	\sl Adv. in Appl. Probab.,
		%	\rm 37   (2005), 134–-159.
		
		\bibitem{97}
		\rm Lepeltier, J. P. and San Mart\'in, J.,
		\it Backward stochastic differential equations with continuous coefficient,
		\sl Statist. Probab. Lett.,
		\rm 32   (1997), 425--430.
		
		
		
		\bibitem{21}
		\rm Li, Y. and Wei, L.,
		\it Reflected Backward Doubly Stochastic Differential Equations with Monotone Coefficients,
		\sl J. Phys.: Conf. Ser.,
		\rm 1865   (2021), 022038.
		
		\bibitem{95}
		\rm Mao, X., 
		\it Adapted solutions of backward stochastic differential equations with non-Lipschitz coefficients, 
		\sl Stochastic Process. Appl., 
		\rm 58   (1995), 281--292.
		
		%	\bibitem{97m}
		%	\rm Matoussi, A.,
		%	\it Reflected solutions of backward stochastic differential equations with continuous coefficient,
		%	\sl Statist. Probab. Lett.,
		%	\rm 34   (1997), 347--354.
		
		%	
		%	\bibitem{matoussi michael}
		%	\rm Matoussi A. and Scheutzow M.,
		%	\it Stochastic PDEs driven by nonlinear noise and backward doubly SDEs,
		%	\sl Journal of Theoretical Probability,
		%	\rm  15  (1)   (2002), 1--39.
		
		%	\bibitem{morlais}
		%	\rm Morlais, M. A.,
		%	\it Quadratic BSDEs driven by a continuous martingale and applications to the utility maximization problem,
		%	\sl Finance Stoch.,
		%	\rm 13   (2009), 121--150.
		
		\bibitem{man16}
		\rm 	Mansouri, B., Ouerdiane, H., and Salhi, I.,
		\it  Backward doubly stochastic differential equations with 	monotone and discontinuous coefficients,
		\sl 	 J. Numer. Math. Stoch.,
		\rm  8  	(2016), 76--87.
		
		%	\bibitem{m16}
		%	\rm Matoussi, A. and Sabbagh, W.,
		%	\it Numerical Computation for Backward Doubly SDEs with random terminal time,
		%	\sl Monte Carlo Methods and Applications,
		%	\rm 22 (3) (2016), 229--258.
		
		%	\bibitem{owo}
		%	\rm 	Owo J M.,
		%	\it  Backward doubly SDEs with continuous and stochastic linear growth coefficients,
		%	\sl  Random Operators and Stochastic Equations,
		%	\rm   26  (3)  (2018), 175--184.
		
		\bibitem{99}
		\rm Pardoux, É.,
		\it BSDEs, weak convergence and homogenization of semilinear PDEs,	
		\sl 	Nonlinear Analysis, Differential Equations and Control (Montreal, QC, 1998), NATO Sci. Ser. C Math. Phys. Sci., 528, Kluwer Acad. Publ., Dordrecht, 
		\rm (1999), 503--549.
		
		
		\bibitem{90}
		\rm Pardoux, É. and Peng, S.,
		\it Adapted solution of a backward stochastic differential equation,
		\sl Syst. Control Lett.,
		\rm 14   (1990), 55--61.
		
		\bibitem{94}
		\rm Pardoux, É. and Peng, S.,
		\it Backward doubly stochastic differential equations and systems of quasilinear SPDEs,
		\sl Probab. Theory Related Fields,
		\rm 98   (1994), 209--227.
		
		\bibitem{peng_open_1999}
		\rm Peng, S.,
		\it Open problems on backward stochastic differential equations,
		\sl Control of Distributed Parameter and Stochastic Systems,
		\rm Kluwer Acad. Publ., Boston, MA, (1999),  265--273.
		
		
		\bibitem{q18}
		\rm Qian, Z. and Xu, M.,
		\it Reflected backward stochastic differential equations with resistance,
		\sl Ann. Appl. Probab.,
		\rm 28   (2018), 888--911.
		
		%	\bibitem{ren}
		%	\rm Ren, Y., Lin A. and Hu L.,
		%	\it Stochastic PDIEs and backward doubly stochastic differential equations driven by Lévy processes,
		%	\sl Journal of Computational and Applied Mathematics,
		%	\rm 223  (2)  (2009),  901--907.
		%	
		
		
		\bibitem{ren10}
		\rm Ren, Y. and Xu, X.,
		\it Reflected backward doubly stochastic differential equations driven by a L\'evy process,
		\sl C. R. Math. Acad. Sci. Paris,
		\rm  348    (2010),  439--444.
		
		\bibitem{05}
		\rm Shi, Y., Gu, Y., and Liu, K.,
		\it Comparison theorems of backward doubly stochastic differential equations and applications,
		\sl Stoch. Anal. Appl.,
		\rm 23   (2005), 97--110.
		
		%	\bibitem{tl}
		%	\rm Tang, S. and Li, X.,
		%	\it Necessary Conditions for Optimal Control of Stochastic Systems with Random Jumps,
		%	\sl SIAM J. Control Optim.,
		%	\rm 32   (1994), 1447--1475.
		
		\bibitem{wy}
		\rm  Wen, J. and Shi, Y.,
		\it Backward doubly stochastic differential equations with random coefficients and quasilinear stochastic PDEs,
		\sl 	J. Math. Anal. Appl.,
		\rm    476  	(2019), 86--100.
		
		
		
		\bibitem{11}
		\rm Wu, Z. and Zhang, F.,
		\it BDSDEs with locally monotone coefficients and Sobolev solutions for SPDEs,
		\sl J. Differential Equations,
		\rm 251   (2011), 759--784.
		
		\bibitem{08x}
		\rm Xu, M.,
		\it Backward stochastic differential equations with reflection and weak assumptions on the coefficients,
		\sl Stochastic Process. Appl.,
		\rm 118   (2008), 968--980.
		
		
		\bibitem{yong_stochastic_1999}
		\rm Yong, J. and Zhou, X.,
		\it Stochastic Controls,
		\sl Springer, New York,
		\rm (1999).
		
		\bibitem{zhang_backward_2017}
		\rm Zhang, J.,
		\it Backward Stochastic Differential Equations,
		\sl Springer, New York,
		\rm (2017).
		
		\bibitem{z10}
		\rm 	Zhang, Q. and Zhao, H.,
		\it 	 Stationary solutions of SPDEs and infinite horizon BDSDEs under non-Lipschitz coefficients,
		\sl 	  J. Differential Equations,
		\rm   248  (2010), 953--991.
	\end{thebibliography}
\end{document}